\documentclass[brochure,12pt]{bourbaki}

\usepackage{todonotes}
\usepackage[T1]{fontenc}
\usepackage{lmodern}
\usepackage{amssymb,bm}
\usepackage{microtype}
\usepackage{stmaryrd}
\newif\ifmathabx
\mathabxfalse
\ifmathabx
\usepackage{mathabx} 
\else
\newcommand{\coloneq}{\mathrel{\mathop:}=}
\DeclareFontFamily{U}{mathb}{\hyphenchar\font45}
\DeclareFontShape{U}{mathb}{m}{n}{
      <5> <6> <7> <8> <9> <10> gen * mathb
      <10.95> mathb10 <12> <14.4> <17.28> <20.74> <24.88> mathb12
      }{}
\DeclareSymbolFont{mathb}{U}{mathb}{m}{n}
\DeclareFontSubstitution{U}{mathb}{m}{n}
\DeclareMathSymbol{\dotdiv}        {2}{mathb}{"01}
\fi
\usepackage{enumerate}
\usepackage{dsfont}
\usepackage[colorlinks=true, pdfstartview=FitV, linkcolor=darkblue,
citecolor=darkred, urlcolor=cyan]{hyperref}
\definecolor{darkred}{RGB}{203,65,84}
\definecolor{darkblue}{RGB}{70,130,180}

\makeatletter
\providecommand\@dotsep{5}
\renewcommand{\listoftodos}[1][\@todonotes@todolistname]{%
	\@starttoc{tdo}{#1}}
\makeatother

\usepackage[french]{babel}

\usepackage[utf8]{inputenc}

\usepackage[
backend=biber,
style=authoryear, 
citestyle=authoryear-comp,
date=year,
maxnames=7
]{biblatex}
\usepackage{csquotes}

\usepackage{bbm,bm}
\usepackage{xcolor,scalerel}

\newcommand{\inv}{^{-1}}
\newcommand{\N}{\mathbb{N}}
\newcommand{\Z}{\mathbb{Z}}
\newcommand{\R}{\mathbb{R}}
\newcommand{\Aut}{\mathrm{Aut}}
\newcommand{\Iso}{\mathrm{Iso}}
\newcommand{\Q}{\mathbb{Q}}
\newcommand{\MAlg}{\mathrm{MAlg}}
\newcommand{\1}{\bm 1}
\newcommand{\0}{\bm 0}
\newcommand{\la}{\left\langle}
\newcommand{\ra}{\right\rangle}
\newcommand{\laBig}{\Bigl\langle}
\newcommand{\raBig}{\Bigr\rangle}
\newcommand{\acts}{\curvearrowright}
\renewcommand{\bigtriangleup}{\triangle}
\newcommand{\Vect}{\mathrm{Vect}}

\input pdf-trans
\newbox\qbox
\def\usecolor#1{\csname\string\color@#1\endcsname\space}
\newcommand\outline[1]{\leavevmode%
	\def\maltext{#1}%
	\setbox\qbox=\hbox{\maltext}%
	\boxgs{Q q 2 Tr \thickness\space w 0 0 0 rg 0 G}{}%
	\copy\qbox%
}
\newcommand\boldify[2][.2]{%
	\def\thickness{#1}%
	\ThisStyle{\outline{$\SavedStyle#2$}}%
}
\def\fplus{\mathbin{\boldify[.6]{+}}}

\newcommand{\abs}[1]{\left\lvert #1\right\rvert}
\newcommand{\absBig}[1]{\Bigl\lvert #1\Bigr\rvert}
\newcommand{\absbig}[1]{\bigl\lvert #1\bigr\rvert}

\newcommand{\norm}[1]{\left\lVert #1\right\rVert}
\newcommand{\normbig}[1]{\bigl\lVert #1\bigr\rVert}
\newcommand{\qf}{\mathrm{qf}}

\DeclareMathOperator{\tp}{\mathrm{tp}}
\DeclareMathOperator{\Th}{\mathrm{Th}}
\DeclareMathOperator{\id}{\mathrm{id}}
\DeclareMathOperator{\dtp}{\partial}

\mathchardef\mhyphen="2D
\newcommand{\lip}{1{\mhyphen}\mathrm{Lip}}

\DeclareDelimFormat{nameyeardelim}{\addcomma\space}

\DeclareNameAlias{sortname}{given-family}

\renewcommand{\bibnamedash}{\leavevmode\raise3pt\hbox to3em{\hrulefill}\space}

\AtEveryBibitem{%
 \clearfield{issn} 
  \clearfield{isbn}

 \ifentrytype{online}{}{
   \clearfield{url}
 }
}

\renewbibmacro{in:}{%
   \ifentrytype{article}{}{\printtext{\bibstring{in}\intitlepunct}}}

\DeclareFieldFormat[article,periodical,inreference]{number}{\mkbibparens{#1}}
\DeclareFieldFormat[article,periodical,inreference]{volume}{\mkbibbold{#1}}
\renewbibmacro*{volume+number+eid}{%
   \setunit*{\addcomma\space}
   \printfield{volume}%
  \setunit*{\adddot}
   \setunit*{\addthinspace}
   \printfield{number}%
   \setunit{\addcomma\space}%
   \printfield{eid}}

\DeclareFieldFormat[article,inbook,incollection]{title}{\enquote{#1}\addcomma} 

\newtheorem*{conv}{Convention}

\date{Mai 2021}
\bbkannee{73\textsuperscript{e} ann\'ee, 2020--2021}  
\bbknumero{1180}                                      

\title{La propriété (T) pour les groupes polonais Roelcke-précompacts}
\subtitle{d'après Ibarlucía,  s'appuyant sur des travaux de Ben Yaacov et Tsankov}

\author{François Le Maître}
\address{
Université de Paris et Sorbonne Université \\
CNRS, IMJ-PRG \\
F-75006 Paris, France\\
}
\email{francois.le-maitre@imj-prg.fr}

\begin{document}

\maketitle

\section*{Introduction}

Née au début des années 2000 dans les travaux de 
\textcite{benyaacovModeltheorymetric2008},
la théorie des modèles métrique, ou théorie des modèles continue,  
permet d'étendre des techniques issues de la théorie des 
modèles classique à de nouveaux objets issus de l'analyse.

Dans ce texte, nous allons nous intéresser à une utilisation récente et remarquable
de la théorie des modèles continue dans le cadre des groupes polonais: la
preuve par \textcite{ibarluciaInfinitedimensionalPolishGroups2021} 
 du fait que 
tout groupe polonais Roelcke-précompact a la propriété~(T) de Kazhdan. 

Cette preuve illustre très bien la capacité de la théorie des modèles à offrir une approche unifiée
à des problèmes (ou plutôt ici, des groupes) a priori très différents. 
Elle s'appuie sur une notion avancée (tout du moins à mes yeux !)
de théorie des modèles continue, la stabilité locale. Elle utilise également le langage des imaginaires métriques 
\parencite{benyaacovRoelckeprecompactPolishGroup2018},
ce qui la rend particulièrement élégante mais difficile à comprendre pour une personne peu familière avec
la théorie des modèles continue.

Il m'a donc semblé préférable
de donner la preuve dans un cas  particulier mais significatif afin que cet exposé reste accessible 
à un lectorat large. Pour la même raison, j'ai évité de recourir aux imaginaires métriques au prix
d'un argument final plus pédestre. 
Le cas particulier est celui du groupe $\Aut([0,1],\lambda)$ des transformations préservant la mesure
de l'intervalle $[0,1]$ muni de la mesure de Lebesgue $\lambda$. Pour ce groupe, on peut énoncer le théorème
d'Ibarlucía comme suit.

\begin{theo}[\cite{ibarluciaInfinitedimensionalPolishGroups2021}
  ]\label{thm:main}
	Il existe deux transformations préservant la mesure $T_1,T_2\in \Aut([0,1],\lambda)$ telles que
	pour toute $\pi: \Aut([0,1],\lambda)\to\mathcal U(\mathcal H)$ représentation
	unitaire continue sur un espace de Hilbert $\mathcal H$, 
	s'il existe un vecteur $\xi\in\mathcal H$ non nul tel que 
	\[
	\norm{\pi(T_1)\xi-\xi}<\sqrt{2-\sqrt 3}\norm{\xi}\text{ et }
	\norm{\pi(T_2)\xi-\xi}<\sqrt{2-\sqrt 3}\norm{\xi},
	\]
	alors il existe en fait un vecteur $\xi'\in\mathcal H$ non nul \emph{invariant}, c'est-à-dire tel que
	pour tout~$T$ dans $\Aut(X,\mu)$, on ait
	$\pi(T)\xi'=\xi'$ .
\end{theo}

Après une première section qui expose quelques exemples de groupes polonais, dont $\Aut([0,1],\lambda)$,
nous présentons la classe des groupes polonais Roelcke-précompacts. On en donne
diverses caractérisations, dues à Ben Yaacov et Tsankov, notamment la description comme groupes
d'automorphismes agissant de manière approximativement oligomorphe. La section~\ref{sec:(T)}
donne un bref aperçu de la propriété~(T) dans le cadre des groupes polonais. 
Le théorème d'Ibarlucía y est énoncé sous sa forme générale et contextualisé. On explicite
les transformations $T_1$ et $T_2$ qui apparaissent dans l'énoncé ci-dessus, qui proviendront d'une
action ``très libre'' préservant la mesure du groupe libre à deux générateurs $\mathbb F_2$. On
y justifie aussi la présence de la constante $\sqrt{2-\sqrt 3}$ dans l'énoncé ci-dessus tout en expliquant
pourquoi ce résultat se ramène à montrer que toute représentation orthogonale 
de $\Aut([0,1],\lambda)$ sans vecteurs
invariants se restreint en une représentation de $\mathbb F_2$ qui est un multiple de sa représentation régulière
(théorème~\ref{thm:mainbis}).
Pour des raisons de simplicité, on se focalisera sur un énoncé plus faible et on montrera simplement
 qu'une telle représentation contient une copie de la représentation régulière de $\mathbb F_2$ (théorème~\ref{THM:MAINTER}).
On indiquera cependant au lectorat familier avec l'indépendance relative dans les espaces de probabilité
comment obtenir le théorème~\ref{thm:mainbis}, et ce à la toute fin de l'exposé.

Les deux
sections suivantes mettent en place les bases de théorie des modèles métrique
nécessaire à la preuve.  Tout groupe polonais peut en effet se voir comme groupe
d'automorphismes d'une \emph{structure métrique séparable}. Cette dernière le
reflète particulièrement bien dans le cadre du théorème de Ryll-Nardzewzki où
le groupe est Roelcke-précompact et la structure $\aleph_0$-catégorique. Cette idée
cruciale est exploitée depuis quelques années 
\parencite{ibarluciaDynamicalHierarchyRoelcke2016,ibarluciaAutomorphismGroupsRandomized2017,benyaacovEberleinOligomorphicGroups2018,benyaacovRoelckeprecompactPolishGroup2018}
 et est apparue pour la première fois
dans les travaux de 
\textcite{benyaacovWeaklyAlmostPeriodic2016}.

Étant donnée une structure métrique, on définit dans la section~\ref{sec:bases} 
une algèbre de fonctions à valeurs réelles sur cette structure
(la $\R$-algèbre de Banach des \emph{prédicats définissables})
qui correspond à \emph{ce que la théorie des modèles continue peut nous dire de cette structure}.
Un des points essentiels est que cette algèbre de fonctions peut en fait être \emph{interprétée}
dans toute structure satisfaisant la même \emph{théorie} que notre structure de départ, 
et il nous faudra commencer par 
développer les notion de \emph{langage métrique} et de \emph{formules} associées afin de
définir proprement ce qu'on entend par là.

Une autre particularité importante de l'algèbre des prédicats définissables 
est qu'elle est stable par
infimums ou supremums sur la structure elle-même. Comme on le verra, cette propriété
est cruciale dans la preuve du théorème principal.
Les \emph{types} sont les caractères (morphismes
à valeur réelle) de 
l'algèbre des prédicats définissables, l'exemple de base d'un type étant donné par l'évaluation des 
prédicats définissables sur un point de la structure. 
On a alors une correspondance de Gelfand qui nous permet 
d'identifier les prédicats définissables aux fonctions continues sur l'espace des types.
Certaines sous-structures particulières
(les sous-structures \emph{élémentaires}) interpréteront les prédicats définissables
de la même manière que la structure ambiante, et en particulier elles verront les mêmes infimums
et supremums que cette dernière lorsque ça a du sens. 

Les \emph{sous-ensembles définissables} d'une structure, étudiés dans la section~\ref{sec: definissable},
sont en quelque sorte les sous-ensembles que la théorie des modèles ``voit''. 
Notamment, lorsque l'on applique un infimum ou un supremum à un prédicat
définissable restreint à 
un sous-ensemble définissable, on obtient un prédicat définissable (cf.\ proposition~\ref{prop:quantif sur def}), ce qui est
crucial pour la preuve du théorème~\ref{thm:main}.
Le théorème de Ryll-Nardzewski implique en effet que le complété à gauche du groupe des automorphismes
de toute structure $\aleph_0$-catégorique s'identifie à un sous-ensemble définissable de la structure,
permettant ainsi de ``transférer'' une action par isométries du groupe sur la
structure (cf.\ \textcite{benyaacovRoelckeprecompactPolishGroup2018}
pour une formulation précise en termes d'imaginaires métriques).
On renvoie à l'appendice pour la preuve du théorème de Ryll-Nardzewski;
on l'établira de manière directe pour $\Aut([0,1],\lambda)$ à la fin de la section~\ref{sec: definissable}, prouvant au passage 
que la théorie de la structure qui lui est associée a l'\emph{élimination des quantificateurs}.

Nous donnerons enfin les éléments de preuve du théorème~\ref{thm:main} dans la section 6. On partira
d'une représentation orthogonale de $G$ sans vecteurs invariants que l'on
l'étendra en une représentation orthogonale de son complété à gauche, 
de sorte à transférer la représentation orthogonale sur la structure via le théorème de Ryll-Nardzewski. 
Il s'agira alors d'exploiter les propriétés d'indépendance de l'action préservant la mesure du groupe libre
afin de voir que la restriction de notre représentation orthogonale au groupe libre est
un multiple de sa représentation régulière. C'est ici que la théorie des modèles jouera son rôle, permettant de 
transférer une propriété d'une sous-structure élémentaire à la structure entière pour mener à une contradiction.

Terminons cette introduction en en mentionnant les différences entre la preuve pour $\Aut([0,1],\lambda)$
et celle du théorème général d'Ibarlucía, qui nécessite un travail supplémentaire 
conséquent.
Dans le cas général, on  part d'une 
action ``très libre'' du groupe libre $\mathbb F_2$ sur une structure $\aleph_0$-catégorique \emph{quelconque}, 
pour laquelle 
la liberté de l'action s'énonce en termes de \emph{relation d'indépendance stable}.
Pour exhiber une telle action, on
utilise notamment une construction combinatoire astucieuse d'\emph{intervalles}
 sur $\mathbb F_2$  ayant 
des propriétés qui généralisent celles des intervalles de $\Z$. 
De plus, il faut se débarrasser de la ``partie compacte'' de $G$ : les 
raisonnements de la section~\ref{sec:pfautxmu} n'auraient aucune chance d'aboutir pour un groupe compact
puisqu'ils reposent sur la richesse du complété à gauche. Enfin, il faut exploiter les propriétés combinatoires
des intervalles de $\mathbb F_2$ afin de comprendre la représentation que l'on obtient.
Je renvoie le lectorat à \textcite{ibarluciaInfinitedimensionalPolishGroups2021},
 notamment la fin de l'introduction
qui décrit clairement la construction si l'on remplace $\mathbb F_2$ par $\Z$, en espérant que cet exposé lui aura auparavant
fourni quelques clés de lecture.

\paragraph*{Remerciements.} Je remercie chaleureusement Tomás Ibarlucía et Todor Tsankov pour les conversations enrichissantes
que nous avons pu avoir autour de la théorie des modèles continue, et que j'espère que nous continuerons d'avoir !
Merci à Todor également pour m'avoir indiqué une preuve du théorème de Ryll-Nardzewski qui ne fasse 
pas appel au théorème de compacité. Enfin, un grand merci à Nicolas Bourbaki, Alessandro Carderi,
Colin Jahel, Adriane Kaïchouh,
 Romain Tessera
et Todor Tsankov pour leur relecture attentive.

\setcounter{tocdepth}{2}
\tableofcontents
\section{Quelques exemples de groupes polonais}

Un \textbf{espace polonais}\footnote{Cette terminologie, suggérée initialement \og{}pour rire\fg{}
	par Godement à Bourbaki en 1949, rend hommage aux nombreuses contributions
	des mathématiciens polonais dans ce domaine, cf.\ note au bas de la page 67 de \textcite{godementAnalyseMathematiqueIV2003}.
} est un espace topologique séparable admettant 
une distance compatible complète. 
Un \textbf{groupe polonais} est un groupe topologique dont la topologie est polonaise.
Les groupes polonais apparaissent naturellement comme groupes de symétries de diverses structures
mathématiques. 
Lorsque la structure est ``de dimension finie''
on obtient assez souvent un groupe localement compact, mais les groupes polonais
dont il sera question dans cet exposé proviendront plutôt de structures de dimension infinie
(mais néammoins séparables), comme par
exemple un espace de Hilbert séparable de dimension infinie.

Voici trois faits très utiles sur les espaces et groupes polonais.
\begin{fait}
	 Tout produit dénombrable d'espaces polonais est un espace polonais.
\end{fait}
\begin{fait}
	Tout fermé d'un espace polonais est un espace polonais.
\end{fait}
\begin{fait}\label{fait:sg fermé}
	Les sous-groupes d'un groupe polonais qui sont polonais pour la topologie
	induite sont exactement ses sous-groupes fermés.
\end{fait}

Le premier fait se montre en exhibant une distance compatible complète sur le produit,
comme celle que l'on considère dans la convention donnée en début de section~\ref{sec: prp}. 
Le deuxième est immédiat une fois qu'on sait qu'un espace métrique est séparable si et
seulement s'il est à base 
dénombrable d'ouverts. Le troisième utilise un renforcement du deuxième qui caractérise les
sous-espaces polonais d'un espace polonais comme en étant les $G_\delta$ 
(intersections dénombrables d'ouverts),
et est une jolie application du théorème de Baire 
(voir la proposition 2.2.1 de \cite{gaoInvariantDescriptiveSet2009}).

\subsection{Groupes d'automorphismes de structures dénombrables}

%
L'intérêt de la définition qui suit réside dans la variété des exemples qu'elle permet de traiter.

\begin{defi}
	Soit $X$ un ensemble non vide.
	Une \textbf{structure discrète} sur $X$ est la donnée d'une famille 
	$(f_i)_{i\in I}$ de \textbf{fonctions} et d'une famille $(R_j)_{j\in J}$ de \textbf{relations} telles que
	\begin{itemize}
		\item pour tout $i\in I$, il existe $n_i\in\N$ tel que $f_i:X^{n_i}\to X$ ($n_i$ est appelé
		l'\textbf{arité} de la fonction $f_i$)
		\item pour tout $j\in J$, il existe $m_j\in\N$ tel que $R_j\subseteq X^{m_j}$ ($m_j$ est 
		appelé l'\textbf{arité} de la relation $R_j$)
	\end{itemize}
\end{defi}

Remarquons que l'on autorise des fonctions d'arité $0$, que l'on voit comme des éléments de $X$,
aussi appelées \textbf{constantes}. Par exemple un groupe $(G,e,\cdot,^{-1})$ est une structure
où $e$ est la constante représentant l'élément neutre, $\cdot$ est le produit de groupe et $^{-1}$ le passage à l'inverse.

Voici des exemples qui
devraient vous convaincre, si besoin est, qu'un grand nombre de structures qui apparaissent en mathématiques 
peuvent être vues comme des structures
discrètes, au prix d'un nombre parfois élevé de fonctions et de relations. 
\begin{itemize}
	\item Un graphe orienté $G=(V,E)$ (sans arête multiple) est une structure discrète sur l'ensemble $V$, 
	avec pour unique relation $E\subseteq V\times V$. Un graphe colorié par deux couleurs de sommets $0$ et $1$
	peut aussi être vu comme une structure discrète en ajoutant deux relation unaires $R_0$ et $R_1$ qui 
	encodent respectivement l'ensemble des sommets de couleur $0$ et ceux de couleur $1$. On peut 
	également
	voir comme des structures discrètes les graphes à arêtes multiples (en ajoutant des relations binaires)
	ou encore les complexes simpliciaux (en ajoutant des relations $n$-aires). 
	\item Un espace vectoriel $E$ sur un corps $\mathbb K$ est une structure discrète : outre le vecteur nul
	qui  est une fonction d'arité $0$ (constante) et l'addition (fonction binaire), 
	on se donne pour chaque $k\in\mathbb K$
	une fonction $m_k:E\to E$ qui représente la multiplication par le scalaire $k$.
\end{itemize}	

Un \textbf{plongement} d'une structure discrète $(X,(f_i)_{i\in I},(R_j)_{j\in J})$ 
est une \emph{injection} $\varphi\colon X\to X$ qui commute aux relations et aux fonctions: pour tout $i\in I$,
tout \mbox{$(x_1,\ldots,x_{n_i})\in X^{n_i}$}, \[
\varphi(f_i(x_1,\ldots,x_{n_i}))=f_i(\varphi(x_1),\ldots,\varphi(x_{n_i}))
\]
et pour tout $j\in J$, tout $(x_1,\ldots,x_{m_j})\in X^{m_j}$, 
\[
(x_1,\ldots,x_{m_j})\in R_j\iff (\varphi(x_1),\dots,\varphi(x_n))\in R_j.
\]
Un \textbf{automorphisme} est un plongement surjectif.

\begin{prop}
	Soit $G$ le groupe des automorphismes d'une structure dénombrable~$X$. Alors
	$G$~est un groupe polonais pour la topologie produit sur~$X^X$, $X$~étant muni
	de la topologie discrète.
\end{prop}
\begin{proof}
	La composition
	définit une application continue $X^X\times X^X\to X^X$, et 
	le passage à l'inverse définit une application continue  $G\to G$ puisque 
	l'ouvert fermé de (pré)-base $\{g\in G: g(x)=y\}$ est envoyé sur 
	l'ouvert-fermé $\{g\in G: g(y)=x\}$ par l'inversion. Ainsi $G$ est bien un groupe topologique. 
	
	Pour voir que $G$ est polonais pour la topologie induite par $X^X$, on commence par remarquer
	que l'espace des plongements $X\to X$ est un fermé, donc polonais.
	L'application qui à $g\in G$
	associe $(g,g^{-1})$ est un homéomorphisme sur son image, et cette image est l'espace
	fermé (donc polonais) des couples de plongements $(f,g)$ tels que
	$f\circ g=g\circ f=\mathrm{id}_X$, donc $G$ est bien polonais.
\end{proof}

\begin{rema}
	La structure uniforme induite par $X^X$ sur $\mathrm{Aut}(X)$ est la structure uniforme gauche sur $G$.
	La même preuve montre en fait que le groupe des automorphismes de toute structure discrète est 
	\emph{Raikov-complet}.
\end{rema}

Étant donnée une structure dénombrable infinie, 
il est fort possible que son groupe d'automorphismes soit dénombrable (donc discret),
voire trivial. 
En général, les structures que nous considérerons posséderont beaucoup d'automorphismes.
Voici un petit échantillon de structures dénombrables dont les groupes d'automorphismes
sont non dénombrables.  
\begin{exem}
	Si l'on ne met aucune structure sur $X$ dénombrable infini, on obtient le groupe $\mathfrak S(X)$
	des permutations de $X$, aussi noté $\mathfrak S_{\infty}$ quand $X=\N$. Remarquons que quand 
	on a une structure discrète sur $X$ dénombrable infini, en identifiant $X$ à $\N$ on ne change pas
	son groupe d'automorphismes.
	Ainsi le groupe des automorphismes de toute structure dénombrable est isomorphe à un sous-groupe
	fermé de $\mathfrak S_{\infty}$. 
\end{exem}

\begin{exem}\label{ex: random graph}
	Considérons le graphe complet sur l'ensemble $\N$ (sans boucles, non orienté). 
	Chaque arête est alors retirée avec probabilité $1/2$, indépendamment. On obtient
	presque sûrement toujours le même graphe à isomorphisme près, appelé le \textbf{graphe aléatoire}.
	C'est  
	le seul graphe dénombrable non vide satisfaisant la propriété~(*) suivante: pour tous ensembles de sommets
	$S_1$ et $S_2$ finis et disjoints, on peut trouver un sommet $s\not\in S_1\cup S_2$ 
	relié à tous les éléments de $S_1$ mais à aucun élément
	de $S_2$. 
	La preuve de ce fait utilise un argument de va-et-vient que l'on donne tout de suite après. 
	Ce graphe est souvent noté $R$ (pour random), son groupe d'automorphismes $\Aut(R)$ est un groupe 
	polonais. On renvoie à \textcite{cameronRandomGraph1997} pour plus d'informations.
\end{exem}
\begin{proof}[Démonstration de l'unicité du graphe aléatoire à isomorphisme près.]
	Soient $G$ et $H$ deux graphes dénombrables non vides satisfaisant la propriété~(*), 
	alors ils sont nécessairement infinis.
	Notons $V$ et $W$ leurs ensembles de sommets
	respectifs.
	Énumérons $V=\{v_n\colon n\in\N\}$ et $W=\{w_n\colon n\in\N\}$.
	La construction par va-et-vient se fait par récurrence sur
	$n\in\N$ en construisant des bijections $\varphi_n: V_n\to W_n$ où 
	\begin{itemize}
		\item $V_n$ est un sous-ensemble fini de $V$ contenant $\{v_0,\dots,v_n\}$;
		\item $W_n$ est un sous-ensemble fini de $W$ contenant $\{w_0,\dots, w_n\}$
		\item $\varphi_n$ est un isomorphisme entre les graphes induits sur $V_n$ et $W_n$;
	\end{itemize}
    On commence par $\varphi_0$ qui envoie $v_0$ sur $w_0$.
    Puis, $\varphi_n$ étant construit, on l'étend en deux étapes (va puis vient !) de la manière suivante, obtenant
    une application $\psi_n$ intermédiaire après la première étape.
    \begin{itemize}
    	\item \emph{Va.}  Si $v_{n+1}$ était dans le domaine de~$\varphi_n$ on pose simplement $\psi_n=\varphi_n$. Sinon,
    	soit~$S_1$ l'ensemble des éléments de~$V_{n}$ reliés à~$v_n$ et $S_2=V_n\setminus S_1$.
    	Alors on demande que $\psi_n$~prolonge~$\varphi_n$ et on choisit~$\psi_n(v_n)$
    	de sorte à ce qu'il soit relié à tous les éléments de $\varphi_n(S_1)$, mais à aucun élément 
    	de $\varphi_n(S_2)$. Par construction $\psi_n$~est toujours un isomorphisme entre les graphes
    	induits sur son image et son domaine.
    	\item \emph{Vient.} Si $w_{n+1}$ est dans l'image de $\psi_n$, on pose $\varphi_{n+1}=\psi_n$. 
    	Sinon, on applique la même construction que précédemment à $\psi_n\inv$ de sorte à obtenir
    	$\psi_{n+1}$ dont le domaine contienne $w_{n+1}$, et on pose $\varphi_{n+1}=\psi_{n+1}\inv$.
    \end{itemize}
	Alors l'application $\varphi$ obtenue comme réunion des $\varphi_n$ est l'isomorphisme de graphes attendu.
\end{proof}
\begin{exem}
	On munit l'ensemble $\Q$ des rationnels de l'ordre induit par les réels, 
	on obtient la structure discrète $(\Q,<)$ dont le groupe polonais d'automorphismes est noté $\Aut(\Q,<)$.
 	L'ensemble ordonné $(\Q,<)$ est à isomorphisme près le seul ordre total
	dénombrable non vide dense (tout intervalle ouvert est non vide) sans maximum ni minimum, 
	ce qui se montre encore par va-et-vient.
	Les sous-groupes dénombrables de $\Aut(\Q,<)$ sont exactement les groupes dénombrables ordonnables. 
	On ne sait pas si ces derniers peuvent avoir la propriété~(T), mais $\Aut(\Q,<)$ tout entier
	a la propriété~(T) \parencite{tsankovUnitaryRepresentationsOligomorphic2012}. Profitons-en pour mentionner la très jolie construction par 
	\textcite{duchesneGroupPropertyActing2020} de groupes polonais unitairement représentables avec la propriété~(T)  plongés de manière non
	élémentaire dans le groupe des homéomorphismes du cercle .
\end{exem}

\begin{exem}
	Soit $\Sigma$ une surface de type infini, son \emph{groupe modulaire étendu} $\mathrm{Map}^\pm(\Sigma)$
	est le quotient du groupe des homéomorphismes de $\Sigma$ par le sous-groupe des homéomorphismes 
	isotopes à l'identité. On peut construire un \emph{graphe des courbes} dont les sommets
	sont les courbes (applications continues $\mathbb S^1\to \Sigma$ à homotopie et renversement
	d'orientation près), reliées par une arête si elles admettent des représentants disjoints.
	Alors $\mathrm{Map}^{\pm}(\Sigma)$ s'identifie au groupe d'automorphismes de ce graphe 
	\parencite{bavardIsomorphismsBigMapping2020,hernandezIsomorphismsCurveGraphs2018}, et est donc
	un groupe d'automorphismes de structure dénombrable. 
	On recommande le survol d'\textcite{aramayonaBigMappingClass2020} pour plus d'informations sur 
	ces groupes.
\end{exem}

Les groupes d'automorphismes de structures dénombrables sont tous des groupes topologiques 
non archimédiens, c'est-à-dire que l'identité admet une base de voisinages formée de sous-groupes
ouverts (donc également fermés). On peut montrer que tout groupe polonais non archimédien 
est isomorphe au groupe d'automorphismes d'une structure dénombrable
(en particulier, c'est le cas de tout groupe polonais localement compact totalement discontinu). 
Nous verrons une construction métrique similaire qui permet d'englober tous
les groupes polonais (cf.\ remarque~\ref{rmq: groupe autor struct metri}).

\subsection{Groupes d'automorphismes de structures métriques séparables}
\label{sec: groupes d'auto metrique}

\begin{defi}
	Une \textbf{structure métrique} est un espace métrique \textbf{complet} $(X,d)$ muni d'une famille 
	$(f_i)_{i\in I}$ de \textbf{fonctions} et d'une famille $(R_j)_{j\in J}$ de \textbf{relations} telles que
	\begin{itemize}
		\item pour tout $i\in I$, il existe $n_i\in\N$ tel que $f_i:X^{n_i}\to X$, et $f_i$ est 
		uniformément continue \footnote{Si on n'est pas familier avec les espaces uniformes, on munit 
			$X^{n_i}$ de la distance donnée dans la convention en début de section~\ref{sec: prp} afin
			de donner du sens à l'uniforme continuité de $f_i$.};
		\item pour tout $j\in J$, il existe $m_j\in\N$ tel que $R_j:X^{m_j}\to\R$,
		 et $R_j$ est uniformément continue.
	\end{itemize}
\end{defi}

Par rapport au cas discret, \emph{la distance $d$ remplace l'égalité} qui était implicitement vue
comme une relation sur $X$ puisqu'on travaille avec des plongements.
Toute structure discrète peut être vue comme une structure métrique en la munissant de la distance
discrète\footnote{La distance discrète est définie par $\delta(x,y)=0$ si $x=y$ et 	
$\delta(x,y)=1$ sinon.}, et en remplaçant chaque relation $R_j$ par la relation métrique $1-\chi_{R_j}$, 
où $\chi_{R_j}$ désigne la fonction
caractéristique de $R_j$.

\begin{rema}
	Nous ne demandons pas pour l'instant que la distance sur $X$ soit bornée, mais un 
	grand nombre de
	résultats de théorie des modèles s'appuient sur cette hypothèse supplémentaire. On
	verra que pour les groupes d'automorphismes Roelcke-précompacts 
	de structures métriques
	on peut toujours supposer $d$ bornée quitte à changer la structure.
\end{rema}

On appelle \textbf{plongement} d'une structure métrique $(X,d,(f_i)_{i\in I},(R_j)_{j\in J})$ toute application
isométrique $X\to X$ telle que pour 
 tout $i\in I$,
tout $(x_1,\ldots,x_{n_i})\in X^{n_i}$, \[
\varphi(f_i(x_1,\ldots,x_{n_i}))=f_i(\varphi(x_1),\ldots,\varphi(x_{n_i}))
\]
et pour tout $j\in J$, tout $(x_1,\ldots,x_{m_j})\in X^{m_j}$, 
\[
R_j(x_1,\ldots,x_{m_j})=R_j(\varphi(x_1),\dots,\varphi(x_{m_j})).
\]
Un \textbf{automorphisme} est un plongement surjectif. 
\begin{prop}
	Soit $G$ le groupe d'automorphismes d'une structure métrique séparable $(X,d)$. Alors
	$G$ est un groupe polonais pour la topologie de la convergence simple,
	c'est-à-dire la topologie produit sur $X^X$, la structure $X$ étant munie
	de la topologie induite par la distance $d$.
\end{prop}
\begin{proof}
	Considérons l'espace $\lip(X,X)$ des applications $1$-lipschitziennes de $X$ vers lui-même, muni
	de la topologie produit. 	
	En utilisant l'inégalité triangulaire et le fait qu'on a des applications $1$-lipschitziennes, on montre que la composition
	définit une application continue $\lip(X,X)\times \lip(X,X)\to \lip(X,X)$.
	Le passage à l'inverse définit également une application continue  $G\to G$ puisque 
	l'ouvert de (pré)-base $\{g\in G: d(g(x),y)<\epsilon\}$ est envoyé sur 
	l'ouvert $\{g\in G: d(g(y),x)<\epsilon\}$ par l'inversion. Ainsi, $G$ est bien un groupe topologique. 
	
	Fixons un sous-ensemble $D\subseteq X$ dénombrable dense. L'application de restriction 
	$\lip(X,X)\to \lip(D,X)$ est bijective car $X$ est complet, et c'est un homéomorphisme si on munit 
	également $\lip(D,X)$ de la topologie produit. L'espace $X^D$ est polonais pour la topologie produit.
	Or $\lip(D,X)$ est fermé dans $X^D$, donc il est polonais, ainsi $\lip(X,X)$ est lui-même polonais.

	Pour voir que $G$ est polonais pour la topologie induite par $X^X$, on remarque alors
	que l'espace des plongements $X\to X$ est fermé dans $\lip(X,X)$, donc polonais.
	L'application qui à $g\in G$
	associe $(g,g^{-1})$ est un homéomorphisme sur son image, et cette image est l'espace
	fermé (donc polonais) des couples de plongements $(f,g)$ tels que
	$f\circ g=g\circ f=\mathrm{id}_X$, donc $G$ est bien polonais.
\end{proof}

\begin{rema}
	Si $G$ est 
	le groupe d'automorphismes d'une structure métrique $(X,d)$, on a $g_n\to g$ si et seulement si
	pour tout $x\in X$, $d(g_n(x),g(x))\to 0$ quand $n\to +\infty$.
	La topologie du groupe d'automorphismes d'une structure métrique séparable étant métrisable,
	elle est complètement décrite par cette remarque.
\end{rema}

Voici deux exemples de groupes d'automorphismes de structures métriques
séparables qui jouent un rôle important dans cet exposé. 

\begin{exem}
	Soit $(X,d)$ un espace métrique séparable complet, vu comme une structure
	sans relation ni fonction. Alors son groupe d'automorphismes est égal à son groupe d'isométries
	$\Iso(X,d)$, qui est donc un groupe polonais pour la topologie de la convergence simple. 
\end{exem}
\begin{exem}
	Soit $\mathcal H$ un espace de Hilbert réel séparable.
	On le munit de 
	la relation donnée par le produit  scalaire
	$\mathcal H^2\to\R$, de la fonction d'addition des vecteurs, de la constante vecteur nul,
	ainsi  que,  pour
	chaque $t\in\R$, de la fonction multiplication par le scalaire $t$.
	Le groupe des automorphismes de la structure ainsi obtenue 
	est le groupe orthogonal de $\mathcal H$, noté $\mathcal O(\mathcal H)$.
	Si $\mathcal H$ est séparable, on obtient donc un groupe polonais, et sa topologie
	est la topologie induite par la topologie forte de l'espace des opérateurs bornés
	$\mathcal B(\mathcal H)$, qui coïncide
	avec la topologie faible en restriction à $\mathcal O(\mathcal H)$.
	
	De manière similaire,
	 le groupe $\mathcal U(\mathcal H)$ des
	unitaires d'un espace de Hilbert complexe~$\mathcal H$ séparable est un groupe polonais pour la topologie forte.
\end{exem}

Nous allons maintenant présenter dans une section à part le \emph{groupe des transformations 
préservant la mesure} de l'intervalle $[0,1]$ muni de la mesure de Lebesgue $\lambda$, noté $\Aut([0,1],\lambda)$.
On va le définir directement comme groupe des automorphismes d'une structure métrique
appelée \emph{algèbre de mesure}. 

\subsection{Le groupe des automorphismes de l'algèbre de mesure de \texorpdfstring{$([0,1],\lambda)$}{[0,1],lambda}}\label{sec: malg}

	Considérons un espace de probabilité $(X,\mathcal B,\mu)$.
	La mesure $\mu$ permet de munir la tribu~$\mathcal B$ d'une pseudo-distance~$d_\mu$ donnée par $d_\mu(A,B)=\mu(A\bigtriangleup B)$. 
	
	Montrons que $d_\mu$ est complète:
	soit $(A_n)$ une suite de Cauchy, quitte à extraire on peut supposer que
	$d_\mu(A_n\bigtriangleup A_{n+1})<2^{-n}$. D'après le lemme de Borel--Cantelli, pour presque
	tout $x\in X$, il existe $N\in\N$ tel que  $x \not\in A_{n}\bigtriangleup A_{n+1}$ pour tout $n\geq N$, 
	ce qui veut dire que $x$ ne ``change plus d'avis'' quant au fait d'appartenir à $A_n$ ou non. 
	On vérifie alors que l'ensemble $\bigcup_{N\in\N}\bigcap_{n\geq N} A_n$ est la limite recherchée.

	Étant donné un espace de probabilité $(X,\mathcal B, \mu)$,
	l'\textbf{algèbre de mesure} de $(X,\mathcal B,\mu)$,
	notée $\MAlg(X,\mu)$, est la structure dont l'espace métrique sous-jacent est
	obtenu par séparation de $(\mathcal B,d_\mu)$ : on identifie deux 
	ensembles $A,B\in\mathcal B$ lorsque $\mu(A\bigtriangleup B)=0$.

	On munit notre algèbre de mesure des fonctions induites par la différence symétrique, l'intersection
	et l'union	(que l'on continue de noter $\bigtriangleup$, $\cap$ et $\cup$),
	qui sont $1$-lipschitziennes donc uniformément continues.
	On le munit également de la relation $1$-lipschitzienne induite par la mesure $\mu$
	(que l'on continue de noter $\mu$) et on ajoute deux constantes représentant
	la classe d'équivalence de l'ensemble vide $\emptyset$ et celle de $X$ tout entier.
	On obtient alors une structure métrique
	$(\MAlg(X,\mu),d_\mu,\emptyset, X,\cup,\cap, \bigtriangleup,\mu)$.
	
	\begin{rema}
	Le nom ``algèbre de mesure'' vient du fait que munie des opérations appropriées,
	il s'agit d'une \emph{algèbre booléenne}, mais la suite
	de l'exposé ne supposera pas de familiarité avec ces dernières. On recommande
	le chapitre 2 de 
	\textcite{coriLogiqueMathematiqueTome2003} pour une introduction aux algèbres booléennes, 
	et l'encyclopédique \textcite{fremlinMeasureTheoryVol2002} pour la théorie des
	algèbres de mesure.
\end{rema}

	\begin{defi}
		Étant donné un espace de probabilité $(X,\mathcal B,\mu)$, on
		note $\Aut(X,\mu)$ le groupe d'automorphismes de la structure 
		métrique $\MAlg(X,\mu)$.
	\end{defi}
	
	Puisque la distance $d_\mu$ est toujours complète, $\Aut(X,\mu)$ est polonais dès 
	lors que $(\MAlg(X,\mu),d_\mu)$ est séparable.
	En utilisant des réunions finies d'intervalles aux extrémités rationnelles,
	on montre que c'est le cas pour $X=[0,1]$ muni de sa tribu borélienne et 
	$\mu$ la mesure de Lebesgue $\lambda$. Ainsi,
	$\Aut([0,1],\lambda)$
	est bien un groupe polonais. 
    Il est souvent présenté de manière plus directe comme un groupe de bijections
	bimesurables (ou \emph{transformations})
	préservant la mesure de $([0,1],\lambda)$, identifiées
	si elles coïncident sur un ensemble de mesure pleine. Il est
	 clair que si $T$ est une bijection bimesurable préservant la mesure 
	entre deux sous-ensembles
	de mesure pleine de $[0,1]$, alors $T$ induit un élément de $\Aut([0,1],\lambda)$, et réciproquement 
	on peut montrer que tout automorphisme de $\MAlg([0,1],\lambda)$ se relève en une bijection bimesurable
	préservant la mesure (voir par exemple \cite[Thm.~2.15]{glasnerErgodicTheoryJoinings2003}), mais nous n'en aurons pas besoin.
	
	On a la propriété fondamentale suivante, dite d'homogénéité approximative.
	Pour nous, une \textbf{partition} d'un espace de probabilité $(X,\mu)$ est une famille finie
	$(A_i)_{i=1}^n$ d'éléments de $\MAlg(X,\mu)$ deux à deux disjoints  telle que $X=A_1\sqcup\cdots\sqcup A_n$
	(on autorise certains~$A_i$ à être vides).

	\begin{prop}\label{prop: approx homo}
		Soient $(A_1,\dots,A_n)$ et $(B_1,\dots,B_n)$ deux partitions de $([0,1],\lambda)$
		telles que pour tout $i\in\{1,\dots,n\}$, on a $\mu(A_i)=\mu(B_i)$.
		Alors pour tout $\epsilon>0$, il existe~$T$ dans $\Aut([0,1],\lambda)$
		tel que pour tout $i=1,\dots,n$, on a
		$\mu(T(A_i)\bigtriangleup B_i)<\epsilon$.
	\end{prop}
	\begin{proof}
		On se ramène tout d'abord, par densité, au cas où les $A_i$ et les $B_i$ sont des
		réunions finies d'intervalles semi-ouverts à extrémités rationnelles. Si $q\geq 1$ est le plus 
		grand commun multiple des dénominateurs des extrémités de ces intervalles, on écrit alors
		chaque $A_i$ et $B_i$ comme réunion disjointe de $q\mu(A_i)=q\mu(B_i)$ 
		intervalles de la forme $[\frac kq,\frac{k+1} q[$. 
		On trouve alors la bijection recherchée 
		en recollant des translations qui envoient les intervalles dans $A_i$ sur ceux dans $B_i$.
	\end{proof}
	\begin{rema}\label{rmq: en fait homo}
		On peut en fait montrer que la structure est exactement homogène, c'est-à-dire qu'on peut 
		trouver $T\in \Aut([0,1],\lambda)$ tel que $T(A_i)=B_i$. Une manière de
		procéder est de construire par va-et-vient un isomorphisme 
		entre des sous-algèbres finies ``de plus en plus denses'',
		en étendant successivement l'isomorphisme naturel entre l'algèbre de mesure finie engendrée 
		par les $A_i$ et celle engendrée par les $B_i$.
	\end{rema}

	On a
	une notion naturelle d'isomorphisme d'algèbres de mesure, de sorte qu'étant donnés deux espaces
	de probabilités $(X,\mathcal B_X,\mu)$ et $(Y,\mathcal B_Y,\nu)$,
	si $\varphi: X_0\to Y_0$ est une bijection bimesurable entre deux sous-ensembles mesurables
	$X_0$ et $Y_0$ de mesure pleine, alors $\varphi$ induit un isomorphisme entre leurs 
	algèbres de mesures. En particulier, on peut utiliser l'écriture dyadique pour vérifier
	que l'algèbre de mesure $\MAlg([0,1],\lambda)$ est isomorphe à l'algèbre de mesure
	$\MAlg(\{0,1\}^\N,\mu)$, où $\mu$ est la mesure
	de Bernoulli standard $(\frac 12 \delta_0+\frac 12\delta_1)^{\otimes\N}$.
	
	Ceci nous permet d'identifier $\Aut([0,1],\lambda)$ à $\Aut(\{0,1\}^\N,\mu)$, ou plus généralement
	à $\Aut(\{0,1\}^{X},(\frac 12 \delta_0+\frac 12\delta_1)^{\otimes X})$, où $X$ est un ensemble 
		dénombrable infini \footnote{C'est aussi une conséquence du résultat suivant, plus difficile à démontrer : 
			si $X$ est un espace polonais muni d'une mesure
			de probabilité $\mu$ borélienne sans atomes, alors il existe une bijection bimesurable
			entre $X$ et $[0,1]$ qui préserve la mesure (cf.\ le théorème 17.41 de \cite{kechrisClassicaldescriptiveset1995}).}.
	Dans la suite 
	de l'exposé on utilisera une algèbre de mesure de cette forme pour construire explicitement l'action 
	préservant la mesure de $\mathbb F_2$ dont on aura besoin pour la preuve du théorème principal
	(cf.\ section~\ref{sec: construction action}).

\section{Groupes polonais Roelcke-précompacts}\label{sec: prp}

On va s'intéresser aux actions continues des groupes polonais sur des espaces métriques
afin de caractériser
les groupes polonais Roelcke-précompacts.
Le thème des actions isométriques continues de groupes polonais a pris une 
grande importance depuis les travaux 
de Rosendal sur la propriété~(OB),
puis sur la géométrie des groupes polonais en général \parencite{rosendalCoarseGeometryTopological2018}, 
mais on en dira très peu de choses ici. 

Un outil important est l'espace 
des adhérences d'orbites d'une action isométrique ou des actions diagonales associées à cette action. 
Ceci implique
de faire un choix de distance sur les espaces produits de sorte qu'un produit d'actions isométriques
reste isométrique.

\begin{conv}
	Étant donnés des espaces métriques $(X_i,d_i)_{i\geq 1}$, pour tout $n\in\N$
	on munira systématiquement l'espace produit 
	$\prod_{i=1}^n X_i$
	 de la distance compatible $d^n$
	donnée par
	$$d^n((x_i)_{i=1}^n,(y_i)_{i=1}^n)\coloneq\sum_{i=1}^nd_i(x_i,y_i).$$
	De plus, on munira $\prod_{i\geq 1} X_i$ de la distance compatible $d^\omega$ donnée par
	$$d^\omega((x_i)_{i\geq 1},(y_i)_{i\geq 1})\coloneq\sum_{i\geq 1}2^{-i}\min(1,d(x_i,y_i)).$$
\end{conv}

Notons que si chaque $d_i$ est complète, alors pour tout $n\in\N$ la distance $d^n$ est complète, et $d^\omega$ 
est également complète. 

Cette section suppose parfois une certaine familiarité avec les structures uniformes, comme définies
dans le chapitre 2 de \textcite{bourbakiTopologieGeneraleChapitres2007}, notamment
afin de présenter la notion de Roelcke-précompacité
dans son cadre naturel (section~\ref{sec:Roelcke uniforme}). 
Nous fournirons toutefois des caractérisations qu'un lecteur 
peu familier avec ces dernières pourra préférer (cf.\ section~\ref{sec:cara prp}). 

On suit de près la section 2 de l'article de 
\textcite{benyaacovWeaklyAlmostPeriodic2016}.

\begin{rema}
	Dans la convention précédente,
	n'importe quel choix de distance qui induise la structure uniforme produit et 
	tel que tout produit d'isométries soit une isométrie sur l'espace produit conviendrait.
\end{rema}

\subsection{Actions approximativement oligomorphes}\label{sec:oligo}

Soit $G$ un groupe agissant par isométries sur un espace métrique $(X,d)$,
alors on peut munir $X$ d'une pseudo-distance\footnote{Le fait que ce soit bien
	une pseudo-distance utilise de manière cruciale le fait que $G$ agisse par isométries.} 
$G$-invariante
$\tilde d_G$ définie par $$\tilde d_G(x,y)\coloneq\inf_{g\in G} d(g\cdot x,y).$$
Remarquons tout de suite que comme l'action est isométrique
$$\tilde d_G(x,y)=\inf_{g,h\in G}d(h\inv g\cdot x,h\inv\cdot y)=\inf_{g,h\in G}d( g\cdot x,h\cdot y).$$
 
On a $\tilde d_G(x,y)=0$ si et seulement si $\overline{Gx}=\overline{Gy}$, et
on note $(G\bbslash X,d_{G})$ l'espace métrique obtenu par séparation,
qui s'identifie donc à l'espace des adhérences d'orbites. L'application de 
passage au quotient
$x\mapsto \overline{Gx}$ est uniformément continue et surjective.

Soit $(Y,d_Y)$ un espace métrique. Par passage au quotient,
on a une identification naturelle entre les fonctions continues $G$-invariantes $X\to Y$
et les fonctions continues $G\bbslash X\to Y$. Cette identification préserve l'uniforme continuité.

\begin{lemm}\label{lem:completude}
	Si $(X,d)$ est un espace métrique complet et $G$ agit par 
	isométries sur~$X$, alors $(G\bbslash X,d_G)$
	est un espace métrique complet.
\end{lemm}
\begin{proof}
	Pour $x\in X$, on note $[x]$ l'adhérence de sa $G$-orbite.
	Soit $([x_n])$ une suite de Cauchy, il suffit de montrer qu'elle admet une sous-suite
	convergente. Quitte à extraire, on peut supposer que $d_G([x_n],[x_{n+1}])<2^{-n}$.
	On construit alors par récurrence une suite $(y_n)$ en posant $y_0=x_0$
	puis en choisissant pour chaque $n\in\N$ un élément $y_{n+1}\in [x_{n+1}]$ tel que
	$d(y_{n+1},y_n)<2^{-n}$. La suite $(y_n)$ est de Cauchy et sa limite $y$
	satisfait bien $[y]=\lim_{n\to+\infty}[x_n]$.
\end{proof}

Notons que pour tout $n\in\N$, l'action diagonale de $G$ sur $X^n$ (donnée par 
$g\cdot (x_1,\dots,x_n)\coloneq(g\cdot x_1,\dots,g\cdot x_n)$) est isométrique pour la distance $d^n$ définie au
début de cette section.

\begin{defi}
	Une action par isométries d'un groupe $G$ sur un espace métrique  
	$(X,d)$ est \textbf{approximativement cocompacte}
	si l'espace métrique $(G\bbslash X,d_G)$ est précompact.
\end{defi}

Concrètement, une action par isométries de $G$ sur $(X,d)$ est approximativement cocompacte
si pour chaque $\epsilon>0$ il existe une partie finie $X_0\subseteq X$ telle que 
pour tout~$x$ dans~$ X$, il existe $g\in G$ tel que $d(g\cdot x,X_0)<\epsilon$ (on rappelle que pour
$A\subseteq X$ non vide on définit $d(x,A)\coloneq\inf_{a\in A} d(x,a)$).

D'après le lemme~\ref{lem:completude}, l'action de $G$ sur un espace 
métrique \emph{complet} $(X,d)$ est
approximativement cocompacte si et seulement si l'espace métrique $(G\bbslash X,d_G)$ est
compact.
De plus par densité l'action de $G$  par isométries sur un espace métrique $(X,d)$
est approximativement cocompacte si et seulement si son extension  au complété $(\widehat X,d)$
est approximativement cocompacte, et le complété $\widehat{G\bbslash X}$ s'identifie 
naturellement à $G\bbslash \widehat X$.
Enfin, si $(Y,d_Y)$ est un espace métrique complet
et l'action par isométries de $G$ sur $(X,d)$ est approximativement cocompacte, alors les fonctions continues $G\bbslash \widehat X\to Y$
s'identifient aux fonctions uniformément continues bornées $G$-invariantes $X\to Y$.

\begin{exem}
	Soit $G$ un groupe polonais, soit $d_{\mathcal L}$ une distance invariante à gauche sur $G$.
	Alors l'action par translation à gauche de $G$ sur $(G,d_{\mathcal L})$ est approximativement
	cocompacte puisque transitive, et l'action par translation à gauche sur le complété de $(G,d_{\mathcal L})$ est donc également approximativement
	cocompacte. 
	Par contre, on verra que l'action diagonale par translation à gauche de $G$ sur 
	$(G\times G,d_{\mathcal L}^2)$
	est approximativement cocompacte si et seulement si $G$ est Roelcke-précompact.
	Pour l'heure, remarquons que si $d_{\mathcal L}$ n'est pas bornée,
	alors $G\bbslash (G\times G)$ n'est pas borné, donc pas précompact et l'action n'est ainsi
	pas approximativement cocompacte.
\end{exem}

\begin{defi}
	Une action par isométries d'un groupe $G$ sur un espace métrique $(X,d)$
	est \textbf{approximativement oligomorphe} si pour tout 
	$n\in\N$, l'action diagonale de $G$ sur $(X^n,d^n)$ est approximativement
	cocompacte.
\end{defi}

\begin{rema}\label{rmq: action approx cocompacte sur Xn}
	Comme remarqué par \textcite{benyaacovWeaklyAlmostPeriodic2016}, 
	étant donnée une action par isométries de $G$ sur $(X,d)$,
	cette action
	est approximativement oligomorphe si et seulement si l'action diagonale de $G$ sur $(X^\N,d^\N)$
	est approximativement cocompacte. L'implication directe se montre
	en revenant  à la définition, notant que la distance
	$d^\N$ est très peu sensible à ce qui se passe sur de grandes coordonnées.
	Pour la réciproque, on note que la projection $X^\N\to G\bbslash X^n$
	est uniformément continue et surjective. Elle 
	passe donc au quotient en une application $G\bbslash X^\N\to G\bbslash X^n$ uniformément continue
	surjective. 
	Ainsi, si 
	$G\bbslash X^\N$ est précompact alors $G\bbslash X^n$ aussi. 
\end{rema}

\begin{exem}
	Considérons un groupe $G$ agissant sur un ensemble discret
	$X$ muni de la distance discrète afin que
	l'action soit isométrique.
	Les orbites de $G$ sont alors fermées et la distance
	$d_G$ est la distance discrète sur l'espace des orbites.
	L'action est approximativement oligomorphe si et seulement si pour tout $n\in\N$,
	l'action de~$G$ sur~$X^n$ n'a qu'un nombre fini d'orbites, 
	ce qui signifie que l'action est \emph{oligomorphe}.
	Un exemple de base est fourni par l'action de $\mathfrak S_{\infty}=\mathfrak S(\N)$
	sur $\N$.
	
	Un autre exemple est donné par le groupe $\Aut(R)$ des automorphismes du graphe aléatoire 
	(exemple~\ref{ex: random graph}).
	On montre que deux $n$-uplets $(x_1,\dots,x_n)$ et $(y_1,\dots,y_n)$ sont dans la même orbite
	si et seulement l'application $x_i\mapsto y_i$ est un isomorphisme pour les graphes induits
	(c'est clairement une condition nécessaire, la réciproque se montre par le même 
	va-et-vient que celui qui suit l'exemple~\ref{ex: random graph}, en remplaçant $\varphi_0$
	par l'isomorphisme entre graphes induits $x_i\mapsto y_i$).
	En particulier, $\Aut(R)\bbslash \N^n$ est bien fini.
\end{exem}

\begin{exem}\label{ex: approxi oligo pour malg}
	L'action de $\Aut([0,1],\lambda)$ sur $\MAlg([0,1],\lambda)$ est approximativement oligomorphe.
	On va le montrer en détail puisqu'il s'agit du fil rouge de cet exposé.

	Soit $n\in\N$, on note $\mathcal P_n$ l'espace des partitions à $n$ éléments de $([0,1],\lambda)$.

	Pour $A\in\MAlg(X,\mu)$, notons $A^0\coloneq A$ et $A^1\coloneq X\setminus A=X\bigtriangleup A$. 
	L'application 
	$\Phi\colon (A_1,\dots,A_n)\mapsto\bigl(\mu(A_1^{\delta(1)}\cap\cdots\cap A_n^{\delta(n)})\bigr)_{\delta\in\{0,1\}^n}$
	est une injection $G$-équivariante 
	de $\MAlg(X,\mu)^n$ dans $\mathcal P_{2^n}$ qui est bilipschitzienne (pour voir que son inverse
	est bien lipschitzien, on note que $A_i$ se reconstruit comme la réunion des 
	$A^{\delta(1)}\cap\cdots\cap A^{\delta(n)}$
	pour~$\delta$ dans~$\{0,1\}^n$ satisfaisant $\delta(i)=0$).
	
	Il suffit donc de montrer que pour tout $n\in\N$, $\Aut(X,\mu)\bbslash \mathcal P_n$
	est précompact. Il est immédiat par la proposition~\ref{prop: approx homo} que l'orbite fermée
	de $(A_1,\dots,A_n)$ est déterminée par la mesure de probabilité sur l'ensemble
	$\{1,\dots, n\}$ donnée par $(\mu(A_1),\dots,\mu(A_n))$.
	Comme l'espace $\mathfrak P_n$ des mesures de probabilité sur $\{1,\dots,n\}$ est compact, il suffit alors
	de remarquer que la bijection
	\begin{align*}
	\mathfrak P_n&\to \Aut(X,\mu)\bbslash \mathcal P_n\\
	(p_1,\dots,p_n)&\mapsto [[0,p_1],[p_1,p_1+p_2],\dots ,[1-p_n, 1]]
	\end{align*}
	est continue pour conclure.
\end{exem}
\begin{rema}\label{rmq:probas sur Cantor}
	En mettant les deux étapes ci-dessus bout à bout, on voit que l'espace d'orbites fermées 
	$\Aut([0,1],\lambda)\bbslash \MAlg([0,1],\lambda)^n$ est naturellement homéomorphe 
	à l'espace des mesures de probabilité sur $\{0,1\}^n$, ce qui permet ensuite de montrer
	que  $\Aut([0,1],\lambda)\bbslash \MAlg([0,1],\lambda)^\N$ est naturellement homéomorphe à 
	l'espace des mesures de probabilité sur l'espace de Cantor $\{0,1\}^\N$.
\end{rema}

\begin{rema}
	Bien que les orbites de l'action diagonale de $\Aut(X,\mu)$ sur $\MAlg(X,\mu)^n$ soient fermées
	pour tout $n\in\N$ (et ce, d'après la remarque~\ref{rmq: en fait homo}),
	celle de l'action diagonale sur $\MAlg(X,\mu)^\N$ ne le sont pas. Cette action n'est pas cocompacte, 
	mais seulement approximativement cocompacte.
	
	En utilisant la théorie des modèles continue,
	Ben Yaacov a montré que le groupe polonais Roelcke-précompact $\Aut^*([0,1],\lambda)$
	des transformations qui quasi-préservent la mesure de Lebesgue
	n'admet \emph{aucune} action par isométries cocompacte sur un espace métrique complet
	\parencite{benyaacovRoelckeprecompactPolishGroup2018}.
\end{rema}

\subsection{Complété de Roelcke d'un groupe polonais}\label{sec:Roelcke uniforme}

Dans ce qui suit, étant donné un groupe $G$, une partie $A\subseteq G$ et un élément $g\in G$, 
on note $Ag\coloneq\{ag\colon a\in A\}$ et $gA\coloneq\{ga\colon 
a\in A\}$. Si de plus $B\subseteq G$, on note $AB\coloneq\{ab\colon a\in A, b\in B\}$.

Rappelons tout d'abord que tout groupe topologique $G$
est muni d'une structure uniforme à gauche, notée $\mathcal L$, qui est engendrée par les entourages de la 
diagonale de la forme $\{(g,h)\in G^2\colon g\in hU\}$ où $U$ parcourt les voisinages ouverts de l'identité.
Il est également muni d'une structure uniforme à droite, notée $\mathcal R$, 
engendrée par les entourages de la diagonale
 de la forme $\{(g,h)\in G^2\colon g\in Uh\}$ où $U$ parcourt les voisinages ouverts de l'identité.
 L'application $g\mapsto g\inv$ induit un isomorphisme d'espaces uniformes
 entre $(G,\mathcal L)$ et $(G,\mathcal R)$. 
  
 Toute  distance 
 $d_{\mathcal L}$ sur $G$ compatible avec sa topologie et invariante à gauche induit la structure uniforme
 $\mathcal L$. De même, toute distance 
 $d_{\mathcal R}$ sur $G$ compatible avec sa topologie et invariante à droite induit $\mathcal R$.
 Lorsque $G$ est polonais, de telles distances existent bien d'après le théorème de Birkhoff--Kakutani
 (cf.\ \cite[Thm.~2.1.1]{gaoInvariantDescriptiveSet2009}). Si $G$ est le groupe d'automorphismes d'une
 structure métrique séparable $(X,d_X)$, une distance compatible invariante à gauche $d_{\mathcal L}$ 
 peut se construire comme suit :
 on fixe $(\alpha_i)_{i\in\N}$ dense dans $X$ et on pose 
 $$d_{\mathcal L}(g,h)\coloneq\sum_{i\in\N} 2^{-i}\min(1,d_X(g\cdot \alpha_i,h\cdot \alpha_i)).$$
Par définition, la \textbf{structure de Roelcke} est la structure uniforme la plus grande contenue
à la fois dans les structures uniformes à gauche et à droite. Elle est notée~\mbox{$\mathcal L\wedge\mathcal R$}.
\begin{rema}
	Une quatrière structure uniforme naturelle sur un groupe topologique quelconque est
	celle de Raikov, qui est la plus petite contenant à la fois les structures uniformes à 
	gauche et à droite.
	Si $d_{\mathcal L}$ et $d_{\mathcal R}$ sont compatibles, invariantes à gauche et à droite
	respectivement, alors $d_{\mathcal L}+d_{\mathcal R}$ est compatible et induit la structure uniforme
	de Raikov.
	On peut alors caractériser les groupes polonais comme étant 
	les groupes topologiques à base dénombrable d'ouverts \emph{Raikov-complets} (complets
	pour la structure uniforme de Raikov).
\end{rema}

\begin{lemm}\label{lem: base Roelcke}
	Une base de la structure uniforme de Roelcke est donnée par les entourages de la diagonale de la 
	forme 
	\[
	\{(g,h)\colon g\in G, h\in UgU\},
	\]
	où $U$ parcourt une base de voisinages de l'identité de $G$.
\end{lemm}
\begin{proof}	
Si $U$ est un voisinage de l'identité, il contient cette dernière, donc pour tout $g\in G$ l'ensemble
 $UgU$ contient à la fois $Ug$ et $gU$.
Ainsi l'ensemble $\{(g,h)\colon g\in G,h\in UgU\}$ contient à la fois 
$\{(g,h)\colon g\in G, h\in Ug\}$ et $\{(g,h)\colon g\in G, h\in gU\}$ 
donc est bien dans $\mathcal L\wedge \mathcal R$. 

Réciproquement, si $W\in\mathcal L\wedge\mathcal R$, alors on trouve $W'\in\mathcal L\wedge \mathcal R$ 
tel que $W'\circ W'\subseteq W$.
Par définition, $W'$ contient à la fois un élément de $\mathcal L$ 
et un élément de $\mathcal R$. 
Quitte à prendre un voisinage de l'identité symétrique $U$
suffisamment petit, on a donc que $g\in hU$ ou $g\in Uh$ implique $(g,h)\in W'$. En composant on voit que
si $h\in UgU$ alors $(g,h)\in W'\circ W'$, donc $(g,h)\in W$.
\end{proof}

Le lemme précédent et la continuité de la multiplication impliquent que la structure uniforme de Roelcke 
est bien compatible avec la topologie.
%

\begin{defi}
	Un groupe polonais est 
	\textbf{Roelcke-précompact} lorsque sa structure uniforme de Roelcke est précompacte.
\end{defi}

\begin{rema}
	D'après le lemme précédent, $G$ est Roelcke-précompact si et seulement si pour tout voisinage ouvert
	de l'identité $U\subseteq G$ , on peut trouver une partie finie $F\subseteq G$ telle que $G\subseteq UFU$.	
\end{rema}

Supposons désormais $G$ métrisable, soit $d_{\mathcal L}$ une distance invariante à gauche sur 
$G$. 
L'action diagonale à gauche de $G$ sur $G\times G$ muni de la distance $d((g_1,g_2),(h_1,h_2))\coloneq
d_{\mathcal L}(g_1,g_2)+d_{\mathcal L}(h_1,h_2)$ est une action par isométries.
Ses orbites sont fermées donc $G{\bbslash} (G\times G)=G\backslash (G\times G)$, et la distance quotient
sur $G{\bbslash} (G\times G)$ est donnée par 
\[
d_G([1,g_1],[1,g_2])=\inf_{h\in G} \left(d_{\mathcal L}(h,1)+d_{\mathcal L}(hg_1,g_2)\right).
\]

Si on pose $d_{\mathcal R}(g,h)\coloneq d_{\mathcal L}(g^{-1},h^{-1})$, alors 
$d_{\mathcal R}$ est une distance compatible invariante à droite sur $G$. 
En remplaçant $h$ par $hg_1^{-1}$ 
et en utilisant la symétrie,
on trouve 
$$d_G([1,g_1],[1,g_2])=\inf_{h\in G} \left(d_{\mathcal L}(g_1,h)+d_{\mathcal R}(h,g_2)\right).$$
L'application $g\mapsto (1,g)$ passe au quotient en une bijection
entre $G$ et $G{\bbslash} (G\times G)=G\backslash (G\times G)$. 
Le lemme suivant nous dit que modulo cette identification entre $G$ et $G\bbslash (G\times G)$, on a obtenu 
une distance compatible avec la structure de Roelcke.
\begin{lemm}
	Une distance compatible avec la structure uniforme de Roelcke sur $G$ est donnée par
	$$d_{\mathcal L\wedge \mathcal R}(g_1,g_2)=\inf_{h\in G}\left(d_{\mathcal L}(g_1,h)+d_{\mathcal R}(h,g_2)\right).$$
\end{lemm}
\begin{proof}
	Donnons-nous un voisinage symétrique de l'identité $U$ de diamètre au plus $\epsilon$ pour 
	$d_{\mathcal L}$ (et donc par symétrie de diamètre au plus $\epsilon$ pour $d_{\mathcal R}$). 
	Si $g_1\in Ug_2U$, on écrit $g_1u_1=u_2g_2$ avec $u_1,u_2\in U$, 
	alors $d_{\mathcal L}(g_1,g_1u_1)\leq \epsilon$ et $d_{\mathcal R}(g_2,u_2g_2)\leq \epsilon$
	donc $d_{\mathcal L\wedge\mathcal R}(g_1,g_2)\leq 2\epsilon$.
	
	Réciproquement, si on se donne $\epsilon>0$ et une boule ouverte $U$
	autour de l'identité de rayon $\epsilon$ pour la distance $d_{\mathcal L}$
	et pour la distance $d_{\mathcal R}$, si $d_{\mathcal L\wedge \mathcal R}(g_1,g_2)<\epsilon$,
	on trouve $h$ tel que $d_{\mathcal L}(g_1,h)+d_{\mathcal R}(h,g_2)<\epsilon$, 
	ainsi $g_1\in hU$ et $h\in Ug_2$ donc $g_1\in Ug_2U$.
\end{proof}

Soit maintenant $\widehat G_{\mathcal L}$ le complété à gauche de $G$,
notons toujours $d_{\mathcal L}$ la distance obtenue sur $\widehat G_{\mathcal L}$
et $d$ la distance sur $\widehat G_{\mathcal L}\times\widehat G_{\mathcal L}$ définie par 
$d((g_1,g_2),(h_1,h_2))=
d_{\mathcal L}(g_1,g_2)+d_{\mathcal L}(h_1,h_2)$. 
L'action précédente de $G$ s'étend par densité et complétude en une action isométrique
sur $(\widehat G_{\mathcal L}\times \widehat G_{\mathcal L},d)$. 

L'espace métrique quotient $G{\bbslash}(\widehat G_{\mathcal L}\times \widehat G_{\mathcal L})$ est complet.
D'après le lemme précédent l'application $g\mapsto [1,g]$ est un plongement  isométrique
de $(G,d_{\mathcal L\wedge\mathcal R})$ dans $G{\bbslash}(G_{\mathcal L}\times G_{\mathcal L})$, 
et son image est dense. Ceci montre le lemme suivant, où par \textbf{complété de Roelcke} on entend
complété pour la structure uniforme de Roelcke.

\begin{lemm}
	Le complété de Roelcke de $G$ s'identifie à l'espace métrique complet 
	$G{\bbslash}(\widehat G_{\mathcal L}\times \widehat G_{\mathcal L})$.\qed
\end{lemm}

\subsection{Caractérisations des groupes polonais Roelcke-précompacts}\label{sec:cara prp}

On vient d'identifier le complété de Roelcke de $G$ avec l'espace métrique 
$G{\bbslash}(\widehat G_{\mathcal L}\times \widehat G_{\mathcal L})$, ce qui a pour corollaire
immédiat le résultat suivant.

\begin{prop}\label{prop: roelcke action Gl}
	Soit $G$ un groupe polonais, soit $d_{\mathcal L}$ une distance invariante à gauche sur $G$
	compatible avec sa topologie et soit $\widehat G_{\mathcal L}$ le complété de $G$ pour 
	$d_{\mathcal L}$.
	Alors $G$ est Roelcke-précompact si et seulement si 
	$G{\bbslash}(\widehat G_{\mathcal L}\times \widehat G_{\mathcal L})$ est compact.\qed
\end{prop}

\begin{rema}\label{rmq: lc roelcke}
	Dans le cas où $G$ admet une distance invariante à gauche compatible complète (autrement
	dit, si sa structure uniforme gauche est complète), on a $\widehat G_{\mathcal L}=G$
	et donc $G{\bbslash}(\widehat G_{\mathcal L}\times \widehat G_{\mathcal L})=
	 G{\bbslash}(G\times G)=G\backslash(G\times G)$,
	qui est homéomorphe à $G$. Ainsi $G$ est Roelcke-précompact si et seulement si $G$ est compact.
	Cette remarque s'applique en particulier aux groupes polonais localement compacts  car
	ils admettent toujours une distance invariante à gauche compatible complète.
\end{rema}

On peut maintenant donner les caractérisations de la Roelcke-précompacité
d'un groupe polonais $G$ en termes d'actions \emph{continues} par isométries
sur des espaces métriques $(X,d)$, c'est-à-dire de morphismes continus $G\to \Iso(X,d)$.
Ce résultat va nous fournir de nombreux exemples de groupes polonais Roelcke-précompacts.

\begin{theo}[{\cite[Thm.~2.4]{benyaacovWeaklyAlmostPeriodic2016}}]
	Soit $G$ un groupe polonais. S'équivalent:
	\begin{enumerate}[(i)]
		\item $G$ est Roelcke-précompact;
		\item dès que $G$ agit continûment par isométries sur des espaces métriques $X$
		et $Y$ de manière approximativement cocompacte, l'action diagonale sur $X\times Y$
		est approximativement cocompacte;
		\item dès que $G$ agit continûment par isométries sur un espace métrique 
		de manière approximativement cocompacte, l'action est approximativement oligomorphe;
		\item on peut plonger $G$ dans le groupe d'isométries d'un espace métrique $(X,d)$ complet
		séparable de sorte que l'action associée sur $X$ soit approximativement oligomorphe;
		\item le groupe $G$ est isomorphe au groupe d'automorphismes d'une structure
		métrique séparable sur laquelle il agit de manière approximativement oligomorphe.
	\end{enumerate}
\end{theo}
\begin{proof}
	
	On montre d'abord que les quatre premières assertions sont équivalentes via la chaîne d'implications
	(i) $\Rightarrow$ (ii) $\Rightarrow$ (iii) $\Rightarrow$ (iv) $\Rightarrow$ (i).\\
	
	(i) $\Rightarrow$ (ii). 
	Supposons que $G$ agit continûment par isométries sur $(X,d_X)$ et $(Y,d_Y)$ de
	manière approximativement cocompacte. On rappelle que par convention on munit $X\times Y$
	de la distance $d^2$ donnée par \[
	d^2((x_1,y_1),(x_2,y_2))=d_X(x_1,x_2)+d_Y(y_1,y_2).
	\]	
	Soit $\epsilon>0$, on va montrer
	que l'on peut recouvrir $G\bbslash X\times Y$ par un nombre fini de boules de rayon $4\epsilon$. 
	
	Par précompacité on trouve $X_0\subseteq X$ et $Y_0\subseteq Y$
	finis tels que $G\bbslash X$ soit recouvert par la réunion sur $x\in X_0$ de
	boules de rayon $\epsilon$ centrées en $[x]$
	et que $G\bbslash Y$ soit recouvert par la réunion sur $y\in Y_0$ de
	boules de rayon $\epsilon$ centrées en $[y]$.  Par continuité de l'action de~$G$
	sur~$X$ et sur~$Y$, on trouve un voisinage de l'identité $U\subseteq G$ tel que
	pour tout~$x$ dans~$ X_0$, $d_X(g\cdot x,x)<\epsilon$ et pour tout $y\in Y_0$,
	$d_Y(g\cdot y,y)<\epsilon$. 
	
	D'après le lemme~\ref{lem: base Roelcke}, on dispose d'une partie finie $F\subseteq G$
	telle que $G=UFU$. On va montrer que tout $[x,y]\in G\bbslash (X\times Y)$ 
	est à distance au plus $2\epsilon$ d'un élément du projeté de $(F\cdot X_0\times Y_0)$
	sur $G\bbslash (X\times Y)$.
	
	Soit donc $[x,y]\in G\bbslash (X\times Y)$. Alors il existe $g_1\in G$ tel que $g_1\cdot x$
	soit $\epsilon$-proche de~$X_0$, donc quitte à remplacer $(x,y)$ par $(g_1\cdot x,g_1\cdot y)$ 
	on peut et va supposer
	que $x$ est $\epsilon$-proche de $x_0\in X_0$. 
	
	On dispose de $g\in G$ tel que $g\cdot y$
	soit $\epsilon$-proche de $y_0\in Y_0$. On écrit $g=u_1fu_2$ avec $u_1,u_2\in U$ et $f\in F$.
	
	Comme $d(x,x_0)<\epsilon$ et l'action est par isométries, 
	l'élément $fu_2\cdot x$ est $\epsilon$-proche de $fu_2\cdot x_0$, lui-même $\epsilon$-proche
	de $f\cdot x_0$ puisque $d(x_0,u_2\cdot x_0)<\epsilon$, donc 
	par inégalité triangulaire $d(fu_2\cdot x, f\cdot x_0)<2\epsilon$.
	
	D'autre part, $d(fu_2\cdot y,y_0)=d(u_1fu_2\cdot y, u_1\cdot y_0)$, or $d(u_1\cdot y_0,y_0)<\epsilon$,
	donc $d(u_1fu_2\cdot y, u_1\cdot y_0)<\epsilon +d(u_1fu_2\cdot y, y_0)$ et
	on conclut que $d(fu_2\cdot y,y_0)<2\epsilon$.
	
	Ainsi $d^2(fu_2\cdot(x,y), (f\cdot x_0,y_0))<4\epsilon$.	
	On conclut que tout élément de $G\bbslash X\times Y$ est à distance au plus
	$4\epsilon$ d'un élément du projeté de l'ensemble fini $F\cdot X_0\times Y$ sur 
	$G\bbslash X\times Y$. Comme $\epsilon$ était arbitraire, ceci montre bien que l'action diagonale
	sur $X\times Y$ est approximativement cocompacte.
	\\

	(ii) $\Rightarrow$ (iii). Si l'action de $G$ sur $X$ est approximativement 
	cocompacte, alors la condition (ii) donne par récurrence que 
	l'action sur $X^n$ est approximativement cocompacte pour tout $n\in\N$, donc l'action 
	$G\acts X$ est approximativement oligomorphe.\\

	(iii) $\Rightarrow$ (iv). L'action transitive de $G$ sur $(G,d_{\mathcal L})$ est 
	isométrique, continue, et l'application $G\to \Iso(G_{\mathcal L})$ associée est un plongement de 
	groupes topologiques puisque $g_n\to g$ si et seulement si $g_n\cdot e\to g\cdot e$.
	
	Par complétude cette action s'étend de manière unique en un plongement 
	$G\to \Iso(\widehat G_{\mathcal L},d_{\mathcal L})$, et l'action associée est approximativement
	cocompacte, donc approximativement oligomorphe par (iii).\\

	(iv) $\Rightarrow$ (i). 
	Soit $(\alpha_i)_{i\in\N}$ une suite dense
	d'éléments de $X$, l'application $\Phi:(G,d_{\mathcal L})\to X^\N$
	définie par $\Phi(g)=(g\cdot \alpha_i)_{i\in\N}$ est uniformément continue. 
	Par densité de $(\alpha_i)_{i\in\N}$ et comme $G$ est plongé dans $\Iso(X,d)$, 
	l'application réciproque est également uniformément 
	continue.
	
	De plus, $\Phi$ est $G$-équivariante, et $\Phi$ est donc un plongement bi-uniformément continu
	de l'action $G\acts G_{\mathcal L}$ dans $G\curvearrowright X^\N$.
	L'application $\Phi\times\Phi$ est alors un plongement bi-uniformément continu de 
	l'action diagonale $G\acts G_{\mathcal L}\times G_{\mathcal L}$ dans $G\acts X^\N\times X^\N\simeq X^\N$.
	
	Comme $G\acts X^n$ est approximativement cocompacte pour tout $n\in\N$,
	l'action $G\acts X^\N$ est également approximativement cocompacte
        (cf.\ remarque~\ref{rmq: action approx cocompacte sur Xn}).
	L'existence d'un plongement bi-uniformément continu de $G\acts G_{\mathcal L}\times G_{\mathcal L}$
	dans $G\acts X^\N$ nous garantit alors que
	$G\acts G_{\mathcal L}\times G_{\mathcal L}$ est également approximativement cocompacte,
	ce qui d'après la proposition~\ref{prop: roelcke action Gl} montre que 
	$G$ est Roelcke-précompact.\\
	
	On a donc montré que les quatre premières assertions sont équivalentes. Reste à montrer 
	que (iv) et (v) sont équivalentes. L'implication (v) $\Rightarrow$ (iv) est claire, et pour montrer la 
	réciproque on va suivre une construction de Melleray qui montre que tout groupe polonais
	est isomorphe au groupe des automorphismes d'une structure métrique séparable.
	
	On suppose donc que $G$ est plongé dans $\Iso(X,d)$ et agit de manière approximativement
	oligomorphe. On identifie $G$ à sont image dans $\Iso(X,d)$ qui est fermée.
	Pour chaque $n\in\N$ et chaque adhérence d'orbite $[(x_1,\dots,x_n)]\in G\bbslash X^n$,
	on considère la relation $R_{[(x_1,\dots,x_n)]}$ donnée par 
	$R_{[(x_1,\dots,x_n)]}(y_1,\dots,y_n)=d^n((y_1,\dots,y_n),[x_1,\dots,x_n])$. Montrons
	que l'ajout à $X$ de toutes les relations $R_{[(x_1,\dots,x_n)]}$ 
	réduit son groupe d'automorphismes à~$G$.
	
	Tout d'abord, il est clair que tout élément de $G$ est un automorphisme de notre structure.
	Pour la réciproque, on remarque tout d'abord que
	d'après le fait~\ref{fait:sg fermé} le groupe $G$ est plongé comme sous-groupe fermé de $\Iso(X,d)$
	car les deux sont polonais. Soit maintenant $\sigma$ un automorphisme de $(X,d)$
	muni de la structure définie ci-dessus. Montrons que $\sigma$ est dans l'adhérence de $G$,
	ce qui terminera la preuve puisque $G$ est fermé.
	
	Donnons-nous $x_1,\dots,x_n\in X$ et $\epsilon>0$,
	on doit trouver $g\in G$ tel que $d(\sigma(x_i),g\cdot x_i)<\epsilon$
	pour tout $i\in\{1,\dots,n\}$. Or $R_{[(x_1,\dots,x_n)]}(x_1,\dots,x_n)=0$,
	donc comme $\sigma$ est un automorphisme 
	$R_{[(x_1,\dots,x_n)]}(\sigma(x_1),\dots,\sigma(x_n))=0$, ce qui par définition
	de $R_{[(x_1,\dots,x_n)]}$
	donne le résultat voulu.
\end{proof}
\begin{rema}\label{rmq: groupe autor struct metri}
	Dans la construction précédente, on aurait pu se contenter de choisir des sous-ensembles
	dénombrables denses $D_n$ de $X^n$ pour tout $n\in\N$ afin de n'ajouter à la structure 
	que les 
	$R_{(x_1,\dots,x_n)}$ où $(x_1,\dots,x_n)\in D_n$; on obtient alors $G$ comme groupe d'automorphismes
	d'une structure métrique séparable avec un nombre dénombrable de relations, et aucune fonction.
	De plus, la preuve montre que tout groupe polonais, Roelcke-précompact ou non,
	se réalise comme groupe d'automorphismes d'une structure métrique séparable.
\end{rema}

\begin{exem}
	Les exemples d'actions approximativement oligomorphes de la section~\ref{sec:oligo} nous fournissent des
	des groupes polonais
	Roelcke-précompacts via la caractérisation (v) du théorème précédent. Ainsi le groupe $\mathfrak S_\infty$ des permutations de $\N$, le groupe $\Aut(R)$ des automorphismes du graphe aléatoire et le groupe $\Aut(\Q,<)$ des bijections préservant l'ordre des rationnels
	sont des groupes d'automorphismes de structures dénombrables dont l'action naturelle
	est oligomorphe: ils sont Roelcke-précompacts\footnote{Dans le cadre restreint des groupes d'automorphismes
		 de structures dénombrables,
	le théorème qui précède est dû à \textcite{tsankovUnitaryRepresentationsOligomorphic2012}.}.
	Dans le cadre continu, l'exemple~\ref{ex: approxi oligo pour malg} nous montre que le groupe $\Aut([0,1],\lambda)$
	des transformations préservant la mesure de l'intervalle est Roelcke-précompact \parencite{glasnerGroupAutRoelcke2012}. 
	Le groupe des transformations qui \emph{quasi}-préservent la mesure de l'intervalle et le groupe des 
	transformations préservant la mesure de $\R$ sont également Roelcke-précompacts.
	Le groupe des homéomorphismes de l'intervalle $[0,1]$ préservant son orientation est Roelcke-précompact
	car il contient une copie de $\Aut(\Q,<)$ comme sous-groupe dense.	
 	Une version continue naturelle
	 de $\Aut(R)$ est le groupe $\Iso(\mathbb U_1)$ des isométries de
	l'espace d'Urysohn borné\footnote{On renvoie au survol de \textcite{mellerayGeometricDynamicalProperties2008}
		pour des informations sur les espaces d'Urysohn en général.}, qui
	est également Roelcke-précompact. Les groupes $\mathcal O(\mathcal H)$
	et $\mathcal U(\mathcal H)$ sont Roelcke-précompacts \parencite{uspenskij1998roelcke}, ce qui peut se voir en appliquant 
	la caractérisation (iv) à leurs actions naturelles sur les boules unités
	de leurs espaces de Hilbert respectifs.
\end{exem}

\begin{rema}
    Rappelons que d'après la remarque~\ref{rmq: lc roelcke}, les seuls groupes  
	localement compacts polonais qui sont Roelcke-précompacts sont les groupes compacts polonais.
	La classe des groupes polonais Roelcke-précompacts a de nombreuses propriétés
	de stabilité analogues à celles des groupes compacts (cf.\  \cite[sec.~2.1]{tsankovUnitaryRepresentationsOligomorphic2012}), 
	à l'exception notable de la stabilité par passage à
	un sous-groupe fermé puisque tout groupe polonais est isomorphe à un sous-groupe fermé du groupe $\Iso(\mathbb U_1)$
	des isométries de l'espace d'Urysohn borné \parencite{uspenskijGroupIsometriesUrysohn1990}. 
\end{rema}

Terminons cette section en mentionnant une propriété importante des groupes polonais Roelcke-précompacts
qui fait le lien avec la théorie géométrique des groupes polonais et motive également l'étude des groupes
polonais localement Roelcke-précompacts dans ce cadre \parencite{zielinskiLocallyRoelckePrecompact2018}.

\begin{prop}[{
\cite[Thm.~6.1]{rosendalTopologicalVersionBergman2009}}]\label{prop: ob}
	Soit $G$ un groupe polonais Roelcke-précompact. Alors toute $G$-action continue
	par isométries sur un espace métrique $(X,d)$ a ses orbites bornées.
\end{prop}
\begin{proof}
	Soit $x\in X$, alors $G\acts G\cdot x$ est approximativement cocompacte, donc d'après (ii)
	du théorème précédent $G\acts G\cdot x\times G\cdot x$ est approximativement cocompacte. 
	En particulier la pseudo-distance $d^2_G$ est bornée par un certain $K>0$. 
	Prenons $g\in G$, alors en calculant $d^2_G((x,x),(x,g\cdot x))$ on obtient 
	$\inf_{h\in G} d(hg\cdot x,x)+d(h\cdot x,x)\leq K$.
	Comme tout $h\in G$ est une isométrie, on a donc 
	$\inf_{h\in G} d(g\cdot x,h\inv\cdot x)+d(x,h\inv x)\leq K$. Par inégalité triangulaire, on 
	conclut que $d(g\cdot x, x)\leq K$, et l'orbite de $x$ est donc bien bornée.
\end{proof}

\section{Autour de la propriété (T) pour les groupes polonais}\label{sec:(T)}
\subsection{Propriété (T), propriété (FH)}
Si $G$ est un groupe topologique et $\mathcal H$ un espace de Hilbert (réel) non nul, 
une \textbf{représentation orthogonale} 
de $G$ sur $\mathcal H$ est la donnée d'un morphisme de groupes \emph{continu} 
$\pi: G\to \mathcal O(\mathcal H)$,
où $\mathcal O(\mathcal H)$ est le groupe orthogonal de $\mathcal H$, 
muni de la topologie de la convergence simple\footnote{Comme $G$ agit
	par \emph{isométries} sur $\mathcal H$ via $\pi$, la continuité de $\pi$ équivaut à la continuité de l'application d'action $G\times\mathcal H\to\mathcal H$ qui à $(g,\xi)$ associe $\pi(g)\xi$.} (appelée aussi topologie
forte, et qui sur $\mathcal O(\mathcal H)$ coïncide avec la topologie faible).

\begin{defi}
	Soit $\pi:G\to\mathcal O(\mathcal H)$ une représentation orthogonale. Soit $Q\subseteq G$, on dit 
	qu'un vecteur $\xi\in\mathcal H\setminus\{0\}$ est $(Q,\epsilon)$\textbf{-invariant} si 
	\[
	\sup_{g\in Q}\norm{\pi(g)\xi-\xi}<\epsilon\norm{\xi}.
	\]	
	La représentation $\pi$ a des \textbf{vecteurs
		presque invariants} si pour toute partie compacte $K\subseteq G$ et tout $\epsilon>0$, 
	il existe un vecteur 
	$\xi\in\mathcal H\setminus\{0\}$ qui est $(K,\epsilon)$-invariant. 
	Elle a un \textbf{vecteur invariant} s'il existe $\xi\in\mathcal H\setminus\{0\}$ tel que
	$\pi(g)\xi=\xi$ pour tout $g\in G$.
\end{defi}

\begin{defi}
	Un groupe topologique a la \textbf{propriété~(T)} si toute représentation orthogonale de $G$ 
	qui a des vecteurs presque invariants admet en fait un vecteur invariant.
\end{defi}

\begin{rema}
	La propriété~(T) est souvent énoncée en parlant de représentations \emph{unitaires} sur
	un espace de Hilbert \emph{complexe}. Les deux versions sont équivalentes, ce qui se voit en utilisant 
	d'une part le fait que tout espace de Hilbert complexe est un espace de Hilbert réel
	si on le munit de la partie réelle de son produit scalaire, et d'autre part 
	la ``complexification'' d'un espace de Hilbert réel qui permet d'en faire un espace
	de Hilbert complexe (cf.\ la preuve du lemme 11 du chapitre 4 de 
	\cite{delaharpeProprieteKazhdanPour1989}). On a introduit la propriété~(T) dans le cadre
	réel parce que l'algèbre des fonctions définissables que nous fournira la théorie
	des modèles continue est une algèbre réelle.
\end{rema}

\begin{rema}
	Remarquons que si $G$ n'a pas de représentation orthogonale non
	triviale, alors $G$ a la propriété~(T). Ne possédant pas de mesure de Haar, de nombreux
	groupes polonais n'ont en fait aucune représentation orthogonale et ont donc la propriété~(T), 
	comme par exemple le groupe des homéomorphismes préservant l'orientation de l'intervalle $[0,1]$ 
	\parencite{megrelishviliEverySemitopologicalSemigroup2001}. Cependant,
	si $X$ est discret et $G$ agit sur $X$ continûment, alors on a une représentation
	orthogonale associée $G\to \mathcal O(\ell^2(X))$ donnée par: pour tout $g\in G$,
	$\xi\in\ell^2(X)$ et $x\in X$, $(\pi(g)\xi)(x)=\xi(g\inv x)$. En particulier, tous
	les groupes d'automorphismes de structures dénombrables admettent des représentations orthogonales. 
	Notons aussi que $\Aut([0,1],\lambda)$ admet une représentation orthogonale 
	sur $\mathrm L^2([0,1],\lambda)$ définie de manière similaire.
\end{rema}

Voici une caractérisation souvent utilisée	comme définition de la propriété~(T).

\begin{lemm}
	Un groupe topologique $G$ a la propriété~(T) si et seulement s'il existe $K\subseteq G$ compact et 
	$\epsilon>0$ tel que toute représentation orthogonale $\pi$ admettant des vecteurs 
	$(K,\epsilon)$-invariants admet un vecteur invariant. 
\end{lemm}
\begin{proof}
	Il est clair que la condition donnée par le lemme est suffisante. Montrons qu'elle est 
	nécessaire: supposons que $G$ a la propriété~(T), mais que pour tout $K\subseteq G$
	compact et $\epsilon>0$, $G$ admet une représentation orthogonale $\pi_{K,\epsilon}$ 
	sur un espace de Hilbert $\mathcal H_{K,\epsilon}$ avec un vecteur $(K,\epsilon)$-invariant
	mais sans vecteur invariant. 
	Alors la représentation obtenue comme somme directe de toutes les
	$\pi_{K,\epsilon}$
	a des vecteurs presque invariants, mais n'a pas de vecteurs invariants, contredisant que $G$
	a la propriété~(T). 
\end{proof}

Lorsqu'on dispose d'une partie $K\subseteq G$ et d'un $\epsilon>0$ tels que toute 
représentation ayant des vecteurs $(K,\epsilon)$-invariants admet un vecteur invariant, on dit que 
$(K,\epsilon)$ est une \textbf{paire de Kazhdan} pour $G$. Le lemme précédent nous dit que 
$G$ a la propriété~(T) si et seulement s'il
admet une paire de Kazhdan $(K,\epsilon)$ avec $K$ compact. Pour tout groupe topologique $G$, le théorème
de projection sur un convexe fermé permet de montrer que $(G,\sqrt 2)$ est une paire de Kazhdan, en particulier les groupes compacts 
ont la propriété~(T). Il est intéressant, étant donné un groupe topologique, de déterminer ses paires
de Kazhdan, par exemple pour le groupe des entiers muni de la topologie discrète 
\parencite{badeaRigiditySequencesKazhdan2019} !

Lorque $G$ admet une paire de Kazhdan
$(K,\epsilon)$ avec $K$ \emph{finie}, on dit que $G$ a la \textbf{propriété~(T) forte}. 
On renvoie à l'excellent \textcite{delaharpeProprieteKazhdanPour1989} pour des applications frappantes
de la propriété~(T) dans le cadre des groupes localement compacts.	
\\

Une propriété cousine de la propriété~(T) qui a également été étudiée pour les groupes polonais
est la propriété~(FH)\footnote{F pour fixe, H pour Hilbert.}. 

\begin{defi}
	Un groupe topologique $G$ a la propriété~(FH) si toute les $G$-actions continues affines par 
	isométries sur des espaces de Hilbert réels admettent un point fixe. 
\end{defi}

On peut montrer que la propriété~(T) implique la propriété~(FH) en toute généralité
(cf.\ \cite[Chap. 4]{delaharpeProprieteKazhdanPour1989}).
Pour les groupes localement
compacts $\sigma$-compacts, Delorme et Guichardet ont montré que la propriété~(FH)
équivaut à la propriété~(T) mais ce n'est pas le cas pour les groupes polonais généraux
\parencite{pestovAmenabilityPropertyNonlocally2018}.

\begin{rema}
Un affaiblissement naturel de la propriété~(T) pour les groupes polonais, 
introduit récemment par \textcite[Sec.~2.2]{andoLargeScaleGeometry2020}, implique également la propriété~(FH) 
et lui est en fait équivalent dans de nombreux cas.
\end{rema}

Un des points importants pour montrer la propriété~(FH) pour un groupe topologique~$G$ 
est de voir par le lemme du centre
que $G$~a la propriété~(FH) si et seulement si toute $G$-action continue par isométries affines sur un espace de Hilbert 
a ses orbites bornées (cf.\ \cite[Chap.~4, Lem.~3]{delaharpeProprieteKazhdanPour1989}).
Ainsi, la proposition~\ref{prop: ob} implique que tous les groupes polonais Roelcke-précompacts 
ont la propriété~(FH).

\subsection{Utilisation du groupe libre}\label{sec: construction action}

Si $\Gamma$ est un groupe discret, il est dit \textbf{moyennable}
lorsque sa représentation régulière sur $\ell^2(\Gamma)$ (donnée par $\pi(\gamma)\xi(g)=\xi(\gamma\inv g)$)
admet des vecteurs presque invariants. Un exemple prototypique est le groupe $\Z$,
pour lequel on vérifie qu'étant donnés une partie finie $K\subseteq \Z$ et $\epsilon>0$,
la fonction caractéristique $\chi_{\{0,\dots,n\}}$ est $(K,\epsilon)$-invariante pour $n$ suffisamment grand.

Inversement, le groupe libre à deux générateurs $\mathbb F_2$ n'est pas moyennable 
\parencite{vonneumannZurAllgemeinenTheorie1929}\footnote{Une manière
concrète de s'en rendre compte est de passer par la caractérisation de la moyennabilité
en termes d'ensembles de Følner et de faire un \emph{schéma de Ponzi} sur le groupe libre, 
cf.\ \textcite[Cor.~6.18]{gromovMetricStructuresRiemannian2007}.
}. On dispose donc
d'une partie finie $F\subseteq \mathbb F_2$ et de $\epsilon>0$ tels que la représentation
régulière $\pi: \mathbb F_2\to\mathcal O(\ell^2(\mathbb F_2))$ n'admet aucun vecteur non
nul $(F,\epsilon)$-invariant. Le théorème de Kesten donne une version quantitative de ce fait 
qui fournit les constantes du théorème~\ref{thm:main} \parencite[Thm.~3]{kestenSymmetricRandomWalks1959}. 
La version qu'on en donne est une reformulation qui se prouve
utilisant l'inégalité de Cauchy--Schwarz (cf.\ \cite[pp.~515-516]{bekkaKazhdanPropertyUnitary2003}).

\begin{theo}[Kesten]\label{thm:kesten}
	Soit $\mathbb F_2=\la a,b\ra$ le groupe libre à $2$ générateurs, soit
	$\pi: \mathbb F_2\to\mathcal O(\ell^2(\mathbb F_2))$ sa représentation régulière. 
	Alors $\pi$ n'a pas de vecteurs non nuls $(\{a,b\},\sqrt{2-\sqrt3})$-invariants: pour
	tout $\xi\in\mathcal H$, on a 
	$$\norm{\pi(a)\xi-\xi}\geq \sqrt{2-\sqrt3}\norm\xi\text{ ou }
	\norm{\pi(b)\xi-\xi}\geq \sqrt{2-\sqrt3}\norm\xi.$$
\end{theo}

Remarquons que ce théorème reste vrai pour une représentation qui serait un multiple de la représentation
régulière, c'est-à-dire un produit tensoriel de la représentation régulière avec une représentation triviale.

Bekka a utilisé le théorème de Kesten afin
de montrer entre autres que lorsque $\mathcal H$ est un
espace de Hilbert réel séparable de dimension infinie, le groupe polonais 
Roelcke-précompact $\mathcal O(\mathcal H)$ a la propriété~(T) 
\parencite{bekkaKazhdanPropertyUnitary2003}. 
Sa démonstration s'appuie sur une classification
des représentations unitaires irréductibles de $\mathcal O(\mathcal H)$, et le fait
que toutes ses représentations se décomposent comme somme directe de sous-représentations irréductibles
\footnote{Il est bien connu que toute représentation unitaire d'un groupe localement compact polonais
	se décompose en une \emph{intégrale} directe de sous-représentations irréductibles.
Ce fait est faux pour les groupes polonais
en général, par exemple le groupe polonais abélien des fonctions mesurables à valeur dans 
le cercle n'a pas de représentation unitaire irréductible 
, mais il admet beaucoup de représentations
unitaires, qui ont d'ailleurs été classifiées par 
\textcite{soleckiUnitaryRepresentationsGroups2014}.}
\parencite{olshanskiiUnitaryRepresentationsInfinitedimensional1978}. Mentionnons 
également les travaux de Shalom, qui avait auparavant montré entre autres que le groupe des applications
continues du cercle à valeurs dans $Sl_n(\mathbb C)$ a la propriété~(T) pour tout $n\geq 3$
avec une approche différente
\parencite[Cor.~4]{shalomBoundedGenerationKazhdan1999}.

\textcite{tsankovUnitaryRepresentationsOligomorphic2012}
a prouvé 
que les
représentations unitaires des groupes d'automorphismes de structures dénombrables
se décomposent toutes en somme directe de représentations irréductibles,
et a classifié ces dernières. 
Comme Bekka l'avait fait auparavant pour $\mathcal O(\mathcal H)$,
il a pu ensuite utiliser cette classification
pour identifier une large classe de groupes d'automorphismes de structures dénombrables
ayant la propriété~(T) avec des paires de Kazhdan explicites, s'appuyant sur le théorème de Kesten.
 Il demandait aussi si tout groupe polonais
Roelcke-précompact a la propriété~(T).

\textcite{evansFreeActionsFree2016} ont ensuite étendu son résultat
pour répondre par l'affirmative dans le cadre des groupes polonais 
non archimédiens. 
Le résultat
d'Ibarlucía y répond affirmativement en toute généralité, englobant les résultats
de Bekka et d'Evans--Tsankov.

\begin{theo}[
\cite{ibarluciaInfinitedimensionalPolishGroups2021}]\label{thm:truemain}
	Soit $G$ un groupe polonais Roelcke-précompact. Alors $G$ a la propriété~(T).
\end{theo}

Le résultat d'Ibarlucía est d'autant plus remarquable qu'il ne fait pas appel à une 
classification des représentations unitaires irréductibles, classification qui
est probablement hors de portée pour les groupes polonais Roelcke-précompacts en général. 
Comme expliqué plus haut, on va donner la preuve seulement dans le cas où $G=\Aut([0,1],\lambda)$,
avec à la clé un énoncé plus précis puisqu'on montre en fait la propriété~(T) forte\footnote{
	Il existe 
	des groupes polonais compacts (en particulier, Roelcke-précompacts) qui n'ont pas la propriété~(T) forte, comme 
	le cercle (cf.\ la proposition 5 de \textcite{bekkaKazhdanPropertyUnitary2003} pour un énoncé plus général). Ainsi l'énoncé du théorème~\ref{thm:truemain} est optimal.
} \parencite[Sec. 5]{ibarluciaInfinitedimensionalPolishGroups2021}. 
Notons que les représentations unitaires de ce groupe ont été classifiées par 
\textcite{neretinCategoriesBistochasticMeasures1992}, mais le fait que
tout représentation unitaire se décompose en somme directe d'irréductibles ne semble écrit nulle part.\\

Définissons maintenant les deux éléments $T_1,T_2\in\Aut([0,1],\lambda)$ qui permettront
de montrer que le groupe $\Aut([0,1],\lambda)$ a la propriété~(T) forte, et plus précisément 
qui satisferont que $(\{T_1,T_2,\},\sqrt{2-\sqrt3})$ est une paire de Kazhdan (théorème~\ref{thm:main}). 
Pour cela, comme mentionné à la fin de la section~\ref{sec: malg}, il est plus clair de travailler
dans l'algèbre de mesure de $X=\{0,1\}^{\N\times\mathbb F_2}$
muni de la mesure $(\frac 12 \delta_0+\frac 12 \delta_1)^{\otimes(\N\times\mathbb F_2)}$. 
On définit alors une action préservant la mesure de probabilité
du groupe libre $\mathbb F_2$ sur $(X,\mu)$ qui est un
\emph{décalage de Bernoulli}:
pour $x\in X$, $(n,g)\in\N\times\mathbb F_2$ et $\gamma\in\mathbb F_2$,
\[
(\gamma\cdot x)(n,g)=x(n,\gamma\inv g).
\]
Cette action induit une action de $\mathbb F_2$ par automorphismes sur l'algèbre de mesure de $(X,\mu)$,
et une fois qu'on a choisis une identification entre cette dernière et $\MAlg([0,1],\lambda)$, on obtient
les éléments $T_1$ et $T_2$ du théorème~\ref{thm:main} en prenant les transformations correspondant aux
générateurs $a$ et $b$ du groupe libre $\mathbb F_2=\la a,b\ra$.
On peut désormais ramener 
le théorème~\ref{thm:main} à l'énoncé suivant qui suit la même stratégie que Bekka.

\begin{theo}\label{thm:mainbis}
	Soit $\pi: \Aut([0,1],\lambda)\to\mathcal O(\mathcal H)$ une représentation orthogonale sans 
	vecteurs invariants. Alors $\pi_{\restriction\la T_1,T_2\ra}$
	est un multiple  de la représentation régulière du groupe
	libre engendré par $T_1$ et $T_2$.
\end{theo}
\begin{proof}[Démonstration du théorème~\ref{thm:main} à partir du théorème~\ref{thm:mainbis}] 
	La preuve se fait par contraposition.
	Soit $\pi: \Aut([0,1],\lambda)\to\mathcal O(\mathcal H)$ une représentation
	orthogonale sans vecteurs invariants. D'après le théorème~\ref{thm:mainbis} elle se retreint à $\mathbb F_2=\la T_1,T_2\ra$ 
	en un multiple de sa représentation régulière	
	et le théorème de Kesten empêche alors l'existence de vecteurs 
	$(\{T_1,T_2\},\sqrt{2-\sqrt 3})$-invariants.
\end{proof}

Nous ébaucherons la preuve du résultat ci-dessus après avoir donné
une preuve complète du théorème plus faible qui suit (cf.\ section~\ref{sec:pfautxmu}). Les 
sections~\ref{sec:bases} et~\ref{sec: definissable} prépareront le terrain.

\begin{theo}\label{THM:MAINTER}
	Soit $\pi: \Aut([0,1],\lambda)\to\mathcal O(\mathcal H)$ une représentation orthogonale sans 
	vecteurs invariants. Alors $\pi_{\restriction\la T_1,T_2\ra}$
	\emph{contient} la représentation régulière du groupe
	libre engendré par $T_1$ et $T_2$.
\end{theo}

Terminons cette partie en mentionnant que puisque les groupes polonais n'ont pas 
de mesure de Haar a priori, ils n'ont pas de représentation régulière et la définition
de la moyennabilité qu'on a donné en début de partie n'a pas de sens dans leur cadre. 
Une bonne généralisation existe cependant, et les groupes polonais peuvent de plus exhiber
un phénomène d'\emph{extrême moyennabilité} inexistant dans le cadre localement compact.
On renvoie à l'ouvrage de 
\textcite{pestovDynamicsinfinitedimensionalgroups2006}
pour ces aspects fascinants.

\section{Bases de théorie des modèles continue}\label{sec:bases}

Nous allons maintenant introduire le cadre de la théorie des modèles continue,
qui s'avère être un outil puissant  non seulement 
pour l'étude des groupes d'automorphismes de structures métriques,
mais aussi pour l'étude de ces structures elles-mêmes (cf., par exemple, \cite{farahLogicOperatorAlgebras2014}).
Notre présentation est très partielle puisque notre objectif est d'arriver rapidement aux résultats 
d'Ibarlucía sur la propriété~(T) pour $\Aut([0,1],\lambda)$.
On omet notamment le théorème de Löwenheim--Skolem, la construction d'ultraproduits, le
 théorème de compacité ou encore l'existence de modèles monstrueux
qui sont incontournables pour aller plus loin et comprendre la preuve d'Ibarlucía en général. 
On renvoie pour cela à \textcite{benyaacovModeltheorymetric2008} qui est notre principale source. 
On a suivi un traitement proche de \textcite{farahModeltheorymathrmC2016} et  \textcite{hallbackMetricModelTheory2020} dans la présentation de 
l'algèbre des prédicats définissables.

\subsection{Langages, formules, théories}\label{sec: basics}

Étant donné le groupe des automorphismes d'une structure métrique, on va avoir besoin de parler de sous-structures,
mais aussi de plongements entre des structures non nécessairement liées à notre structure de départ.
Évidemment, on ne peut pas considérer des plongements entre des objets de nature différente, et la notion
de \emph{langage} permet de formaliser cette restriction en fixant une fois pour toute
 des \emph{symboles} de fonctions et 
de relations, qui seront ensuite \emph{interprétés} dans les structures appropriées.
On commence par donner la définition d'un langage dans le cadre discret.

\begin{defi}
	Un \textbf{langage discret} est un couple $\mathcal L=((f_i,n_i)_{i\in I},(R_j,m_j)_{j\in J})$, où
	les $f_i$ et $R_j$ sont des symboles deux à deux distincts, et les $n_i$ et $m_j$ sont des entiers
	représentant leurs arités respectives.
	Les $f_i$ sont appelés \textbf{symboles de fonctions}, les~$R_j$ sont les
	\textbf{symboles de relations}. Quand $f_i$ est d'arité $n_i=0$, on dit aussi que $f_i$ est un
	\textbf{symbole de constante}.
\end{defi}

%
La définition dans le cas métrique est compliquée par le fait que l'on voudra que
les interprétations des symboles du langage soient uniformément continues bornées,
et ce d'une manière qui ne dépende pas de la structure considérée.
Rappelons qu'étant donnés 
deux espaces métriques $(X,d_X)$ et $(Y,d_Y)$ et une application 
$\delta\colon\mathopen]0,+\infty\mathclose[\to \mathopen]0,+\infty\mathclose[$, une application $f:X\to Y$ est  dite uniformément continue
de \textbf{module de continuité uniforme} $\delta$ si pour tout $\epsilon>0$ et tous $x_1,x_2\in X$,
si $d_X(x_1,x_2)<\delta(\epsilon)$ alors $d_Y(f(x_1),f(x_2))<\epsilon$.


Dans la définition qui suit, les distances sur les puissances $M^n$ sont celles de
la convention donnée en début de section~\ref{sec:cara prp}, et on munit $\R$ de sa distance usuelle,
induite par la valeur absolue.

\begin{defi}\label{def: Lstructure}
	Un \textbf{langage métrique} est un langage discret \[\mathcal
          L=((f_i,n_i)_{i\in I}, (R_j,m_j)_{j\in J})\] 
	où on se donne en plus 
	\begin{itemize}
		\item pour chaque $i\in I$ un module de continuité uniforme $\delta_i$;
		\item pour chaque $j\in J$ un module de continuité uniforme $\delta_j$ et une \emph{borne} $D_j>0$;
		\item une borne $D>0$. 
	\end{itemize}	
	Une \textbf{$\mathcal L$-structure}
	est alors un espace métrique \emph{complet} $(M,d^M)$ de diamètre au plus $D$ muni
	des \emph{interprétations} suivantes:
	\begin{itemize}
		\item pour chaque symbole de fonction $f_i$, on a une application $f_i^M: M^{n_i}\to M$ 
		de module de continuité uniforme
		$\delta_i$;
		\item pour chaque symbole de relation $R_j$, on a une application $R_j^M: M^{m_j}\to [0,D_j]$ 
		de module de continuité
		uniforme $\delta_j$.
	\end{itemize}
\end{defi}

Comme dans la définition des structures métriques, 
on autorise les symboles de fonctions $0$-aires qui correspondent alors à 
des éléments de la structure $M$. Le symbole de distance $d$
est également interprété comme $d^M$ dans $M$.

\begin{exem}
	Tout espace métrique complet est une structure dont le langage est l'ensemble vide
	(on parle alors du langage des espaces métriques).
\end{exem}

\begin{exem}
	Les algèbres de mesure telles que définies dans la section~\ref{sec: malg} sont des 
	$\mathcal L$-structures pour le langage $\mathcal L$ qui suit. 
	Tout d'abord, on prendra pour chaque symbole de relation
	ou de fonction le module de continuité uniforme $\delta=\id_{]0,+\infty[}$ (module de continuité uniforme
	des fonctions $1$-lipschitziennes), et la borne $D=1$. 
	Le langage est alors constitué  des symboles  suivants
	dont on précise les interprétations dans une algèbre de mesure 
	$(\MAlg(X,\mu),d_\mu,\emptyset, X,\cup,\cap, \bigtriangleup,\mu)$ :
	\begin{itemize}
		\item deux symboles de fonctions $\mathds{O}$ et $\mathds 1$ d'arité $0$ (constantes), 
		qui seront interprétées comme $\emptyset$ et $X$ ;
		\item trois symboles de fonctions $\vee, \wedge$ et $\fplus$ d'arité $2$ et de module de continuité
		uniforme $\delta$, 
		qui seront interprétées comme $\cup, \cap$ et $\bigtriangleup$ respectivement ;
		\item un symbole de relation $m$ d'arité $1$, de module de continuité uniforme $\delta$, 
		qui sera interprété comme $\mu$.
	\end{itemize}	
Le langage des algèbres de mesure est donc formellement la famille
\[\big(
\left((\mathds O,0),(\mathds1,0),(\vee,2),(\wedge,2),(\fplus,2)\right),
\; \left((m,1))\right)\big),
\]
avec les modules de continuité uniforme tous égaux à $\id_{]0,+\infty[}$ et les bornes toutes égales à $1$.
\end{exem}

On suppose également avoir à disposition un ensemble infini de \textbf{symboles de variables} 
$\{x_1,\ldots,x_k,\ldots\}$ et 
que pour chaque tel symbole de variable $x_k$ on dispose des \textbf{symboles de quantification}
$\inf_{x_k}$ 
et $\sup_{x_k}$.

Un langage métrique $\mathcal L$ étant fixé, 
nous allons maintenant définir les briques de base des \emph{formules}. Formellement,
une formule est une suite finie de symboles, mais qui est destinée à avoir un sens (une interprétation)
 dans les structures que nous considérerons.
 La définition que nous donnons des termes et formules est imprécise afin de 
 ne pas nous noyer dans des notations, et on peut rendre tout cela plus rigoureux,
 de la même manière que l'on peut définir proprement les polynômes sur un corps
 $\mathbb K$ comme étant des suites de support fini à valeurs dans $\mathbb K$.
On renvoie donc le lecteur inquiet à \textcite[Chap.~3]{coriLogiqueMathematiqueTome2003}
pour un traitement complet dans le cadre de la théorie des modèles classique.

\begin{defi}
On définit les $\mathcal L$-\textbf{termes} (et leurs \textbf{ensembles de variables libres})
 par induction de la manière suivante:
\begin{enumerate}
	\item Chaque symbole de variables $x_k$ est un terme dont l'ensemble des variables libres
	est $\{x_k\}$.
	\item Si $f$ est un symbole de fonction d'arité $n$ et $t_1,\ldots,t_n$ sont des termes,
	alors $f(t_1,\ldots,t_n)$ est un terme dont l'ensemble des variables libres est la réunion
	des ensembles de variables libres de $t_1,\ldots, t_n$.
\end{enumerate}
\end{defi}

En particulier pour $n=0$, les symboles de constantes sont des termes dont l'ensemble de variables 
libres est vide.

\begin{exem}
	Dans le langage des algèbres de mesure, $\mathds 1+ x_1$ est un terme, que l'on interprétera
	bientôt comme le passage au complémentaire.  Un autre exemple est donné par $x_1\vee x_2$, 
	qu'on interprétera comme l'opération de réunion.
\end{exem}

\begin{defi}\label{def: formule atomique}
	Une \textbf{formule atomique} est une expression de la forme
	\[
	R(t_1,\ldots,t_m),
	\]
	où $t_1,\ldots,t_m$ sont des $\mathcal L$-termes et $R$ est un symbole de relation d'arité $m$.
	Son ensemble de variables libres est la réunion des ensembles de variables libres des termes
	qui la constituent.
\end{defi}

\begin{exem}
	Dans le langage des espaces métriques, les seules formules atomiques sont de la forme $d(x_i,x_j)$.
	Dans le langage des algèbres de mesure, $m((x_1\vee x_2)\wedge(x_3\vee x_4\vee \mathds O))$
	est une formule atomique.
\end{exem}

Un \textbf{connecteur logique} est une fonction continue $c:\R^n\to \R$. 
Ici encore, on s'autorise $n=0$,
auquel cas $c$ est un réel. On note $\1$ le connecteur logique d'arité nulle
correspondant au réel $1$ et 
$\0$ le connecteur logique correspondant au réel $0$. On notera également
$\dotdiv:\R^2\to \R$ la fonction qui à $(t_1,t_2)\in\R^2$ associe $t_1\dotdiv t_2\coloneq\max(t_1-t_2,0)$.

\begin{defi}\label{df: formules}
L'ensemble des \textbf{formules} (et les ensembles de variables libres associés) 
est défini par induction de la manière suivante:
\begin{enumerate}
	\item Toute formule atomique est une formule;
	\item Si $\varphi_1,\dots,\varphi_n$ sont des formules et 
	$c:\R^n\to \R$ est un connecteur logique alors $c(\varphi_1,\dots,\varphi_n)$ est une formule 
	dont l'ensemble des variables libres est égal à la réunion des ensembles de variables 
	libres de $\varphi_1,\dots,\varphi_n$;
	\item Si $\varphi$ est une formule et $x_k$ est un symbole de variable alors $\inf_{x_k}\varphi$
	et $\sup_{x_k}\varphi$ sont des formules, dont l'ensemble des variables libres
	est celui de $\varphi$ privé de $\{x_k\}$.
\end{enumerate}
\end{defi}

\begin{exem}	
	La formule $\inf_{x_2} m(x_1\wedge  x_2)$ est une formule 
	dans le langage des algèbres de mesure dont la seule variable libre est $x_1$.
\end{exem}

Un \textbf{énoncé} est une formule dont l'ensemble des variables libres est vide.
On notera $\varphi(x_1,\ldots,x_k)$ une formule dont l'ensemble des variables libres est inclus dans 
$\{x_1,\ldots,x_k\}$.

Étant donnés une structure métrique $M$ de signature $\mathcal L$, une formule $\varphi(x_1,\ldots, x_k)$,
et un uplet $(a_1,\ldots,a_k)\in M^k$, on peut enfin définir l'\textbf{interprétation} $\varphi^M$ 
de $\varphi(x_1,\ldots,x_k)$ dans $M$ évaluée en $(a_1,\dots,a_k)$, 
qui est un nombre réel que l'on notera $\varphi^M(a_1,\ldots, a_k)$.
Cette dernière est obtenue en remplaçant $x_i$ par $a_i$ dans
la formule, chaque symbole de fonction ou relation par son interprétation, en interprétant chaque
connecteur logique comme la fonction $\R^n\to\R$ correspondante, et les symboles $\inf_{x_i}$ et 
$\sup_{x_i}$ comme des infimums/supremums portant sur $x_i\in M$.
On obtient alors une fonction $\varphi^M: M^k\to \R$.
Pour une définition formelle de l'interprétation d'une formule, on renvoie à 
\textcite[Def.~3.3]{benyaacovModeltheorymetric2008}, ici on se contente de donner quelques exemples dans le
cadre des algèbres de mesure.

\begin{exem}
	Soit $(M,d^M)$ un espace métrique complet borné. 
	
	Alors $\varphi(x_1,x_2)\coloneq d(x_1,x_2)$
	est une formule à deux variables libres dont l'interprétation dans $M$ est la fonction $d^M$.
	Ainsi l'énoncé
	$$\sup_{x_1}\sup_{x_2}d(x_1,x_2)$$
	s'interprète comme le supremum des $d^M(a_1,a_2)$ où $(a_1,a_2)\in M^2$, c'est-à-dire
	le \emph{diamètre} de $(M,d^M)$.
	Un autre exemple est fourni par la formule $d(x_1,x_2)+\cdots +d(x_{2n+1},x_{2n+2})$ qui s'interprète 
	comme la distance produit 
	sur $M^n$ suivant la convention donnée au début de la section~\ref{sec: prp}.
\end{exem}

\begin{exem}
	Soit $M=\MAlg(X,\mu)$ une algèbre de mesure. La formule 
	$\varphi_1(x_1)\coloneq\inf_{x_2} m(x_1\wedge  x_2)$ s'interprète
	dans $\MAlg(X,\mu)$ comme la fonction $\varphi_1^M:\MAlg(X,\mu)\to \R$ qui à $A\in\MAlg(X,\mu)$
	associe 
	$$\inf_{B\in\MAlg(X,\mu)}\mu(A\cap B).$$
	Notons que $\varphi^M$ est constante égale à $0$, puisque on peut prendre pour $B$ l'ensemble vide.
	On va utiliser des connecteurs logiques pour la rendre plus intéressante. 
	Considérons la formule suivante:
	$$\varphi_2(x_1)\coloneq\inf_{x_2} \abs{m(x_1\wedge x_2)-\frac{m(x_1)}2},$$
	alors $\varphi_2^M(A)=0$ si et seulement si on peut trouver des sous-ensembles de $A$ de mesure arbitrairement
	proche de $\frac{\mu(A)}2$.
	On peut alors donner un énoncé dont l'interprétation est nulle si et seulement si $\MAlg(X,\mu)$ est sans atomes:
	$$\sup_{x_1}\inf_{x_2} \abs{m(x_1\wedge x_2)-\frac{m(x_1)}2}.$$
\end{exem}

Nous avons maintenant à notre disposition, pour chaque $\mathcal L$-structure $M$,
	d'un ensemble de fonctions uniformément continues bornées sur $M$ et ses puissances,
	qui sont les interprétations $\varphi^M$ de formules $\varphi$. 
De plus, en utilisant les bornes et modules de continuité uniforme fournis par la définition~\ref{def: Lstructure} et le fait que la restriction de tout connecteur logique à un compact est
uniformément continue, on peut trouver pour chaque formule 
 $\varphi$ une borne ainsi qu'un module d'uniforme continuité valables
  dans \emph{toutes} les interprétations de $\varphi$ (cf.\ \cite[Thm.~3.5]{benyaacovModeltheorymetric2008}).


Identifions deux formules lorsque leurs interprétations sont les mêmes
dans toute $\mathcal L$-structure (on dit aussi qu'elles sont logiquement équivalentes).
On munit alors l'ensemble~$\mathfrak F^{\mathcal L}$ des formules
 d'une structure de $\R$-algèbre commutative en utilisant les connecteurs
logiques correspondants à la multiplication par un réel, à l'addition et à la multiplication. 
On note $\mathfrak F^{\mathcal L}_n$ la sous-algèbre des formules dont les variables libres sont 
parmi $\{x_1,\dots,x_n\}$. Pour $n=0$, on obtient l'algèbre des \emph{énoncés} $\varphi$ (formules sans
variables libres), qui ont donc une interprétation $\varphi^M\in\R$ dans chaque $\mathcal L$-structure~$M$.

\begin{defi}
La \textbf{théorie} d'une $\mathcal L$-structure $M$ est l'application linéaire qui à chaque énoncé
$\varphi\in \mathfrak F^{\mathcal L}_0$ associe $\varphi^M$.
\end{defi}

Comme une application linéaire est complètement déterminée par son noyau, on appellera en fait théorie
de $M$, notée $\mathrm{Th}(M)$, le noyau de l'application $\varphi\mapsto \varphi^M$.
De manière générale, une \textbf{théorie} est un ensemble $T$ formé d'énoncés. 
On dit que $M$ est un \textbf{modèle} 
d'une théorie $T$ si $T\subseteq \mathrm{Th}(M)$. 
Les théories admettant des modèles sont dites \textbf{consistantes},
et si de plus tous les modèles de $T$ ont la même théorie 
(interprètent identiquement \emph{tous} les énoncés), 
on dit que $T$ est \textbf{complète}.
Dans ce qui suit, on s'intéressera uniquement à des théories complètes. Notons que par définition,
si on se fixe une $\mathcal L$-structure $M$, la théorie $\mathrm{Th}(M)$ est complète. 

\begin{rema}Un enjeu 
important (mais secondaire dans le cadre de cet exposé) est de trouver une \textbf{axiomatisation}
naturelle de $\mathrm{Th}(M)$, c'est-à-dire une théorie complète \emph{la plus petite} 
(et la plus naturelle) possible au sein de $\mathrm{Th}(M)$. 
On pourra par exemple consulter \textcite[Sec.~16]{benyaacovModeltheorymetric2008} pour 
une axiomatisation de la théorie de l'algèbre de mesure $\MAlg([0,1],\lambda)$, qui est en fait 
égale à celle 
de toute algèbre de mesure sans atomes.
\end{rema}

\subsection{Prédicats définissables et espace des types}

Fixons désormais une théorie $T$ complète.
Si $M$ est un modèle de $T$, on peut alors
définir une semi-norme sur l'algèbre des formules en considérant la norme infinie
de leurs interprétations dans $M$:
$$
\norm{\varphi(x_1,\ldots,x_k)}_{M,T}=
\sup_{(a_1,\dots,a_k)\in M^k}\abs{\varphi^M(a_1,\ldots,a_k)}.
$$
Notons au passage l'importance d'avoir dans la définition~\ref{def: Lstructure} des bornes 
sur les symboles de relation afin que cette expression soit bien finie.
Le lemme suivant dit que la  semi-norme obtenue ne dépend pas du modèle de $T$ choisi. 
\begin{lemm}\label{lem: norme independante de M}
	Si $M$ et $N$ sont deux modèles de $T$ et si $\varphi(x_1,\ldots,x_k)$ est une formule,
	alors
	$$\norm{\varphi(x_1,\ldots,x_k)}_{M,T}=\norm{\varphi(x_1,\ldots,x_k)}_{N,T}.$$
\end{lemm}
\begin{proof}
	Considérons l'énoncé $\sup_{x_1}\cdots\sup_{x_k} \abs{\varphi(x_1,\ldots,x_k)}$. 
	Par définition, son interprétation dans $M$ (resp.\ $N$) est égale à
	$\norm{\varphi(x_1,\ldots,x_k)}_{M,T}$ (resp.\ $\norm{\varphi(x_1,\ldots,x_k)}_{N,T}$),
	or comme $M$ et $N$ sont de modèles de $T$ complète ces interprétations doivent être égales.
\end{proof}

On peut donc définir sans ambiguité une semi-norme $\norm{\cdot}_T$ sur l'algèbre des formules
en posant $\norm{\varphi(x_1,\ldots,x_k)}_T=\norm{\varphi(x_1,\ldots,x_k)}_{M,T}$ où $M$ est un
modèle de $T$.

Autrement dit, si $T$ est une théorie complète, on vient 
de munir $\mathfrak F^{\mathcal L}$ d'une semi-norme de sorte que l'application qui à 
$\varphi(x_1,\ldots,x_n)\in\mathfrak F^{\mathcal L}_n$ 
associe son interprétation sur $M^n$ soit un plongement isométrique\footnote{Comme on a seulement 
	une semi-norme sur l'espace de départ, ce n'est pas une injection donc le terme de
	plongement est peut-être abusif, mais on va tout de suite quotienter par l'idéal
	des éléments de semi-norme nulle de manière à avoir une vraie norme, et donc un vrai plongement
	par passage au quotient.}
$(\mathfrak F^{\mathcal L}_n,\norm\cdot_T)\to (UCB(M^n),\norm{\cdot}_\infty)$ (où $UCB(M^n)$ est
l'algèbre des fonctions uniformément continues bornées sur $M$), et ce quel que soit le
modèle $M$ de $T$.

\begin{rema}\label{rmk:unif continuite}
	Si $\varphi$ est une formule, on peut montrer que son module de continuité uniforme est 
	également encodé dans la théorie~$T$ complète. Dans le cas où $\varphi(x_1)$ n'a qu'une 
	variable libre on procède ainsi: à $\epsilon,\delta>0$ fixé, dire que dans un modèle~$M$ de~$T$, 
	on a pour tous $(a_1,a_2)\in M^2$ l'implication
	$d(a_1,a_2)\leq \delta\Rightarrow\abs{\varphi^M(a_1)-\varphi^M(a_2)}\leq \epsilon$ équivaut à dire 
	que $M$~satisfait l'énoncé suivant, où l'on rappelle que le connecteur logique binaire $\dotdiv$ est défini
	par $t_1\dotdiv t_2=\max(t_1-t_2,0)$:
	\[
	\sup_{x_1,x_2} \min\big(\delta\dotdiv d(x_1,x_2), \abs{\varphi(x_1)-\varphi(x_2)}\dotdiv\epsilon\big),
	\]
	ce qui permet de conclure par complétude de $T$. On aurait donc pu se passer des modules d'uniforme
	continuité dans la définition~\ref{def: Lstructure}. Cependant, cela poserait problème 
	au sein d'une théorie incomplète 
	(comme par exemple dans \cite{farahModeltheorymathrmC2016})
	puisque l'absence de module d'uniforme continuité commun à toutes les structures 
	empêche la construction
	d'ultraproduits raisonnables.
	On a donc préféré donner
	la définition générale.
\end{rema}

\begin{rema}
	De même, la théorie encode quelles formules sont à valeurs positives.
	Par exemple, une formule $\varphi(x_1)$ est à valeurs positives si et seulement si l'énoncé suivant est dans $T$:
	 \(\inf_{x_1}\min(\varphi(x_1),\0)\). 
\end{rema}

\begin{defi}
	Étant donnée une théorie $T$ complète, l'algèbre
	$\mathfrak P^{\mathcal L,T}$
	 des \textbf{prédicats définissables}
	est la $\R$-algèbre de Banach obtenue par séparation puis complétion
	de $(\mathfrak F^{\mathcal L},\norm{\cdot}_T)$.
\end{defi}

\begin{rema}
En toute rigueur, on devrait parler de l'algèbre des prédicats $T$-définissables puisque la notion
dépend de la théorie complète $T$ sous-jacente.
\end{rema}

\begin{conv}
	Dans ce qui suit, on aura affaire à des uplets finis et infinis (mais dénombrables),
	et on dira qu'on a des $n$-uplets pour un $n\leq\omega$ ; 
	plus précisément $n<\omega$ veut dire que $n\in\N$,
	et $n=\omega$ veut dire que l'on travaille sur des suites  indexées par $\N$.
\end{conv}

Pour $n\leq \omega$, définissons $\mathfrak P_n^{\mathcal L,T}$ comme l'adhérence dans
$\mathfrak P^{\mathcal L,T}$ de $\mathfrak F^{\mathcal L}_n$. 
C'est par définition la sous-algèbre des prédicats 
définissables dont les variables libres sont parmi $\{x_1,\ldots, x_n\}$, pour $n=\omega$ on
a $\mathfrak P_\omega^{\mathcal L,T}=\mathfrak P^{\mathcal L,T}$.
Notons qu'il existe des prédicats définissables \emph{avec une infinité de variables libres},
par exemple $\sum_{n\geq 0} 2^{-n}d(x_{2n+1},x_{2n+2})$ (que l'on définit formellement 
comme étant la limite de suite de Cauchy des formules données par les sommes partielles). Un prédicat
définissable sans variable libre (cas $n=0$) sera appelé un \textbf{énoncé}, ce qui étend naturellement
la définition donnée en section~\ref{sec: basics}.

On a quotienté $\mathfrak F^{\mathcal L}$ par l'idéal $\mathfrak I$ 
des formules dont les
interprétations sont toujours nulles, ainsi les éléments de la pré-algèbre de Banach 
$\mathfrak F^{\mathcal L}_n$ ont toujours une interprétation bien définie dans chaque modèle de $T$, 
qui définit une application uniformément continue et bornée sur $M^n$.
Comme tout élément de $\mathfrak P^{\mathcal L}_n$ est limite uniforme d'éléments de 
$\mathfrak F^{\mathcal L}_n/\mathfrak I$, tout prédicat définissable $\varphi\in\mathfrak P^{\mathcal L}_n$ a 
une interprétation bien définie $\varphi^M$ dans
tout modèle~$M$ de~$T$, qui est une fonction uniformément continue sur $M^n$.

Ceci nous permet en particulier de voir que si $M$ est un modèle de $T$, et $n<\omega$
chaque élément $(a_1,\ldots, a_n)\in M^n$ définit un morphisme de $\R$-algèbre non nul
noté $\tp(a_1,\dots,a_n): \mathfrak P^{\mathcal L}_n\to \R$, défini par 
$$\tp(a_1,\dots,a_n)(\varphi)=\varphi^M(a_1,\ldots, a_n).$$
De même pour $n=\omega$ chaque suite $(a_i)_{i\in\N}$ définit un morphisme 
$\mathfrak P^{\mathcal L}_\omega\to \R$ donné par
$$\tp((a_i)_{i\in\N})(\varphi)=\varphi^M((a_i)_{i\in\N}).$$

\begin{defi}
	Soit $T$ une théorie complète et $n\leq\omega$. L'espace des \textbf{$n$-types} sur~$T$ 
	est l'espace des morphismes
	multiplicatifs non nuls $p:\mathfrak P^{\mathcal L}_n\to \R$ continus de norme~\mbox{$\leq 1$}.
\end{defi}

Notons que par densité et continuité, tout type est complètement déterminé par sa restriction aux formules.

\begin{rema}
	En fait tout morphisme $\mathfrak P^{\mathcal L,T}_n\to \R$ 
	est automatiquement continu de norme $\leq 1$ car il préserve l'ordre naturel sur
	$\mathfrak P^{\mathcal L,T}_n$ et envoie $\1$ sur $1$, et tout prédicat définissable $\varphi$
	satisfait l'inégalité $\varphi\leq \norm{\varphi}_T\1$.
\end{rema}

On munit l'espace des types de la topologie faible-*, ce qui en fait un espace topologique compact
(cette topologie est appelée la \textbf{topologie logique} dans \cite{benyaacovModeltheorymetric2008}).
Étant donné un type $p$, on notera souvent $\varphi(p)$ l'élément $p(\varphi)$ afin d'avoir une 
notation cohérente avec la définition du type d'un élément de $M^n$.

\begin{defi} Pour un type $p\in S_n(T)$, un modèle $M$ de $T$ et un uplet $\bar a\in M^n$,
	on dit que $\bar a$ \textbf{réalise} le type $p$ lorsque $\tp(\bar a)=p$. 
\end{defi}
\begin{rema}
	Le théorème de compacité garantit que pour tout $n$-type $p\in S_n(T)$, 
	on peut trouver un modèle de $T$ dans lequel
	il existe une réalisation de $p$.
\end{rema}

Comme pour les théories, on identifiera 
 un type et son noyau, donc on dira qu'une formule \emph{appartient à} $p$ ou \emph{est dans} $p$ 
si elle appartient à son noyau.

\begin{exem}
	Dans l'algèbre de mesure $\MAlg([0,1],\lambda)$, si $\varphi(x_1,x_2)\coloneq m(x_1\wedge x_2)$,
	et $p=\tp([0,\frac 13],[\frac 14,1])$, alors $p(\varphi)=\mu([\frac 14,\frac 13])=\frac 1{12}$.
	Ainsi la formule $\varphi-\frac 1{12}$ appartient à~$p$.
\end{exem}

\begin{rema}\label{rmq: base topo logique}
	Il est utile de noter qu'une \emph{base} de la topologie logique est donnnée par les ouverts
	de la forme 
	$U_{\varphi,\epsilon}\coloneq\{p: \abs{\varphi(p)}<\epsilon\}$, 
	puisque une intersection finie d'ouverts de la forme $U_{\varphi_i,\epsilon_i}$
	contient l'ouvert $U_{\max_i\varphi_i,\min_i\epsilon_i}$.
\end{rema}

Chaque prédicat définissable $\varphi\in \mathfrak P^{\mathcal L,T}_n$ définit une fonction continue 
$\widehat \varphi$ sur l'espace des types $S_n(T)$
par $\widehat \varphi(p)=p(\varphi)$.
La première partie du théorème qui suit est un cas très particulier d'un théorème de Arens 
(le corollaire du théorème 6 dans
\cite{arensRepresentationAlgebras1947}).

\begin{theo}[Dualité de Gelfand]\label{thm:dualité}
	Pour tout $n\leq\omega$, l'application $\varphi\mapsto \widehat \varphi$ est un isomorphisme de 
	$\R$-algèbres de Banach entre 
	$\mathfrak P^{\mathcal L,T}_n$
	et l'espace $\mathcal C(S_n(T),\R)$ des fonctions continues sur $S_n(T)$ à valeurs réelles.
	De plus, si $n<\omega$, pour tout modèle $M$ de $T$, l'application 
	$(a_1,\dots,a_n)\in M^n\mapsto \tp(a_1,\dots,a_n)$
	est uniformément continue d'image dense dans $S_n(T)$.
	De même l'application $(a_i)_{i\in\N}\in M^\omega\mapsto \tp((a_i)_{i\in\N})$ est uniformément continue
	d'image dense dans $S_n(T)$.
\end{theo}
\begin{proof}
	Supposons $n<\omega$.
	L'application $\varphi\mapsto \widehat \varphi$ est une contraction puisque les éléments de $S_n(T)$ 
	sont de norme $\leq 1$.
	Elle est en fait isométrique puisque si
	$\varphi$ est une formule, 
	\[
	\norm{\varphi}_T=\sup_{(a_1,\dots,a_n)\in M^n}\abs{\varphi^M(a_1,\dots,a_n)}
	=\sup_{(a_1,\dots,a_n)\in M^n}\widehat\varphi(\tp(a_1,\dots,a_n))\leq\norm{\widehat\varphi}.
	\]
	La surjectivité de $\varphi\mapsto \widehat \varphi$ découle du théorème de Stone--Weierstrass, puisque 
	par définition l'algèbre formée par les $\widehat \varphi$ sépare les points de 
	l'espace compact $S_n(T)$.
	
	Enfin, le fait que l'application  $(a_1,\dots,a_n)\in M^n\mapsto \tp(a_1,\dots,a_n)$
	soit d'image dense provient du fait que toute fonction continue $\widehat \varphi\in \mathcal C(S_n(T),\R)$ est complètement
	déterminée par sa restriction à $\tp(M^n)$, qui n'est d'autre que la restriction de $\varphi$ à $M^n$, 
	et ce d'après le lemme~\ref{lem: norme independante de M}. Son uniforme continuité
	découle du fait que chaque $\varphi^M$ est uniformément continu.
	Pour $n=\omega$ la preuve est identique, il faut juste adapter la notation.
\end{proof}

\begin{rema}\label{rmq: identification prédicats}
	Pour $n\leq \omega$, définissons $UC(M^n)$ comme l'algèbre des fonctions uniformément
	continues bornées $M^n\to\R$, muni de la norme $\norm{\cdot}_{\infty}$.
	La fin de la preuve utilise que pour $n\leq \omega$, 
	l'application $\mathfrak P_n^{\mathcal L,T}\to UC(M^n)$ qui associe à un
	prédicat définissable son interprétation dans $M$ modèle de $T$ est un une plongement isométrique
	(ce qui découle directement de  la définition de $\norm\cdot_T$ et du 
	lemme~\ref{lem: norme independante de M}). 
	\emph{Il pourra ainsi être plus 
	commode de voir l'algèbre des prédicats définissables comme une algèbre de fonctions sur $M^\omega$
	où $M$ est un modèle fixé de $T$, ce que l'on s'autorisera à faire implicitement par la suite.}
\end{rema}

\begin{rema}
	Pour toute partie compacte $K$ de $\R^n$, 
	les fonctions polynomiales à coefficients rationnels sont denses dans 
	l'espace des fonctions continues sur $K$ à valeurs réelles.
	Il en découle que dans la définition des connecteurs logiques, 
	on aurait pu se 
	contenter de prendre pour connecteurs les 
	fonctions $\0$ et $\1$ (connecteur $0$-aire), 
	les multiplications par des scalaires $q\in\Q$ (connnecteurs unaires)
	la multiplication et l'addition sur les réels (connecteurs binaires),
	et les projections $\R^n\to\R$ (connecteurs $n$-aires).
	On obtient ainsi naturellement, sous l'hypothèse que $\mathcal L$ soit dénombrable,
	 une $\Q$-algèbre dénombrable de formules qui nous donne, après séparation/complétion,
	 la même $\R$-algèbre de prédicats définissables.
	 En particulier, sous l'hypothèse que $\mathcal L$ est dénombrable, $\mathfrak{P}^{\mathcal L, T}$
	 est une algèbre de Banach séparable, et donc pour tout $n\leq \omega$, l'espace $S_n(T)$
	 est compact \emph{métrisable}.
\end{rema}

Terminons cette section en étendant un certain nombre d'opérations sur les formules 
aux prédicats définissables. 
Tout d'abord, les quantificateurs $\sup_{x_k}$ et $\inf_{x_k}$ induisent
chacun une contraction sur l'espace 
des formules $(\mathfrak F^{\mathcal L},\norm{\cdot}_T)$. Ces dernières s'étendent donc de manière unique 
en des contractions
sur $\mathfrak P^{\mathcal L,T}$. En particulier si $\varphi(x_1,\ldots,x_k,\ldots)$ est un prédicat
définissable, $\inf_{x_k} \varphi{{}}(x_1,\ldots,x_k,\ldots)$ l'est aussi,
et son interprétation dans un modèle $M$ de $T$
satisfera, pour tous $(a_i)_{i\in\omega}\in M^\omega$:
\[
\big(\inf_{x_k}\varphi(x_1,\ldots,x_k,\ldots)\big)^M(a_1,\ldots,a_{k-1},a_{k+1},\ldots)=
\inf_{a_k\in M} \varphi^M(a_1,\ldots,a_k,\ldots).
\]
On peut également prendre un infimum de n'importe quel prédicat définissable  
sur \emph{un sous-ensemble infini de l'ensemble des variables}
puisqu'une telle opération nous définit une contraction de norme $1$ sur l'ensemble des formules
$(\mathfrak F^{\mathcal L},\norm{\cdot}_T)$, et de même pour le supremum. 
Si $(x_{i_1},\ldots,x_{i_k},\ldots)$ est l'uplet
(possiblement infini) des variables sur lesquelles on prend l'infimum et $\varphi(x_1,\ldots,x_k,\ldots)$ 
est un prédicat définissable, on note simplement
\[
\inf_{(x_{i_1},\ldots,x_{i_k},\ldots)}\varphi(x_1,\ldots,x_k,\ldots)
\]
le prédicat définissable obtenu. Si on avait pris l'infimum sur toutes les variables, l'interprétation
sera un
nombre réel, qui correspond au minimum du prédicat définissable sur l'espace des types (par densité
de $M^\omega$ dans $S_{\omega}(T))$.

Enfin, la composition par un connecteur logique $c:\R^n\to\R$ 
nous définit une application uniformément continue 
$(\mathfrak F^{\mathcal L})^n\to\mathfrak F^{\mathcal L}$, et donc 
si $\varphi_1,\ldots,\varphi_n$ sont des prédicats définissable, alors $c(\varphi_1,\ldots,\varphi_n)$
l'est aussi et s'interprète dans chaque modèle $M$ comme $c\circ (\varphi_1^M,\ldots,\varphi_n^M)$.

\subsection{Formules sans quantificateur, plongements élémentaires}

Remettons nous dans le cadre général, sans théorie $T$ sous-jacente pour le moment mais avec
un langage $\mathcal L$ fixé.
On a déjà défini 
dans la section~\ref{sec: groupes d'auto metrique} ce que l'on entend par un plongement
d'une $\mathcal L$-structure vers elle même, et on définit de même un \textbf{plongement}
$\rho: M\to N$, où $M$ et $N$ sont deux $\mathcal L$-structures, comme une isométrie telle que
pour tout $n\in\N$:
\begin{itemize}
	\item pour tout symbole de relation $n$-aire $R$, tout $(a_1,\dots,a_n)\in M^n$,
	on a $R^M(a_1,\dots,a_n)=R^N(\rho(a_1),\dots,\rho(a_n))$, et 
	\item pour tout symbole de fonction $n$-aire $f$, tout $(a_1,\dots,a_n)\in M^n$, on a
	$f^M(a_1,\dots,a_n)=f^N(\rho(a_1),\dots,\rho(a_n))$.
\end{itemize}

Autrement dit, un plongement commute aux interprétations
de symboles de relations et de fonctions, ce qui implique qu'il commute aux interprétations
des formules atomiques. 

Plus généralement,
une \textbf{formule sans quantificateur} est une formule obtenue en appliquant des connecteurs logiques
à des formules atomiques (autrement dit, par rapport aux formules en général, on s'interdit d'utiliser la
règle 3 de la définition~\ref{df: formules}). Les formules sans quantificateur forment une sous-algèbre
de l'algèbre des formules, et une preuve par induction immédiate montre le lemme suivant.

\begin{lemm}
	Soient $M$ et $N$ deux $\mathcal L$-structures, soit $\rho: M\to N$ un plongement. 
	Alors pour toute formule sans quantificateur $\varphi(x_1,\dots,x_n)$ et tout $(a_1,\dots,a_n)\in M^n$
	on a $\varphi^M(a_1,\dots,a_n)=\varphi^N(\rho(a_1),\dots,\rho(a_n))$. \qed.
\end{lemm}

Il n'est par contre pas vrai qu'un plongement commute aux formules en général, même entre deux 
structures partageant la même théorie. C'est d'ailleurs une des raisons
de la richesse de la théorie des modèles. 

\begin{exem}Considérons par exemple le groupe $(\Z,+)$ vu
	comme structure métrique pour la distance discrète $\delta$ et la formule
	$$\varphi(y)=\inf_x \delta(y,x+x).$$
	Cette formule est satisfaite par $a\in \Z$ (i.e. $\varphi^\Z(a)=0$) si et seulement si $a$ est divisible par $2$. 
	Alors si $m_2:\Z\to \Z$ est le 
	plongement $\Z\to \Z$ donné par la multiplication par $2$, on a que 
	$\varphi^\Z(1)=1$ mais $\varphi^\Z(m_2(1))=0$, et ce plongement 
	n'est donc pas élémentaire au sens suivant.
\end{exem}
\begin{defi}
	Un plongement $\rho: M\to N$ est dit \textbf{élémentaire} s'il commute aux interprétations des formules: 
	pour toute formule $\varphi(x_1,\dots,x_n)$, 	
	tout $(a_1,\dots,a_n)\in M^n$
	on a $\varphi^M(a_1,\dots,a_n)=\varphi^N(\rho(a_1),\dots,\rho(a_n))$.
\end{defi}

Une \textbf{sous-structure} d'une structure $M$ est un sous-ensemble fermé $N$ de $M$ qui est 
stable par les interprétations des symboles de fonctions, de sorte
à être elle-même une structure pour les interprétations des symboles de $\mathcal L$
restreints à $N$. 
L'image d'un plongement est toujours une sous-structure. Par définition une sous-structure $N$ 
de $M$ est \textbf{élémentaire} si l'inclusion de $N$ dans $M$ est un plongement élémentaire: pour 
toute formule $\varphi(x_1,\dots,x_n)$, on a 
$\varphi^M(a_1,\dots, a_n)=\varphi^N(a_1,\dots, a_n)$.
En particulier si $\varphi$ est un énoncé alors $\varphi^M=\varphi^N$ donc $\Th(N)=\Th(M)$.

L'exemple précédent se reformule alors en disant que $2\Z$ n'est pas une sous-structure élémentaire
du groupe $(\Z,+)$ car $\varphi^\Z(2)\neq\varphi^{2\Z}(2)$. 
On pourrait résumer grossièrement ce fait en disant que $2\Z$ n'a pas assez d'éléments pour voir les
même infimums que $\Z$.
Par contre, tout \emph{isomorphisme} (plongement surjectif) entre deux structures est élémentaire
puisqu'il établit une bijection entre les ensembles sur lesquels les infimums et supremums portent. 

\begin{rema}
	Lorsque $M$ est une sous-structure élémentaire de $N$, on dit aussi que $N$ est une 
	extension élémentaire de $M$. La construction d'extensions élémentaires est fondamentale 
	en théorie des modèles et fait souvent
	appel au théorème de compacité.  Elle apparaît d'ailleurs de manière cruciale dans
	la preuve du théorème d'Ibarlucía dans sa version générale (voir
	la section 3 d'\cite{ibarluciaInfinitedimensionalPolishGroups2021}).
\end{rema}

\begin{rema}
	Soit $G$ le groupe des automorphismes d'une structure métrique~$M$. Alors tout élément 
	de $G$ définit un isomorphisme de $M$ vers elle-même, donc est élémentaire. On verra plus
	tard (et c'est la raison principale pour laquelle on a besoin des plongements élémentaires)
	que dans les cas qui nous intéressent, l'adhérence de $G$ dans $M^M$ sera égale au 
	semi-groupe des plongements élémentaires de $M$ dans elle-même.
\end{rema}

En combinant l'observation que tout automorphisme est élémentaire 
avec le théorème~\ref{thm:dualité}, on obtient le résultat suivant qui fait
le lien avec la section~\ref{sec: prp}.

\begin{prop}\label{prop: quotient type}
	Soit $T$ une théorie complète. Soit $M$ un modèle de $T$ et $n\leq \omega$.
	L'application $\tp: M^n\to S_n(T)$ passe au quotient en une application
	uniformément continue
	$\Aut(M)\bbslash M^n\to S_n(T)$. \qed
\end{prop}

\begin{rema}
	En général, l'application $\Aut(M)\bbslash M^n\to S_n(T)$ n'est ni injective ni surjective.
	Quand l'action de $\Aut(M)$ sur $M^n$ est approximativement cocompacte, par compacité
	de $\Aut(M)\bbslash M^n$ on obtient qu'elle est surjective,
	et on verra même
	qu'elle est bijective, ce qui demande plus de travail (cf.\ théorème~\ref{thm: Ryll-Nardzewski}).
	Avant ça, on établira directement dans le cas concret de $\MAlg([0,1],\lambda)$
	que c'est une bijection, établissant au passage que sa théorie \emph{élimine les quantificateurs}
	(cf.\ théorème~\ref{thm:malg elim}).
\end{rema}

Finissons cette section par une connexion fondamentale entre les plongements élémentaires 
d'une structure séparable et l'espace des suites réalisant un certain type.
Soit $T$ une théorie complète, soient $M$ et $N$ deux modèles de $T$.
Si $\rho: M\to N$ est un plongement élémentaire, alors par densité
des formules dans l'espace des prédicats définissables, on a que pour tout suite $(a_i)_{i\in\N}$, 
$\tp((a_i)_{i\in\N})=\tp(\rho(a_i)_{i\in\N})$. De plus, par le théorème~\ref{thm:dualité},
cette propriété caractérise les plongements élémentaires. On va renverser le point de vue
sur les plongements élémentaires en s'appuyant sur cette remarque.

\begin{prop}\label{prop: plongements comme type}
	Soit $T$ une théorie complète, soient $M$ et $N$ deux modèles séparables de $T$.
	Soit $(\alpha_i)_{i\in\N}$ dense dans $M$, soit $p$ son type.
	L'application qui à un plongement élémentaire $\rho: M\to N$
	associe la suite $(\rho(\alpha_i))_{i\in\N}$ est une injection dont l'image est 
	l'ensemble des suites $(a_i)_{i\in\N}$ telles que $\tp((a_i)_{i\in\N})=p$.
\end{prop}
\begin{proof}
	Notons $\Phi$ l'application $\rho\mapsto (\rho(\alpha_i))_{i\in\N}$.
	L'injectivité de $\Phi$ est conséquence de la densité de $(\alpha_i)$
	et du fait que dans un espace métrique complet, tout isométrie est déterminée
	par sa restriction à une partie dense.
	D'après le paragraphe précédent, l'image de $\Phi$ est incluse dans 
	l'ensemble des suites $(a_i)_{i\in\N}$ telles que $\tp((a_i)_{i\in\N})=p$.
	
	Réciproquement, soit $(a_i)_{i\in\N}$ telle que  $\tp((a_i)_{i\in\N})=p$.
	Alors en particulier, $d(a_i,a_j)=p(d(x_i,x_j))=d(\alpha_i,\alpha_j)$ pour tout $i,j\in\N$,
	donc l'application qui associe à chaque $\alpha_i$ l'élément $a_i$ est une isométrie qui s'étend
	de manière unique en une isométrie $\rho: M\to N$. On voit ensuite que $\rho$ est 
	un plongement en notant plus généralement que 
	$\varphi^N(a_1,\dots,a_n)=\varphi^M(\alpha_1,\dots,\alpha_n)$ pour
	toute formule atomique $\varphi$, et en utilisant la densité. 
	On montre enfin que $\rho$ est élémentaire
	par le même raisonnement, en considérant cette fois-ci toutes les formules $\varphi$.
\end{proof}
\begin{rema}
	La même preuve fonctionne si l'on suppose seulement que $(\alpha_i)_{i\in\N}$
	engendre $M$ (ce qui veut dire que la plus petite sous structure
	contenant les $\alpha_i$ est égale à $M$; rappelons que par définition
	une sous-structure est fermée).
\end{rema}

\subsection{Élimination des quantificateurs}

\begin{defi} 
	On dit qu'une théorie complète $T$ \textbf{élimine les quantificateurs}
	lorsque la sous-algèbre  des formules 
	sans quantificateur est dense dans l'algèbre des prédicats définissables.
\end{defi}

Remarquons que si une théorie élimine les quantificateurs, alors tout
plongement entre des modèles de $T$ est élémentaire.
On va utiliser la preuve que l'action de $\Aut([0,1],\lambda)$ sur
$\MAlg([0,1],\lambda)$ est approximativement oligomorphe 
pour montrer
que la théorie de $\MAlg([0,1],\lambda)$ élimine les quantificateurs.

Avant ça, définissons les types sans quantificateur.
Disons qu'un prédicat définissable est \emph{sans quantificateur}
s'il est dans l'adhérence des formules sans quantificateur, et
notons $\mathfrak P^{\mathcal L,T}_{\qf}$ l'espace des 
prédicats définissables sans quantificateur. Pour $n\leq\omega$,
on note $\mathfrak P^{\mathcal L,T}_{n,\qf}=\mathfrak P_n^{\mathcal L,T}\cap \mathfrak P^{\mathcal L,T}_{\qf}$.

\begin{defi}
	Soit $n\leq\omega$.
	Un $n$-type sans quantificateur est un morphisme continu de norme $1$ de 
	l'algèbre des prédicats définissables sans quantificateur dans $\R$.
\end{defi}

On note $S_{n}^{\qf}(T)$ l'espace des types sans quantificateur (qf pour \emph{quantifier free}).
De même que pour les types, on montre que c'est un espace compact. 
Si $M$ est un modèle de $T$,
 pour chaque $\bar a\in M^n$ on a un type sans quantificateur
associé $p=\tp_{\qf}(\bar a)$ défini par $p(\varphi)=\varphi(\bar a)$, 
et l'application $\bar a \mapsto \tp_{\qf}(\bar a)$ est d'image dense. On a une projection naturelle
$\pi: S_n(T)\to S_n^{\qf}(T)$ donnée par $\pi(p)(\varphi)=p(\varphi)$ de sorte que pour tout $\bar a\in M^n$, 
$$\pi(\tp(\bar a))=\tp_{\qf}(\bar a).$$
L'application $\tp:M^n\to S_{n}^{\qf}(T)$ étant  d'image dense, 
on en déduit que $\pi$ est surjective. 
Par dualité, $\pi:S_\omega(T)\to S^{\qf}_\omega(T)$ est injective si et seulement si $T$ élimine les quantificateurs.
Enfin, on a l'analogue suivant de la proposition~\ref{prop: quotient type}.

\begin{prop}\label{prop:orbites fermées qf}
	Soit $T$ une théorie complète. Soit $M$ un modèle de $T$ et $n\leq \omega$.
	L'application $\tp_{\qf}: M^n\to S_n^{\qf}(T)$ passe au quotient en une application
	uniformément continue
	$\Aut(M)\bbslash M^n\to S_n^{\qf}(T)$. \qed
\end{prop}

\begin{theo}\label{thm:malg elim}
	La théorie $T$ de $\MAlg([0,1],\mu)$ élimine les quantificateurs, et pour tout $n\leq\omega$, 
	l'application 
	$\Aut([0,1],\lambda)\bbslash \MAlg([0,1],\lambda)^n\to S_n(T)$ est une bijection.
\end{theo}
\begin{proof}
	Comme l'action de $\Aut([0,1],\lambda)$ sur $\MAlg([0,1],\lambda)$ est approximativement oligomorphe 
	(exemple~\ref{ex: approxi oligo pour malg}),
	on a d'après les propositions~\ref{prop: quotient type} et~\ref{prop:orbites fermées qf} 
	que les applications $\tp_{\qf}$ et $\tp$ sont surjectives pour tout $n\leq\omega$ par compacité et densité.
	Soit $n\leq\omega$. On note $\pi:S_n(T)\to S_n^{\qf}(T)$ la projection naturelle.
	Pour tout $\bar a \in M^n$, on a
	$\pi(\tp(\bar a))=\tp_{\qf}(\bar a)$, donc pour montrer que $\pi$ est injective (et donc que $T$
	élimine les quantificateurs), il suffit de montrer que l'application 
	$\Aut([0,1],\lambda)\bbslash \MAlg([0,1],\lambda)^n\to S_n^{\qf}(T)$ est injective.
	
	Supposons d'abord $n<\omega$. Soient $[A_1,\dots,A_n]$ et $[B_1,\dots,B_n]$ des éléments de
	 $\Aut([0,1],\lambda)\bbslash\MAlg([0,1],\lambda)^n$
	tels que $\tp_{\qf}(A_1,\dots,A_n)=\tp_{\qf}(B_1,\dots,B_n)$. 
	Considérons, pour chaque 
	$\delta\in\{\mathds O,\mathds 1\}^n$, la formule sans quantificateur
	$$\varphi_\delta(x_1,\dots,x_n)=m((\delta(1)\fplus x_1)\wedge\cdots\wedge(\delta(n)\fplus x_n)).$$
	Alors, en reprenant les notations de l'exemple~\ref{ex: approxi oligo pour malg}, 
	$\varphi_\delta^M$ est la fonction qui à $(C_1,\dots,C_n)\in \MAlg([0,1],\lambda)^n$
	associe 
	$$\Phi(C_1,\dots,C_n)(\delta)=\lambda(C_1^{\delta(1)}\cap\cdots\cap C_n^{\delta(n)}).$$
	Comme $\tp_{\qf}(A_1,\dots,A_n)=\tp_{\qf}(B_1,\dots,B_n)$, on a 
	$\varphi_\delta^M(A_1,\dots,A_n)=\varphi_\delta^M(B_1,\dots,B_n)$ pour tout 
	$\delta\in\{\mathds O,\mathds 1\}^n$, et donc $\Phi(A_1,\dots,A_n)=\Phi(B_1,\dots,B_n)$. 
	Comme $\Phi$ est injective d'après la preuve au sein de l'exemple~\ref{ex: approxi oligo pour malg},
	on conclut que $[A_1,\dots,A_n]=[B_1,\dots,B_n]$ comme voulu. 
	
	Le cas $n=\omega$ est maintenant une conséquence aisée des cas $n<\omega$ que l'on vient d'établir.
\end{proof}

\begin{rema}
	Le fait que l'application $\Aut([0,1],\lambda)\bbslash \MAlg([0,1],\lambda)^n\to S_n(T)$ soit une 
	bijection se généralise aux structures métriques séparables dont le groupe d'automorphismes
	agit de manière approximativement oligomorphe (théorème~\ref{thm: Ryll-Nardzewski}). Il s'agit d'une partie
	du théorème de Ryll-Nardzewski, qui comme expliqué dans l'introduction joue un rôle fondamental dans 
	la preuve d'Ibarlucía ainsi que dans l'analyse des groupes polonais Roelcke-précompacts en général.
\end{rema}

\section{Définissabilité de sous-ensembles}\label{sec: definissable}

On fixe dans cette section $n\leq \omega$ et une théorie complète $T$.
On va travailler sur 
des prédicats définissables dans $\mathfrak P^{\mathcal L,T}_n$. Afin d'alléger les notations,
on notera désormais $\bar x\coloneq x_1,\dots,x_n$ 
(si $n<\omega$) et $\bar x\coloneq x_1,\ldots$ si $n=\omega$. Ainsi, un élément de
 $\mathfrak P^{\mathcal L,T}_n$
sera noté~$\varphi(\bar x)$.

Quitte à réindexer notre ensemble de variables, on suppose qu'on a également des variables 
$y_1,\ldots,y_n,\ldots$
(par exemple en remplaçant $x_{2n}$ par $x_n$ et $x_{2n+1}$ par $y_n$ pour chaque $n<\omega$), 
et on notera de même $\bar y\coloneq y_1,\dots,y_n$ 
(si $n<\omega$) et $\bar y\coloneq y_1,\ldots$ si $n=\omega$.

On commence par un détour par la notion d'implication en théorie des modèles métrique.

\subsection{Implications}

Le fait de n'avoir que des infimums et des supremums pour quantifier
rend la notion d'implication plus subtile en théorie des modèles continue qu'en théorie des modèles classique,
prenant une forme  quantitative.
Le but reste de donner une condition sur deux prédicats définissables
$\varphi(\bar x)$ et $\psi(\bar x)$ qui fera que dans \emph{tout} modèle $M$ de $T$, si un uplet 
$\bar a\in M^n$ \emph{satisfait} $\varphi(\bar x)$ (c'est-à-dire si $\varphi^M(\bar a)=0$), alors
$\bar a$ satisfait également $\psi(\bar x)$ (et donc $\psi^M(\bar b)=0$).

\begin{defi}
	Soit $T$ une théorie complète, soient $\varphi(\bar x)$ et $\psi(\bar x)$ deux prédicats définissables.
	On dit que $\varphi(\bar x)$ \textbf{implique} $\psi(\bar x)$ 
	si  il existe un modèle $M$ de $T$ tel que 
	pour tout  $\epsilon>0$, il existe $\delta>0$ tel que pour tout $\bar a\in M^n$,
	$$\text{si }\abs{\varphi^M(\bar a)}\leq \delta \text{ alors } \abs{\psi^N(\bar a)}\leq \epsilon.$$
\end{defi}

	\begin{exem} Dans la théorie d'une algèbre de mesure, 
		la formule qui dit que deux ensembles sont de mesure égale
		implique la formule qui dit que la mesure de leur union est plus petite que deux fois
		la mesure du premier.
		Plus précisément, la formule $\abs{m(x_1)-m(x_2)}$ implique la formule
		$m(x_1\vee x_2)\dotdiv 2m(x_1)$.
	\end{exem}

En utilisant des idées très proches de celles de la remarque~\ref{rmk:unif continuite}, 
on montre que si $\varphi(\bar x)$ implique $\psi(\bar x)$, alors elle le fait 
``de manière uniforme sur les modèles de $T$''.
\begin{prop}
	Si $\varphi(\bar x)$ implique $\psi(\bar x)$, alors pour tout $\epsilon>0$, il existe $\delta>0$
	tel que pour \emph{tout} modèle $N$ de $T$, et tout $\bar b\in N^n$,
	$$\text{si }\abs{\varphi^N(\bar b)}\leq \delta \text{ alors } \abs{\psi^N(\bar b)}\leq \epsilon.$$
	En particulier, si $\varphi^N(\bar b)=0$ alors $\psi^N(\bar b)=0$.
\end{prop}
\begin{proof}
	Soit $M$ un modèle de $T$ témoignant du fait que $\varphi(\bar x)$ implique $\psi(\bar x)$.
	Soit $\epsilon>0$, on dispose de $\delta>0$ tel que pour tout $\bar a\in M^n$,
	si $\abs{\varphi(\bar a)}\leq \delta$ alors $\abs{\psi(\bar a)}\leq \epsilon$. Mais ceci 
	est équivalent à dire que $M$ satisfait l'énoncé suivant:
	\[
	\sup_{\bar x} \min\big(\delta\dotdiv \abs{\varphi(\bar x)}, \abs{\psi(\bar x)}\dotdiv\epsilon\big).
	\]
	Par complétude de $T$, un tel énoncé sera satisfait dans tout modèle de~$T$, 
	ce qui nous donne le résultat voulu.
\end{proof}

Le lemme suivant va nous permettre de donner une autre manière de comprendre l'implication.

\begin{lemm}\label{lem:implication}
	Soit $X$ un ensemble.
	Soient $f,g: X\to \R^+$ bornées telles que pour tout $\epsilon>0$, il existe $\delta>0$ tel que
	pour tout $x\in X$, $f(x)\leq \delta$ implique $g(x)\leq \epsilon$. Alors il existe 
	$\alpha:\R^+\to\R^+$ continue telle que $\alpha(0)=0$ et pour tout $x\in X$, 
	$g(x)\leq\alpha(f(x))$.
\end{lemm}
\begin{proof}
	Posons dans un premier temps, pour $t\geq 0$,
	$$
	\beta(t)=\sup
	\{ g(x)\colon f(x)\leq t\}.
	$$
	Alors $\beta$ est croissante, et par hypothèse on a $\lim_{t\to 0}\beta(t)=0$.
	De plus par définition pour tout $t\geq 0$, on a que  pour tout $x\in X$, 
	si $g(x)\leq t$ alors $f(x)\leq \beta(t)$, et donc on a 
	$g(x)\leq \beta(f(x))$. Le seul problème
	est que $\beta$ n'est pas continue. On prend alors une suite
	strictement croissante indexée par les entiers relatifs $(t_n)_{n\in\mathbb Z}$
	de réels strictement positifs avec $\lim_{n\to-\infty} t_n=0$ (par exemple $t_n=2^{n}$),
	et on définit $\alpha:\R^+\to\R^+$ comme l'unique fonction affine sur chaque intervalle
	$[t_n,t_{n+1}]$
	telle que $\alpha(t_n)=\beta(t_{n+1})$ pour tout $n\in\mathbb Z$, et $\alpha(0)=0$. La
	fonction $\alpha$ est continue et puisque
	$\alpha\geq \beta$ on a bien que pour tout $x\in X$, 
	$g(x)\leq \alpha(f(x))$. 
\end{proof}

\begin{prop}
	Soit $T$ une théorie complète, soient $\varphi(\bar x)$ et $\psi(\bar x)$ deux prédicats
	définissables. S'équivalent:
	\begin{enumerate}[(i)]
		\item $\varphi(\bar x)$ implique $\psi(\bar x)$;
		\item Il existe une fonction continue $\alpha:\R\to\R$ telle que $\alpha(0)=0$ 
		et telle que dans tout modèle~$M$ de~$T$
		on a, pour tout $\bar a\in M^n$, que $\abs{\varphi^M(\bar a)}\leq \alpha(\psi^M(\bar a))$.
		\item Il existe une fonction continue $\alpha:\R\to\R$ telle que $\alpha(0)=0$ 
		et il existe un modèle $M$ de $T$
		tel que pour tout $\bar a\in M^n$, on a $\abs{\varphi^M(\bar a)}\leq \alpha(\psi^M(\bar a))$.
	\end{enumerate}
\end{prop}
\begin{proof}
	On va montrer que (ii) équivaut à (iii), puis que (i) équivaut à (iii).
	
	L'implication (ii) $\Rightarrow$ (iii) est claire. Réciproquement, si (iii) est satisfaite, 
	la fonction continue $\alpha$ donnée par hypothèse est
	un connecteur logique, ainsi $\alpha(\psi(\bar x))$ est un prédicat définissable,
	et on remarque que $M$ satisfait l'énoncé
	$$\sup_{\bar x} \left(\abs{\varphi^M(\bar x)}\dotdiv \alpha(\psi(\bar x)\right).$$
	Comme $T$ est complète, tout modèle de $T$ satisfera cet énoncé, ce qui veut dire 
	que (ii) est vérifiée.
	
	Si (i) est satisfaite, soit $M$ un modèle témoignant du fait que 
	$\varphi(\bar x)$ implique $\psi(\bar x)$.
	Par le lemme~\ref{lem:implication} on dispose de $\alpha:\R^+\to\R^+$ continue telle 
	que $\alpha(0)=0$ et pour tout $\bar a\in M^n$, $\abs{\varphi^M(\bar a)}\leq \alpha(\psi^M(\bar a))$.
	On la prolonge en une fonction continue sur $\R$ en posant $\alpha(t)=0$
	pour tout $t<0$, et (iii) est alors établie. 
	Réciproquement si (iii) est satisfaite, la continuité de $\alpha$
	nous assure que $M$ témoigne du fait que $\varphi(\bar x)$ implique $\psi(\bar x)$, donc (i) 
	est vérifiée.	
\end{proof}

\begin{rema}
	En utilisant le théorème de compacité, on peut montrer que $\varphi(\bar x)$
	implique $\psi(\bar x)$ si et seulement si dans \emph{tout} modèle $M$ de $T$, on a
	pour tout $\bar a \in M^n$ que si $\varphi^M(\bar a)=0$, alors $\psi^M(\bar a)=0$
	(cf.\ \cite[Prop.~7.15]{benyaacovModeltheorymetric2008}).
\end{rema}

\subsection{Définissabilité dans un modèle}\label{sec: defin dans modèle}

On va avoir besoin de travailler sur $M^n$ pour $n\leq\omega$ avec une distance compatible.
Afin de gagner en lisibilité, on utilise les notations suivantes.
Pour $n<\omega$, on note désormais $d$ le prédicat définissable à $2n$ variables libres 
$d(x_1,y_1)+\dots+d(x_n,y_n)$, et si $n=\omega$ on note également $d$ le prédicat définissable
$\sum_{i\in\N} \frac 1{2^{i}}d(x_i,y_i)$ si $n=\omega$. Ce prédicat définissable s'interprète
donc comme une distance compatible sur $M^n$ si $M$ est un modèle de $T$.

\begin{defi}
	Soit $D\subseteq M^n$ non vide, où $n\leq\omega$. On dit que $D$ est définissable si la 
	fonction $\bar a \in M^n\mapsto d^M(\bar a,D)$ coïncide avec l'interprétation 
	sur $M^n$ d'un prédicat définissable.
\end{defi}

Au vu de la remarque~\ref{rmq: identification prédicats}, un tel prédicat définissable est alors unique, 
et on s'autorisera un abus de langage en disant directement que $d(\bar x,D)$ est un prédicat définissable.

En pratique, on dispose d'un critère plus souple pour décider de la définissabilité d'un sous-ensemble
de $M^n$. Il s'appuie sur le lemme~\ref{lem:implication} qui nous avait fourni une autre manière
de comprendre l'implication.

\begin{theo}\label{thm:definissabilité}
Un sous-ensemble~$D$ de~$M^n$ est définissable si et seulement s'il existe un prédicat définissable 
$\varphi(\bar x)\in\mathfrak P^{\mathcal L,T}_n$ à valeurs positives tel que pour tout $\epsilon>0$,
il existe $\delta>0$ tel que  $D=(\varphi^M)^{-1}(\{0\})$ et 
pour tout $\bar a\in M^n$, si $\varphi^M(\bar a)\leq \delta$
alors~$d(\bar a, D)\leq \epsilon$.
\end{theo}

\begin{rema}
	Une fois que l'on sait que $d(\bar x, D)$ est définissable, on voit que notre nouvelle condition
	de définissabilité demande en particulier que $\varphi(\bar x)$ implique $d(\bar x,D)$. 
	De plus, on a également que $d(\bar x,D)$ implique $\varphi(\bar x)$ par uniforme continuité
	de $\varphi^M$
	et le fait que $D\subseteq (\varphi^M)^{-1}(\{0\})$.
\end{rema}

\begin{proof}
	L'implication directe est claire. Réciproquement, supposons donc donné un prédicat 
	définissable $\varphi(\bar x)$ dont $D$ est l'ensemble des zéros et tel que pour tout 
	$\bar a\in M^n$, on a l'implication 
	$\varphi^M(\bar a)\leq \delta\Rightarrow d(\bar a, D)\leq \epsilon$.
	Quitte à remplacer $\varphi$ par $\abs{\varphi}$, on peut supposer que $\varphi$ est 
	à valeurs positives.
	
	Par le lemme~\ref{lem:implication} (où $X=M^n$), 
	il existe une fonction continue $\alpha:\R^+\to \R^+$ telle que 
	pour tout $\bar a\in M^n$, $d(\bar a,D)\leq \alpha(\varphi^M(\bar a))$. 
	On étend alors $\alpha$ 
	en un connecteur logique en posant $\alpha(t)=0$ pour $t\leq 0$.
    
    Considérons
     le prédicat définissable\footnote{Noter que la définition qui suit fait sens
    	d'après la fin de la section précédente:
    	ici l'infimum peut porter sur un uplet dénombrable
    	infini.}
    \[
    \psi(\bar x)=\inf_{\bar y} \alpha(\varphi(\bar x))+d(\bar x,\bar y).
    \]
    On va montrer que son interprétation sur $M$ coïncide avec la fonction $d^M(\cdot,D)$,
    ce qui montrera bien que cette dernière est un prédicat définissable.
    
    Soit donc $\bar a\in M^n$. On a $d^M(\bar a, D)\leq \alpha(\varphi^M(\bar x))$ et 
    donc $\psi^M(\bar a)\geq d^M(\bar a,D)$. Inversement, on dispose d'une suite $(\bar d_m)$ 
    d'éléments de $A$ telle que $d^M(\bar a,\bar d_m)\to d^M(\bar a,D)$, et comme 
    $\alpha(\varphi(\bar d_m))=0$ pour tout $m$, on conclut que $\psi^M(\bar a)\leq d^M(\bar a, D)$
    comme voulu.
\end{proof}

Si $A$ est un ensemble définissable dans $M$
modèle d'une théorie complète $T$, on peut donc lui associer un prédicat définissable
particulier: la distance à $A$ dans $M$. On va montrer dans la section suivante que
ce prédicat restera la distance à une partie non vide dans \emph{tout}
modèle de $T$, en utilisant une caractérisation des fonctions distances
à des fermés donnée par \textcite[Thm.~9.12]{benyaacovModeltheorymetric2008}.

Finissons cette section par une propriété importante qui nous permet de définir 
des prédicats définissables à partir de sous-ensembles définissables.

\begin{prop}\label{prop:quantif sur def}
	Soient $n,m\leq \omega$.
	Soit $D$ un sous-ensemble de $M^m$ définissable, soit $\varphi(\bar x,\bar y)$ un prédicat définissable.
	Alors la fonction $\tilde\psi:M^n\to \R$ donnée par
	$$\tilde\psi(\bar a)=\sup_{\bar b\in D} \varphi^M(\bar a,\bar b)$$
	coïncide avec l'interprétation d'un unique prédicat définissable.
\end{prop}
\begin{proof}
	La fonction 
	$(\bar a_1,\bar a_2,\bar b,\bar c)\mapsto \abs{\varphi^M(\bar a_1,\bar b)-\varphi^M(\bar a_2,\bar c)}$
	est uniformément continue, donc d'après le lemme~\ref{lem:implication}
	on dispose de $\alpha: \R\to \R$ continue telle que pour tous $\bar a\in M^n$, 
	et $\bar b,\bar c\in M^m$, on a 
	$\abs{\varphi^M(\bar a,\bar b)-\varphi^M(\bar a,\bar c)}\leq \alpha(d^M(\bar b, \bar c))$ et $\alpha(0)=0$.
	Commme à la fin de la preuve du théorème~\ref{thm:definissabilité}, on va  conclure la preuve 
	en montrant que $\tilde\psi$ coïncide avec l'interprétation du prédicat définissable 
	$\psi(\bar x)\coloneq\sup_{\bar y}\left(\varphi(\bar x,\bar y)-\alpha(d(\bar y,D))\right)$ 
	par double inégalité.
	
	Si $\bar a\in M^n$, par définition on a  $\psi^M(\bar a)\geq \tilde\psi(\bar a)$.
	Réciproquement, à $\epsilon>0$ fixé, prenons $\bar c\in M^n$.
	Par continuité de $\alpha$ et définition de $d^M(\cdot,D)$, on trouve $\bar b\in D$
	tel que $\abs{\alpha(d^M(\bar c,\bar b))-\alpha(d^M(\bar c, D))}<\epsilon$. Alors par définition
	$\varphi^M(\bar a,\bar b)\leq \tilde\psi(\bar a)$, et comme 
	$\abs{\varphi^M(\bar a,\bar b)-\varphi^M(\bar a,\bar c)}\leq\alpha(d^M(\bar b,\bar c))$,
	on a $$\varphi^M(\bar a,\bar c)\leq \varphi^M(\bar a,\bar b)+\alpha(d^M(\bar b,\bar c))\leq 
	\tilde\psi(\bar a)+\alpha(d^M(\bar c,D))+\epsilon.$$
	Ainsi $\varphi^M(\bar a, \bar c)-\alpha(d^M(\bar c,D))\leq \tilde\psi(\bar a)+\epsilon$, ce qui en faisant 
	tendre $\epsilon$ vers $0$ nous donne bien $\psi^M(\bar a)\leq \tilde \psi(\bar a)$.
\end{proof}
\begin{rema}
	Encore une fois, le prédicat définissable $\psi$ dont l'interprétation sur~$M$ coïncide avec $\tilde\psi$
	est unique, et on se permettra de le noter $\sup_{\bar y\in D}\varphi(\bar x,\bar y)$. 
	On utilisera aussi la notation suivante lorsque l'on préfère ne pas travailler dans un modèle fixé 
	et que l'on a noté $\phi(\bar y)$ le prédicat définissable $d(\bar y, D)$:
	$$\displaystyle\sup_{\phi(\bar y)=0}\varphi(\bar x,\bar y).$$
\end{rema}

\begin{rema}
	Le même résultat est vrai en remplaçant le supremum par un infimum, et c'est
	sous cette forme qu'il apparait dans \textcite[Thm.~9.17]{benyaacovModeltheorymetric2008}.
\end{rema}

\subsection{Définissabilité dans tous les modèles, types isolés}

Soit $T$ une théorie complète et soit $M$ un modèle de $T$. À un sous-ensemble
définissable~$D$ de~$M^n$ correspond un unique prédicat définissable s'interprétant 
dans $M$ comme la distance à l'ensemble $D$, que l'on a noté abusivement $d(\bar x,D)$. 
Le but de cette section est 
de vérifier que ce prédicat définissable s'interprète 
dans \emph{tout} modèle $N$ de $T$ comme la distance à un  sous-ensemble fermé non vide de $N^n$.
On va ensuite utiliser ceci pour montrer que si $p$ est un $n$-type dont l'ensemble $D$ des réalisations
est définissable
dans \emph{un} modèle $M$, 
alors $d(\bar x,D)$ s'interprète dans \emph{tout} modèle $N$ de $T$ comme la distance
à l'ensemble des réalisations de $p$ dans $N$, qui est non vide.

Afin de voir que $d(\bar x,D)$ s'interprète toujours comme la distance à un sous-ensemble fermé non vide,
il nous faut voir que la théorie $T$ ``sait'' que $d(\bar x,D)$ est la distance à un sous-ensemble fermé
\parencite[Thm.~9.12]{benyaacovModeltheorymetric2008}. 
Remarquons qu'une axiomatisation similaire des distances à des sous-ensembles fermés était déjà apparue en 
théorie descriptive des ensembles, afin de munir l'espace des fermés de tout espace métrique 
complet séparable
d'une topologie polonaise appelée topologie de Wijsman \parencite{beerPolishtopologyclosed1991}.

\begin{lemm}
	Soit $(X,d)$ un espace métrique complet non vide. 
	Alors une fonction continue $f: X\to \R^+$ est la distance à l'ensemble non vide de ses zéros si et seulement si 
	elle satisfait les conditions suivantes
	pour tout $x\in X$:
	\begin{align}
		&\inf_{y\in X} \max(f(y), \abs{f(x)-d(x,y)})=0;\label{cond:distance1}\\
		 &f(x)=\inf_{y\in X} f(y)+d(x,y).\label{cond:distance2}
	\end{align} 
\end{lemm}
\begin{proof}
	Pour voir l'implication directe, prenons $D\subseteq X$ non vide fermé et considérons l'application
	$f:x\mapsto d(D,x)$. 
	On constate directement que la condition~\eqref{cond:distance1} est vérifiée en considérant,
	pour chaque $\epsilon>0$, un élément $y\in D$ tel que $d(D,x)\leq d(y,x)<d(D,x)+\epsilon$. 
	On montre de même
	que $\inf_{y\in X} f(y)+d(x,y)\leq f(x)$, ce qui en utilisant le fait que $f$ est $1$-lipschitzienne
	nous permet de conclure que~\eqref{cond:distance2} est satisfaite.
	
	Réciproquement, soit $f: X\to \R^+$ satisfaisant les conditions~\eqref{cond:distance1} et~\eqref{cond:distance2}, et soit $x\in X$.
	Soit $D$ l'ensemble des zéros de $f$, on va d'abord montrer que $d(x,D)\leq f(x)$ en exhibant
	pour tout $\epsilon>0$ un $y\in D$ tel que $d(x,y)\leq f(x)+\epsilon$, ce qui montrera en 
	particulier que $D$~est non vide.

	Soit donc $\epsilon>0$. On doit trouver $y\in D$ tel que $d(x,y)\leq f(x)+\epsilon$,
	et on va l'obtenir comme limite d'une suite de Cauchy. On pose $y_0=x$ puis, $y_i$ étant construit,
	on applique la condition~\eqref{cond:distance1} pour choisir $y_{i+1}\in X$
	tel que 
	$$f(y_{i+1})\leq \frac \epsilon{2^{i+3}} \text{ et } 
	\abs{f(y_i)-d(y_i,y_{i+1})}\leq \frac \epsilon{2^{i+2}}.$$
	En assemblant ces deux conditions, on obtient que pour tout $i\geq 1$, 
	$d(y_i,y_{i+1})\leq \frac{\epsilon}{2^{i+1}}$,
	donc la suite $(y_i)$ est de Cauchy et de plus sa limite $y$ 
	satisfait 
	$$d(y_0,y)\leq d(y_0,y_1) + \sum_{i\geq 1} \frac{\epsilon}{2^{i+1}}
	=d(y_0,y_1)+\frac{\epsilon}2\leq f(x)+\epsilon,$$
	la dernière inégalité étant conséquence du fait que 
	$\abs{d(y_0,y_1)-f(y_0)}\leq \frac\epsilon2$ et $y_0=x$ . 
	Comme $f$ est continue
	à valeurs positives
	et comme pour tout $i\geq 1$ on a $f(y_{i+1})\leq \frac \epsilon{2^{i+3}}$, on a de plus $f(y)=0$,
	et donc $y\in D$ comme voulu. Ainsi on a bien que $d(x,D)\leq f(x)+\epsilon$ pour tout $\epsilon>0$,
	et donc que $d(x,D)\leq f(x)$.
	
	Enfin, la condition~\eqref{cond:distance2} nous garantit que pour tout $x\in X$ et tout $y\in D$, on a
	$f(x)\leq d(x,y)$, et donc $f(x)\leq d(x,D)$, ce qui termine la preuve.	
\end{proof}

\begin{prop}
	Soit $M$ un modèle d'une théorie $T$ complète, soit $n\leq\omega$, soit $D\subseteq M^n$
	définissable. Alors le prédicat définissable $d(\bar x,D)$ s'interprète dans tout modèle 
	de $T$ comme la distance à un sous-ensemble fermé non vide. 
\end{prop}
\begin{proof}
	Le fait que $d^M(\bar x,D)$ soit une fonction positive se traduit par le fait que
	$M$ satisfasse l'énoncé $\inf_{\bar x}\min(\0,d(\bar x,D))$.
	De plus, les conditions~\eqref{cond:distance1} et~\eqref{cond:distance2} se traduisent
	par le fait que $M$ satisfasse les énoncés suivants, qui sont donc dans $T$:
	\begin{enumerate}[(i)]
		\item $\displaystyle \sup_{\bar x}\inf_{\bar y} d(\bar x,D)-d(\bar x,\bar y)$;
		\item $\displaystyle \sup_{\bar x}\abs{d(\bar x,D)-\inf_{\bar y}\left( d(\bar x,D)
			  +d(\bar x,\bar y)\right)}$.
	\end{enumerate}
	Ainsi $d(\bar x,D)$ sera 
	interprétée dans tout modèle de $T$ comme la distance à un ensemble non vide d'après
	la proposition précédente.
\end{proof}

\begin{defi}\label{df:type isolé}
	Soit $T$ une théorie complète.
	Un type $p\in S_n(T)$ est dit \textbf{isolé} s'il existe un modèle $M$ dans 
	lequel l'ensemble des réalisations de $p$ est définissable (en particulier, non vide).
\end{defi}

\begin{theo}\label{thm:principal est real partout}
	Soit $p\in S_n(T)$ un type isolé. Alors $p$ est réalisé dans tout modèle de $T$, et
	il existe un (unique) prédicat définissable $d(\bar x,p)$
	dont l'interprétation dans tout modèle de $T$ est la distance à l'ensemble des réalisations 
	de $p$. 
\end{theo}
\begin{proof}
	Soit $M$ un modèle de $T$ tel que l'ensemble $D$ des réalisations de $p$ soit définissable.
	On pose alors $d(\bar x,p)\coloneq d(\bar x,D)$. Ce prédicat définissable appartient clairement 
	à $p$. Pour chaque prédicat définissable $\varphi(\bar x)$ appartenant à $p$,
	l'uniforme continuité de $\varphi^M$ nous donne que $d(\bar x, p)$ implique $\varphi(\bar x)$. En 
	effet à $\epsilon>0$ fixé on dispose de $\delta>0$ tel que
	pour tous $\bar a,\bar b\in M^n$, si $d^M(\bar a,\bar b)<\delta$ alors 
	$\abs{\varphi^M(\bar a)-\varphi^M(\bar b)}<\epsilon$. Donc si $d(\bar a,p)<\delta$
	on dispose de $\bar b$ satisfaisant $p$ tel que $d^M(\bar a,\bar b)<\delta$,
	donc  $\varphi^M(\bar b)=0$ et donc~$\abs{\varphi^M(\bar a)}<\epsilon$.

	Soit alors $N$ un modèle quelconque de $T$. Soit $D'$ l'ensemble des zéros de $d^N(\bar x,p)$.
	Alors comme $d(\bar x,p)$ appartient à $p$, toute réalisation de $p$ doit appartenir à $D'$.
	Réciproquement, soit $\bar a\in D'$, soit $\varphi(\bar x)\in p$. 
	Comme $d(\bar x,p)$ implique $\varphi(\bar x)$,
	on a bien  $\varphi^N(\bar a)=0$. Ainsi $\bar a$ est une réalisation de $p$, ce qui termine la
	preuve.
\end{proof}

Nous pouvons enfin appliquer le résultat précédent dans le cas simple de $\MAlg([0,1],\lambda)$.
Le cas général est donné dans la proposition~\ref{prop:tp realise si compacité}.

\begin{theo}\label{thm: real defin xmu}
	Soit $T$ la théorie de $\MAlg([0,1],\lambda)$. Alors tout $\omega$-type $p\in S_\omega(T)$
	est isolé.
\end{theo}
\begin{proof}
	On a déjà vu que $p$ est réalisé dans $\MAlg([0,1],\lambda)$ (théorème~\ref{thm:malg elim}). 
	Soit $D\subseteq\MAlg([0,1],\lambda)^\N$ l'ensemble fermé 
	$\Aut([0,1],\lambda)$-invariant des réalisations de~$p$.
	Considérons la fonction $\tilde\varphi$ sur $\MAlg([0,1],\lambda)^\N$ donnée par 
	$\tilde\varphi(\bar a)=d^{\MAlg(X,\mu)}(\bar a,D)$. Alors $\tilde\varphi$ est 
	$\Aut([0,1],\lambda)$-invariante, c'est donc l'interprétation d'un prédicat définissable.
	On conclut que l'ensemble des réalisations de $p$ est définissable, donc $p$ est bien
	isolé.
\end{proof}

\section{Preuve du théorème~\ref{THM:MAINTER}}\label{sec:pfautxmu}

Soit $T$ la théorie de l'algèbre de mesure $\MAlg([0,1],\lambda)$.
Nous avons déjà montré que l'action de $\Aut([0,1],\lambda)$  sur $\MAlg([0,1],\lambda)$
est approximativement oligomorphe (cf.\ exemple~\ref{ex: approxi oligo pour malg}), 
ce qui nous a permis d'identifier
$\Aut([0,1],\lambda)\bbslash \MAlg([0,1],\lambda)^\N$ avec $S_\omega(T)$. Rappelons que 
les fonctions continues $\Aut([0,1],\lambda)\bbslash \MAlg([0,1],\lambda)^\N\to\R$
s'identifient par passage au quotient avec les fonctions $\MAlg([0,1],\lambda)^\N\to\R$ continues 
$\Aut([0,1],\lambda)$-invariantes: \emph{l'algèbre des prédicats définissables s'identifie à l'algèbre
des fonctions  $\MAlg([0,1],\lambda)^\N\to\R$ continues 
$\Aut([0,1],\lambda)$-invariantes}.

Il va désormais être plus commode, comme expliqué à la fin de la section~\ref{sec:(T)}, de travailler
dans l'algèbre de mesure de $X=\{0,1\}^{\N\times\mathbb F_2}$ muni de la mesure 
$\mu=(\frac 12 \delta_0+\frac 12 \delta_1)^{\otimes(\N\times\mathbb F_2)}$.
Elle est
isomorphe à l'algèbre de mesure de $Y\coloneq\{0,1\}^\N$ muni de la mesure 
$\nu\coloneq(\frac 12 \delta_0+\frac 12 \delta_1)^{\otimes\N}$ puisque l'on peut trouver,
en utilisant une bijection $\N\to\N\times\mathbb F_2$, une bijection bimesurable $Y\to X$ qui préserve
la mesure.

Pour $\gamma\in\mathbb F_2$, soit $f_{\gamma}:X\to Y$ définie par 
$f_{\gamma}(x)=(x(n,\gamma))_{n\in\N}$. Alors $f_{\gamma}$ préserve la mesure :
pour tout borélien $A\subseteq Y$, $\nu(A)=\mu(f_{\gamma}\inv(A))$. 
On vérifie alors 
que l'application qui à un borélien $A$ de $Y$ associe $f_{\gamma}\inv(A)$
passe au quotient en un plongement d'algèbres de mesure que l'on note $f_\gamma^*:\MAlg(Y,\nu)\to\MAlg(X,\mu)$.

\begin{rema}Ce passage au quotient est un phénomène 
complètement général. De plus, sous de bonnes hypothèses, tout plongement 
entre algèbres de mesure est de cette forme (voir la première section
du chapitre 2 de \textcite{glasnerErgodicTheoryJoinings2003}).
\end{rema}

Notons $M_\gamma$ l'image de $f_\gamma^*$, qui est une sous-structure de $\MAlg(X,\mu)$
correspondant à la sous-algèbre des parties $f_\gamma$-mesurables. Il est crucial 
pour la suite de remarquer que pour tous $\gamma_1, \gamma_2\in\mathbb F_2$ distincts, 
les sous-structures $M_{\gamma_1}$
et $M_{\gamma_2}$ sont \textbf{indépendantes} au sens de la théorie des probabilités:
pour tout $A\in M_{\gamma_1}$ et tout $B\in M_{\gamma_2}$, on a $\mu(A\cap B)=\mu(A)\mu(B)$. Ainsi
si on fixe $A\in M_{\gamma_1}$, alors pour $B\in M_{\gamma_2}$ la valeur de $\mu(A\cap B)$
\emph{ne dépend que de $\mu(B)$}. En particulier elle ne dépend que du type de $B$.
Plus généralement, on a la formulation suivante, qui est un cas particulier
du fait que la \emph{relation d'indépendance stable} coïncide avec l'indépendance usuelle des
probabilités (\cite[Thm.~4.1]{benyaacovSchrodingerCat2006}, voir aussi \cite[Sec~16]{benyaacovModeltheorymetric2008}).

\begin{lemm}\label{lem:stabilite}
	Soient $M_1$ et $M_2$ deux sous-structures de $\MAlg(X,\mu)$ qui sont des modèles indépendants de $T$.
	Soit $\varphi(\bar x,\bar y)$ un prédicat définissable. Alors pour tout $\bar A\in M_1^\N$, 
	la fonction $\bar B\in M_2^\N\mapsto \varphi(\bar A,\bar B)$ coïncide avec l'interprétation
	dans $M_2$
	d'un prédicat définissable que l'on note $\mathsf{d}_{\bar A,\varphi}(\bar y)$.
\end{lemm}
\begin{proof}
	On a besoin de revisiter l'exemple~\ref{ex: approxi oligo pour malg}.
	Pour $A\in\MAlg(X,\mu)$ on note $A^0=A$ et $A^1=X\setminus A$.
	On note $\{0,1\}^{<\N}$ l'ensemble dénombrable des suites finies d'éléments de $\{0,1\}$. Un élément de
	$\{0,1\}^{<\N}$ est de la forme
	$\delta=(\delta(1),\dots,\delta(n))$ pour un unique $n\in\N$ appelé sa \emph{longueur}, noté $\abs\delta$.
	Pour une suite $(A_i)_{i\in\N}$
	d'éléments de $\MAlg(X,\mu)$, on note alors 
	$$\Phi([(A_i)_{i\in\N}])=
	(\mu(A_1^{\delta(1)}\cap\cdots\cap A_{\abs\delta}^{\delta(\abs{\delta})}))_{\delta\in\{0,1\}^{<\N}}
	\in [0,1]^{\{0,1\}^{<\N}}.$$
	Alors $\Phi:\Aut(X,\mu)\bbslash\MAlg(X,\mu)^\N\to [0,1]^{\{0,1\}^{<\N}}$ est continue, injective d'après la preuve de l'exemple~\ref{ex: approxi oligo pour malg}, 
	et $\Aut(X,\mu)\bbslash\MAlg(X,\mu)^\N$ est compact d'après la remarque~\ref{rmq: action approx cocompacte sur Xn}.
	L'image de $\Phi$ est donc un compact 
	de $[0,1]^{\{0,1\}^{<\N}}$ que l'on note $\mathfrak P$. \emph{L'application~$\Phi$ nous fournit
	alors une identification entre les prédicats définissables (vus comme fonctions continues
	$\Aut(X,\mu)$-invariantes sur $\MAlg(X,\mu)^\N$) et les fonctions continues sur $\mathfrak P$.}

	Soit $\tilde \varphi$ la fonction continue sur $\mathfrak P$ obtenue via cette identification,
	alors par définition pour tous $\bar A=(A_i)_{i\in\N}\in \MAlg(X,\mu)^\N$ et 
	$\bar B=(B_i)_{i\in\N}\in \MAlg(X,\mu)^\N$ on a 
	$$\varphi^{\MAlg(X,\mu)}(\bar A,\bar B)=\tilde\varphi(\Phi(A_1,B_1,A_2,B_2,\dots)).$$
	Pour une suite finie $\delta\in\{0,1\}^{<\N}$ notons $\delta_1$ la suite extraite en ne prenant
	que les termes impairs de $\delta$, et $\delta_2$ la suite extraite en ne prenant que les termes
	pairs de $\delta$.
	Alors si $\bar A\in M_1^\N$ et $\bar B\in M_2^\N$, on a par indépendance
	\[
	\Phi(A_1,B_1,A_2,B_2,\dots)_\delta=
	\mu(A_1^{\delta_1(1)}\cap A_2^{\delta_1(2)}\cdots\cap A_{\abs{\delta_1}}^{\delta_1(\abs{\delta_1})})
	\mu(B_1^{\delta_2(1)}\cap B_2^{\delta_2(2)}\cdots\cap B_{\abs{\delta_2}}^{\delta_2(\abs{\delta_2})})
	\]
	En particulier, si $\bar A\in M_1^\N$ est fixé, la fonction 
	$M_2^\N\to \R$ qui à $\bar B$ 
	associe le réel $\tilde\varphi(\Phi(A_1,B_1,A_2,B_2,\dots))$ est une fonction qui dépend continûment 
	du type de $\bar B$,
	et c'est donc l'interprétation dans $M_2$ d'un prédicat définissable d'après le théorème~\ref{thm:dualité}
	(en des termes plus dynamiques, c'est une fonction $\Aut(M_2)$-invariante, donc l'interprétation d'un
	prédicat définissable d'après le début de cette section).	
\end{proof}
	\begin{rema} 
		Notons $()$ la suite vide, de longueur $0$. Si $\delta=(\delta(1),\dots,\delta(n))$ et 
		$\epsilon\in\{0,1\}$, on note $\delta\smallfrown\epsilon$ la suite finie $(\delta(1),\dots,\delta(n),\epsilon)$.
		On peut vérifier
	que l'ensemble $\mathfrak P$ défini ci-dessus
	est le sous-ensemble fermé du compact $[0,1]^{\{0,1\}^{<\N}}$ formé des familles
	$(p_\delta)_{\delta\in\{0,1\}^{<\N}}\in [0,1]^{\{0,1\}^{<\N}}$
	telles que $t_{()}=1$ 
        et pour tout $\delta\in\{0,1\}^{<\N}$,
	$$p_{\delta}=p_{\delta\smallfrown0}+p_{\delta\smallfrown1}.$$
	Cet ensemble s'identifie naturellement
	à l'espace des mesures de probabilité sur les boréliens de $\{0,1\}^\N$ : on construit une telle mesure en assignant à chaque ouvert-fermé 
	$$N_\delta\coloneq\{(x_i)\in\{0,1\}^\N\colon (x_1,\dots,x_{\abs\delta})=\delta\}$$ la mesure $p_\delta$.
	Ceci concrétise la remarque~\ref{rmq:probas sur Cantor}.
\end{rema}

Rappelons qu'on a fixé une action de $\mathbb F_2=\la a,b\ra$ sur $(X,\mu)$ donnée par: 
pour tout $x\in X$, $(n,g)\in\N\times\mathbb F_2$ et $\gamma\in\mathbb F_2$,
\[
(\gamma\cdot x)(n,g)=x(n,\gamma\inv g).
\]
Cette dernière induit une action par automorphisme sur $\MAlg(X,\mu)$ en passant au quotient
son action sur les boréliens de $X$ donnée par\footnote{On peut
remarquer que l'action naturelle est en fait une action à droite, nous sollicitons donc l'indulgence des
puristes pour ce choix peu fonctoriel.} : pour tout $\gamma\in\mathbb F_2$, pour tout $A\subseteq X$
borélien,
$\gamma A=\{\gamma\cdot x\colon x\in A\}$.
On vérifie alors que pour tout $\gamma\in\mathbb F_2$, et tout $B\in \MAlg(Y,\nu)$,
$$\gamma (f_{g}^* B)=f_{\gamma g}^*B.$$
En particulier $\gamma M_g=M_{\gamma g}$ pour tous $\gamma,g\in\mathbb F_2$. On identifie $\mathbb F_2$
au sous-groupe de $\Aut(X,\mu)$ donné par cette action.
Finissons ces préliminaires avec un ersatz du corollaire~\ref{coro:closure}.

\begin{prop}\label{prop: adhérence Aut}
	L'adhérence de $\Aut(X,\mu)$ dans $\MAlg(X,\mu)^{\MAlg(X,\mu)}$ est l'ensemble 
	des plongements élémentaires $\MAlg(X,\mu)\to\MAlg(X,\mu)$ et s'identifie naturellement
	au complété de $\Aut(X,\mu)$ pour sa structure uniforme gauche.
\end{prop}
\begin{proof}
	Il est clair que l'adhérence de $\Aut(X,\mu)$ est incluse dans l'espace fermé des plongements élémentaires.
	Réciproquement, si $\rho:\MAlg(X,\mu)\to\MAlg(X,\mu)$ est un plongement élémentaire
    et si on se 
	donne $A_1,\dots,A_n\in\MAlg(X,\mu)$, alors
	$\tp(\rho(A_1),\dots,\rho(A_n))=\tp(A_1,\dots, A_n)$ et donc d'après le théorème~\ref{thm:malg elim}
	on a que $(\rho(A_1),\dots,\rho(A_n))$ est dans l'adhérence de la $\Aut(X,\mu)$-orbite
	de $(A_1,\dots,A_n)$, ce qui conclut la preuve.
	L'identification avec le complété $\widehat{\Aut(X,\mu)}_{\mathcal L}$ provient du fait que
	si on fixe $(A_i)_{i\in\N}$ dense dans $\MAlg(X,\mu)$, la distance sur l'espace des plongements
	élémentaires
	$d(\rho_1,\rho_2)=\sum_{i\in\N}2^{-i}d_\mu(\rho_1(A_i),\rho_2(A_i))$ est complète
	et se restreint à $\Aut(X,\mu)$ en une distance invariante à gauche qui induit sa topologie.
\end{proof}
\begin{rema}
	Rappelons que puisque $T$ élimine les quantificateurs, tous les plongements sont élémentaires,
	et on peut donc enlever l'adjectif ``élémentaire'' dans la proposition précédente.
\end{rema}

\begin{proof}[Démonstration du théorème~\ref{THM:MAINTER}]
	Soit $\pi: \Aut(X,\mu)\to \mathcal O(\mathcal H)$ une représentation orthogonale sans vecteurs invariants,
	on veut montrer que $\pi$ se restreint à $\mathbb F_2$ en une représentation 
	contenant sa représentation régulière.
	Comme toute représentation orthogonale est somme directe de sous-représentations cycliques, 
	il suffit de montrer le résultat lorsque $\pi$ est cyclique, c'est-à-dire lorsqu'il existe 
	$\xi\in\mathcal H$
	tel que $\pi(\Aut(X,\mu))\xi$ engendre un sous-espace dense de $\mathcal H$.
	Supposons donc $\pi$ cyclique et fixons un tel vecteur $\xi\in\mathcal H$ de norme $1$.
	
	Notons $\Iso(\mathcal H)$ le semi-groupe des isométries linéaires (non nécessairement surjectives)
	de $\mathcal H$ dans $\mathcal H$. On 
	identifie par 
	la proposition précédente le complété à gauche de $\Aut(X,\mu)$ au 
	semi-groupe $\mathcal E(M,M)$ des plongements élémentaires 
	$M\to M$ où $M=\MAlg(X,\mu)$.
	La représentation $\pi$ s'étend par uniforme continuité en une 
	application continue que l'on note toujours $\pi:\mathcal E(M,M)\to \Iso(\mathcal H)$.
	Pour $g\in\mathbb F_2$, notons $\mathcal E(M,M_g)$ le sous-semi-groupe des plongements élémentaires de $M$ 
	dont l'image est contenue dans~$M_g$. Avant de continuer la preuve, remarquons que $\mathcal E(M,M_g)$ 
	est non vide, et que pour tout $\gamma\in \mathbb F_2$, comme $\gamma M_g=M_{\gamma g}$ on a
	$\gamma \mathcal E(M,M_g)=\mathcal E(M,M_{\gamma g})$. On définit enfin le sous-espace de Hilbert
	$$\mathcal H_g=\overline{\Vect\left(\pi(\mathcal E(M,M_g))\xi\right)},$$
	et on a donc $\pi(\gamma)\mathcal H_g=\mathcal H_{\gamma g}$. Il suffit alors de montrer que pour 
	tous $g_1,g_2\in \mathbb F_2$ distincts, $\mathcal H_{g_1}$ est orthogonal à $\mathcal H_{g_2}$, 
	ce qui va reposer de manière cruciale sur le fait que $M_{g_1}$ et $M_{g_2}$ sont indépendants.
	Ce qui suit est une reformulation de la preuve de la seconde
	partie de la proposition 4.1 de \textcite{ibarluciaInfinitedimensionalPolishGroups2021} dans notre
	cadre restreint.\\
	
	Fixons donc $g_1\neq g_2\in \mathbb F_2$, supposons par l'absurde que $\mathcal H_{g_1}$ ne soit
	pas orthogonal à~$\mathcal H_{g_2}$. On dispose alors d'un élément $\rho_0\in \mathcal E(M,M_{g_1})$
	tel que la projection orthogonale~$\eta$ de $\pi(\rho_0)\xi$ sur $\mathcal H_{g_2}$ ne soit pas nulle.
	Alors comme $\pi$ n'a pas de vecteurs invariants et $\xi$ est cyclique, 
	on trouve	$\epsilon>0$, et $T_1,T_2\in \Aut(X,\mu)$  tels que 
	\begin{equation}\label{eq:contrad}
	\abs{\la \eta-\pi(T_1)\eta, \pi(T_2)\xi\ra}>\epsilon
	\end{equation}
	
	Il est maintenant temps d'utiliser l'identification entre $\mathcal E(M,N)$ et l'espace des réalisations
	d'un $\omega$-type dans $N$ donnée par la proposition~\ref{prop: plongements comme type}; 
	pour rendre cette dernière plus lisible et plus proche des notations de 
	la proposition, on note désormais par des minuscules les éléments 
	de nos algèbres de mesure. Fixons donc $\bar \alpha=(\alpha_i)_{i\in\N}$ dense dans~$M$, soit
	$p\coloneq\tp(\bar\alpha)$.
	Pour chaque modèle $N$ de $T$,
	l'ensemble des suites $\bar a\in N^\omega$ tels que $\tp(\bar a)=p$ s'identifie aux
	plongements élémentaires de $M$ dans $N$ via $\rho\mapsto \rho(\bar \alpha)$. Pour chaque $\bar a\in N^\N$
	tel que $\tp(\bar a)=p$, on note $\rho_{\bar a}$ le plongement élémentaire correspondant, de sorte que 
	$\rho_{\bar a}(\bar \alpha)=\bar a$. Remarquons que pour tout $\bar a\in M^\N$ tel que $\tp(\bar a)=p$,
	 tout $T\in\Aut(X,\mu)$, on a
	$\rho_{T\bar a}(\bar \alpha)=T\bar a=T\rho_{\bar a}(\bar \alpha)$ donc $\rho_{T\bar a}=T\rho_{\bar a}$
	par densité de $\bar \alpha$.
	
   	Pour un modèle $N$ de $T$, on note $p(N)$ l'ensemble des réalisations de $p$
	dans $N$. Considérons la fonction 
	$\tilde \varphi: (\bar a,\bar b)\in p(M)^2\mapsto \la \pi(\rho_{\bar a})\xi, \pi(\rho_{\bar b})\xi\ra$.
	C'est une fonction continue $\Aut(X,\mu)$-invariante, donc elle définit une fonction continue sur l'espace
	fermé
	$[p(M)\times p(M)]\subseteq \Aut(X,\mu)^\N\bbslash M^\N$. Par le lemme de Tietze--Urysohn, elle se prolonge 
	en une
	fonction continue sur le compact 
	$\Aut(X,\mu)^\N\bbslash M^\N$. On note $\varphi(\bar x,\bar y)$ le prédicat définissable
	correspondant, on a alors : pour tous $\bar a, \bar b\in p(M)$
	\begin{equation}\label{eq:cacestphi}	
	\varphi^M(\bar a,\bar b)=\la \pi(\rho_{\bar a})\xi, \pi(\rho_{\bar b})\xi\ra.
	\end{equation}
	On va maintenant approcher le vecteur $\eta$ par une combinaison linéaire afin d'aboutir à une 
	contradiction.
	On fixe des plongements $\rho_1,\dots,\rho_n\in \mathcal E(M,M_{g_2})$ et des réels
	$\lambda_1,\dots\lambda_n$ tels que 
	\begin{equation}\label{eq: close to eta}
	\norm{\eta-\sum_{i=1}^n \lambda_i\pi(\rho_i)\xi}\leq\frac\epsilon 4.
	\end{equation}

	Par le même raisonnement qui a mené à la définition de $\varphi$, on trouve un prédicat définissable 
	$\psi(\bar x_1,\dots,\bar x_n,\bar y)$ tel que pour tous $\bar a_1,\dots,\bar a_n,\bar b\in p(M)$, on a 
	$$
	\psi^{M}(\bar a_1,\dots,\bar a_n, \bar b)=
	\la \lambda_1\pi(\rho_{\bar a_1})\xi+\cdots+\lambda_n\pi(\rho_{\bar a_n})\xi, \pi(\rho_{\bar b})\xi\ra.
	$$
	Remarquons que comme $M_{g_2}$ est une sous-structure élémentaire de $M$, on a $p(M_{g_2})\subseteq p(M)$
	et un cas particulier de l'équation précédente est donc que pour tous $\bar a_1,\dots,\bar a_n,\bar b\in p(M_{g_2})$
	\begin{equation}\label{cacestpsi}
	\psi^{M_{g_2}}(\bar a_1,\dots,\bar a_n, \bar b)=
	\la \lambda_1\pi(\rho_{\bar a_1})\xi+\cdots+\lambda_n\pi(\rho_{\bar a_n})\xi, \pi(\rho_{\bar b})\xi\ra.
	\end{equation}
	Soit $\bar c=\rho_0(\bar \alpha)\in M_{g_1}^\N$, de sorte que $\rho_0=\rho_{\bar c}$. D'après le lemme~\ref{lem:stabilite}, on dispose d'un prédicat définissable $\mathsf d_{\bar c,\varphi}(\bar y)$ tel que
	$\varphi^M(\bar c,\bar b)= \mathsf d_{\bar c,\varphi}^{M_{g_2}}(\bar b)$ pour tous $\bar b\in M_{g_2}^\N$.
	
	Remarquons que pour tout $\bar b\in p(M_{g_2})$
	on a 
	$$\varphi^M(\bar c,\bar b)=
	\la \pi(\rho_{\bar c})\xi, \pi(\rho_{\bar b})\xi\ra=\la \eta, \pi(\rho_{\bar b})\xi\ra$$
	puisque $\eta$ est la projection orthogonale sur $\mathcal H_{g_2}$ de $\pi(\rho_{\bar c})\xi=\pi(\rho_0)\xi$.
	Ainsi, par l'inégalité de Cauchy--Schwarz et l'inéquation~\eqref{eq: close to eta} on a que 
	pour tout $\bar b\in M_{g_2}^\N$ tel que $\tp(\bar b)=p$,
	$$
	\absBig{\Bigl\langle \sum_{i=1}^n\lambda_i\pi(\rho_i)\xi,\pi(\rho_{\bar b})\xi\Bigr\rangle-
		\la \pi(\rho_{\bar c})\xi, \pi(\rho_{\bar b})\xi\ra}\leq \frac \epsilon 4.
	$$ 
	On reformule cette inégalité une dernière fois avant le tour de magie : notons 
	$\bar a_i=\rho_i(\bar\alpha)\in M_{g_2}^\N$, alors en utilisant~\eqref{eq:cacestphi} et~\eqref{cacestpsi},
	on a que pour tout $\bar b\in M_{g_2}^\N$ tel que $\tp(\bar b)=p$,
	$$
	\abs{
	\psi^{M_{g_2}}(\bar a_1,\dots,\bar a_n,\bar b)- \mathsf d_{\bar c,\varphi}^{M_{g_2}}(\bar b)
	}
	\leq \frac \epsilon 4.
	$$
	D'après le théorème~\ref{thm: real defin xmu} et le théorème~\ref{thm:principal est real partout},
	on dispose d'un prédicat définissable $d(\bar x,p)$ qui s'interprète dans tout modèle de $T$ 
	comme la distance à l'ensemble des réalisations de $p$, donc le prédicat définissable suivant est bien
	défini d'après la proposition~\ref{prop:quantif sur def}:
	$$
	\kappa(\bar x_1,\dots,\bar x_n)\coloneq\sup_{d(\bar y,p)=0} 	\abs{
		\psi(\bar x_1,\dots,\bar x_n,\bar y)- \mathsf d_{\bar c,\varphi}(\bar y)
	}.
	$$
	L'inégalité que nous venions d'obtenir se reformule alors en le fait
	que $$\kappa^{M_{g_2}}(\bar a_1,\dots,\bar a_n)\leq \frac \epsilon 4.$$
	Mais alors, comme l'inclusion de $M_{g_2}$ dans $M$ est élémentaire,
	on obtient que 
	$$
	\kappa^M(\bar a_1,\dots,\bar a_n)\leq \frac \epsilon 4,$$
	ce qui veut dire que pour tout $\bar d\in p(M)$,
	on a
		$$
	\abs{
		\psi^{M}(\bar a_1,\dots,\bar a_n,\bar d)- \mathsf d_{\bar c,\varphi}^{M}(\bar d)
	}
	\leq \frac \epsilon 4.
	$$
	Or $\mathsf d_{\bar c,\varphi}^{M}$ est un prédicat définissable, en particulier 
	il est $\Aut(X,\mu)$-invariant. D'autre part 	
	$\psi^{M}(\bar a_1,\dots,\bar a_n,\bar d)=	\la \sum_{i=1}^n\lambda_i\pi(\rho_i)\xi,\pi(\rho_{\bar d})\xi\ra$
	donc par inégalité triangulaire pour tout $T\in\Aut(X,\mu)$,
	\[\absBig{\laBig \sum_{i=1}^n\lambda_i\pi(\rho_i)\xi,\pi(\rho_{\bar d})\xi\raBig - 
		\laBig \sum_{i=1}^n\lambda_i\pi(\rho_i)\xi,\pi(\rho_{T\inv\bar d})\xi\raBig}\leq \frac \epsilon 2.\]
	Notons $\tilde \eta=\sum_{i=1}^n\lambda_i\pi(\rho_i)\xi$, qui d'après~\eqref{eq: close to eta} 
	est $\epsilon/4$-proche de $\eta$. L'inégalité précédente se réécrit
	\[\abs{\la \tilde \eta,\pi(\rho_{\bar d})\xi\ra - 
		\la \tilde \eta,\pi(\rho_{T\inv\bar d})\xi\ra}\leq \frac \epsilon 2.\]
	
Or $\rho_{T\inv\bar d}=T\inv\rho_{\bar d}$ et $\pi_{\restriction\Aut(X,\mu)}$ est orthogonale, donc on trouve que pour tout $T\in\Aut(X,\mu)$
et tout $\bar d\in p(M)$,
	\[\abs{\la \tilde \eta,\pi(\rho_{\bar d})\xi\ra - 
		\la \pi(T)\tilde \eta,\pi(\rho_{\bar d})\xi\ra}\leq \frac \epsilon 2.\]
	Autrement dit, pour tout  $\rho\in\mathcal E(M,M)$ on a 
	\[\abs{\la \tilde \eta-\pi(T)\tilde \eta,\pi(\rho)\xi\ra}\leq \frac \epsilon 2,\]
	En particulier, en reprenant les notations de l'équation~\eqref{eq:contrad}, pour $\rho=T_2$
	et $T=T_1$ on a
	\[\abs{\la \tilde \eta-\pi(T_1)\tilde \eta,\pi(T_2)\xi\ra}\leq \frac \epsilon 2.
	\]
	ce qui au vu de~\eqref{eq:contrad}, du fait que $\norm{\eta-\tilde\eta}\leq \frac \epsilon 4$, et de l'inégalité
	de Cauchy--Schwarz, est la contradiction attendue.
\end{proof}

Pour finir, donnons des éléments de preuve du théorème~\ref{thm:mainbis}. La difficulté 
est que les $(\mathcal H_{g})$ précédents ne nous permettent pas de recouvrir complètement $\mathcal H$,
et on a a priori aucune idée de ce à quoi ressemble la représentation $\pi_{\restriction \mathbb F_2}$
en dehors de $\bigoplus_{g\in\mathbb F_2}\mathcal H_g$.
Pour corriger ce problème, on considère pour chaque partie finie $F\Subset \mathbb F_2$ la sous-structure 
$M_F$ des parties $f_F$-mesurables, où $f_F:X\to \{0,1\}^{\N\times F}$
associe à $x\in X$ la suite $(x(n,f))_{n\in\N, f\in F}$.
Notons qu'avec nos notations précédentes $M_{g}$ devient $M_{\{g\}}$, et que pour tout $\gamma\in\mathbb F_2$
et tout $F\Subset\mathbb F_2$ on a $\gamma M_F=M_{\gamma F}$. 
En particulier, si on note  $$\mathcal H_F=\overline{\Vect\left(\pi(\mathcal E(M,M_F))\xi\right)},$$
alors $\mathcal H_{\gamma F}=\pi(\gamma)\mathcal H_F$.
La réunion des $M_F$ est dense
dans $M$, ce qui permet de montrer que 
$\Aut(X,\mu)\subseteq \overline{(\bigcup_{F\Subset\mathbb F_2}\mathcal E(M,M_F))}$, et donc que 
$\mathcal H=\overline{\Vect\left(\bigcup_{F\Subset \mathbb F_2} \mathcal H_F\right)}$.
La preuve précédente donne plus généralement que $\mathcal H_{F_1}\perp\mathcal H_{F_2}$
dès lors que $F_1$ et $F_2$ sont deux parties finies disjointes de $\mathbb F_2$. La clé
est qu'on a la version \emph{relative} plus générale suivante, où on note
$\mathcal H_1\perp_{\mathcal H_3}\mathcal H_2$ lorsque 
$\mathcal H_1\ominus\mathcal H_3\perp \mathcal H_2\ominus\mathcal H_3$ (le sous-espace $\mathcal H_i\ominus\mathcal H_j$
étant l'orthogonal de $\mathcal H_j\cap \mathcal H_i$ au sein de $\mathcal H_i$).
\begin{lemm}[{\cite[Lem.~4.1]{ibarluciaInfinitedimensionalPolishGroups2021}}]\label{lem:ortho full}
	Pour tous $F_1,F_2\Subset \mathbb F_2$, on a $\mathcal H_{F_1}\perp_{\mathcal H_{F_1\cap F_2}}\mathcal H_{F_2}$.
\end{lemm}

Une fois ce lemme établi, on pose pour tout $n\in\N$,
$\mathcal K_n=\overline{\Vect(\bigcup_{\abs F=n}\mathcal H_F)}$,
puis $\mathcal H_n=\mathcal K_n\ominus \mathcal K_{n-1}$, et on montre que la restriction de la représentation
de $\mathbb F_2$
à chaque~$\mathcal H_n$ est un multiple de la représentation régulière en observant que pour toute partie
finie $F\Subset \mathbb F_2$ de taille $n$ et tout $\gamma\in\mathbb F_2\setminus\{1\}$ on a 
$\pi(\gamma)\mathcal H_F\perp_{\mathcal K_{n-1}}\mathcal H_F$. Ceci finit d'établir le théorème~\ref{thm:mainbis}
modulo le lemme ci-dessus.

Ce dernier s'appuie une version relative du lemme~\ref{lem:stabilite},
où l'indépendance relative est celle de la théorie des probabilités. 
Pour l'énoncer on utilise la notion importante de \emph{$A$-définissabilité} dans un modèle $M$ de $T$ complète,
pour une partie $A\subseteq M$ que l'on introduit brièvement.
Étant
donnée $A\subseteq M$ on \emph{enrichit} le langage $\mathcal L$ de notre théorie~$T$ en ajoutant, 
pour chaque $a\in A$,
un symbole de constante $c_a$, interprété dans $M$ comme~$a$.
On note $\mathcal L_A$ le langage obtenu. On travaille alors avec la théorie de $M$ dans 
ce nouveau langage que l'on note $\Th_{\mathcal L_A}(M)$, et on a une notion de prédicat définissable 
dans cette théorie complète, que l'on
appelle prédicat $A$-définissable. Par exemple, si $\varphi(x_1,x_2)$ était un prédicat définissable et $a\in A$
est fixé, $\varphi(x_1,c_a)$ est un prédicat $A$-définissable, interprété dans $M$
comme la fonction $\varphi^M(\cdot,a)$. Une remarque important à faire est que si $N$ est une sous-structure
élémentaire (pour le langage $\mathcal L$) de $M$ qui contient $A$, alors on peut la voir naturellement comme
une sous $\mathcal L_A$-structure de $M$ qui reste élémentaire. Voici donc la version relative
du lemme~\ref{lem:stabilite} qui sert pour la preuve du lemme~\ref{lem:ortho full}.

\begin{lemm}\label{lem:stabilite2}
	Soient $M_1$, $M_2$ et $M_3$ deux sous-structures de $\MAlg(X,\mu)$ qui sont des modèles de $T$, 
	supposons que $M_1$ et $M_2$ sont relativement indépendants au dessus de $M_3$.
	Soit $\varphi(\bar x,\bar y)$ un prédicat définissable. Alors pour tout $\bar a\in M_1^\N$, 
	la fonction $\bar b\in M_2^\N\mapsto \varphi(\bar a,\bar b)$ coïncide avec l'interprétation
	dans $M_2$
	d'un prédicat $M_3$-définissable.
\end{lemm}

La preuve élémentaire que nous avons fait du lemme~\ref{lem:stabilite} peut s'adapter en utilisant l'espérance conditionnelle
et en s'appuyant sur l'observation
suivante, que l'on montre
par densité des fonctions en escalier : si $N$ est une sous-algèbre de $M=\MAlg(X,\mu)$, et $f:(X,\mu)\to \R$ est une fonction 
$N$-mesurable intégrable,
alors la fonction $A\in M\mapsto \int_A f(x)d\mu(x)$ coïncide avec l'interprétation un prédicat $N$-définissable.

\appendix\section{Le théorème de Ryll-Nardzewski}

\subsection{Va-et-vient}

Cette section forme une partie importante de la preuve du théorème~\ref{thm: Ryll-Nardzewski} de Ryll-Nardzewski que l'on énoncera
dans la dernière section, 
et est une bonne application des outils développés dans la section~\ref{sec: definissable}.

\begin{theo}\label{theo:vaetvient}
	Soit $T$ une théorie complète, supposons que pour tout $m<\omega$, tout $m$-type est isolé. Soient 
	$M$ et $N$ deux modèles \emph{séparables} de $T$. Soit $n<\omega$, soient 
	$\bar a\in M^n$ et $\bar b\in N^n$
	tels que $\tp(\bar a)=\tp(\bar b)$, et soit $\epsilon>0$.
	Alors il existe un isomorphisme $\varphi: M\to N$ tel que
	$d^N(\varphi(\bar a),\bar b)<\epsilon$.
\end{theo}
\begin{rema}
	Le théorème reste vrai pour $n=\omega$ puisque la distance sur un produit infini n'est pas très 
	sensible à ce qui se passe dans les grandes coordonnées.\end{rema}
\begin{rema}
	Pour $N=M$, la conclusion du théorème se reformule en disant que l'application naturelle
	$\Aut(M)\bbslash M^n\to S_n(T)$ est injective. Dans ce cas, on dit aussi que la structure $M$ 
	est \emph{approximativement homogène}.
\end{rema}
\begin{proof}[Démonstration du théorème~\ref{theo:vaetvient}]
On va faire un argument de va-et-vient
utilisant le point crucial suivant, qui par symétrie reste valable
en échangeant les rôles de $M$ et $N$.

\begin{lemm}\label{lem:va-et-vient}
	Soient $\bar a\in M^n$, $\bar b\in N^n$ tels que $\tp(\bar a)=\tp(\bar b)$. 
	Soit $c\in M$ et $\epsilon>0$.
	Alors il existe $c'\in N$ et $\bar b'\in N^n$ tel que $\tp(\bar a,c)=\tp(\bar b',c')$
	et $d^N(\bar b,\bar b')<\epsilon$.	
\end{lemm}	

\begin{proof}[Preuve du lemme]
	Soit $p$ le type  de $(\bar a,c)$.
	D'après le théorème~\ref{thm:principal est real partout}, on dispose d'un 
	prédicat définissable $d(\cdot, p)$ 
	dont l'interprétation dans tout modèle de $T$ est la distance à l'ensemble
	des réalisations de $p$.
	Le type de  $\bar a$ contient alors le prédicat définissable suivant:
	\[
	\inf_{y} d((\bar x,y),p).
	\]	
	Soit $D$ l'ensemble des réalisations de $p$ dans $N$.
	Comme $\tp\bar a=\tp\bar b$, on trouve $c_1\in N$ tel que $d^N((\bar b,c_1),D)<\epsilon$.
	Donc il existe $(\bar b',c')\in N^n\times N$ réalisant $\tp(\bar a,c)$ avec 
	$d^N((\bar b,c),(\bar b',c'))<\epsilon$
	en particulier $d^N(\bar b,\bar b')<\epsilon$.	
	\renewcommand*{\qedsymbol}{\(\square_\text{lemme}\)}\end{proof}

Soient maintenant $\bar a,\bar b\in M^n$ avec $\tp(\bar a)=\tp(\bar b)$ et $\epsilon>0$.
On voudrait définir un isomorphisme qui étende l'application $\bar a\mapsto \bar b$,
mais ce n'est pas possible en général, et il va donc falloir perturber les 
choses en suivant le lemme précédent.
Soit 
$\epsilon>0$, fixons une suite $(\epsilon_k)_{k\in\N}$ de réels positifs
telle que $\sum_k \epsilon_k<\epsilon$.

On va construire par récurrence sur $k\in\N$ deux suites $(\bar u_k)$ et 
$(\bar v_k)$ telles que pour tout $k\in\N$, les uplets $\bar u_k$ et $\bar v_k$
sont des $(n+2k)$-uplets de même type, $\bar u_0=\bar a$, $\bar v_0=\bar b$, et si on note 
$\pi_{n+2k}$ la projection sur les $n+2k$ premières coordonnées 
alors
\begin{align*}
d^M(\bar u_k,\pi_{n+2k}(\bar u_{k+1}))&<\epsilon_k,\\
d(\bar v_k,\pi_{n+2k}(\bar v_{k+1}))&<\epsilon_k.
\end{align*}
Pour que cette construction nous permette de construire un isomorphisme entre~$M$ et~$N$,
on a besoin de se donner également une partie $C^M\subseteq M$ 
dénombrable dense et une énumération $(c^M_k)_{k\in\N}$ de $C$ telle que chaque élément 
de~$C^M$ apparaisse une infinité de fois dans l'énumération.
On se donne de même une énumération $(c^N_k)_{k\in\N}$ d'une partie dénombrable dense $C^N$ de $N$
où chaque élément apparait une infinité de fois.
On va faire en sorte
que pour tout $k$, $c^M_k$ apparaisse dans l'uplet $\bar u_k$, et $c^N_k$ apparaisse dans l'uplet $\bar v_k$. 

La construction par récurrence se fait de la manière suivante: on doit commencer
par $\bar u_0=\bar a$ et $\bar v_0=\bar b$. Puis, $\bar u_k$ et $\bar v_k$ étant construits,
\begin{itemize}
	\item On considère le $(n+2k+1)$-uplet $(\bar v_k, c^N_{k+1})$, et on applique alors le lemme~\ref{lem:va-et-vient} pour trouver $\bar w_{k+1}$ tel que 
	$\tp(\bar w_{k+1})=\tp(\bar v_k,c_{k+1})$ et 
	$d(\pi_{n+2k}(\bar w_{k+1}),\bar x_k)<\epsilon_k$ (étape du \emph{va}).
	\item On pose ensuite $\bar u_{k+1}=(\bar w_{k+1},c^M_{k+1})$, et le lemme~\ref{lem:va-et-vient} nous fournit $\bar v_{k+1}$ de même type que $\bar u_{k+1}$
	avec $d(\pi_{n+2k+1}(\bar y_{k+1}), (\bar y_k, c_{k+1}))<\epsilon_k$ donc en 
	particulier $d(\pi_{n+2k}(\bar y_{k+1}),\bar y_k)<\epsilon_k)$ (étape du \emph{vient}).
\end{itemize}
Puisque $\sum_k \epsilon_k$ est finie, pour tout $n\geq 1$ la $n$-ième 
composante de $\bar u_k$ (resp.\ $\bar v_k$)
converge, et on note $u_n$ (resp.\ $v_n$) sa limite. On pose $\bar u=(u_1,\ldots)$ et
$\bar v=(v_1,\ldots)$, alors par construction $\tp(\bar u)=\tp(\bar v)$. De plus 
$\bar u$ est d'image dense dans $M$  puisque pour chaque élément $c\in C^M$ il
y a une infinité de $k\geq 1$ tels que  $c$ apparaisse dans $\bar u_k$, et de même $\bar v$
est d'image dense dans $N$.

Comme $\bar u$ et $\bar v$ ont le même type \footnote{En fait ici seulement
le fait qu'ils aient le même type sans quantificateur compte.}
 et sont d'image dense, l'application
$\alpha: x_i\mapsto y_i$ est une isométrie qui doit préserver les interprétations
des fonctions et des relations. Elle s'étend donc par complétude en l'automorphisme de $M$
recherché, qui envoie $\bar a$ à distance au plus $\epsilon$ de $\bar b$ puisque 
$\sum_k\epsilon_k <\epsilon$.	
\end{proof}

En appliquant le théorème pour des uplets $\bar a$ et $\bar b$ vides, notons qu'on a le corollaire suivant,
qui fait partie du théorème de Ryll-Nardzewski.

\begin{coro}\label{cor: aleph categ}
	Soit $T$ une théorie complète dont tous les $n$-types sont isolés pour $n<\omega$. Alors 
	tous les modèles séparables de $T$ sont isomorphes. \qed
\end{coro}

Lorsque tous les modèles séparables d'une théorie sont isomorphes, on dit qu'elle est 
\textbf{$\aleph_0$-catégorique}
ou \textbf{séparablement catégorique}. 
En utilisant le théorème de compacité, on montre la réciproque du corollaire précédent: si on 
a une théorie dans un langage dénombrable qui est $\aleph_0$-catégorique, alors tous ses $n$-types 
sont isolés. Mentionnons également le théorème d'omission des types, qui donne un cadre plus général au
résultat ci-dessus et à ce qui suit (cf.\ \cite[Sec.~12]{benyaacovModeltheorymetric2008}).

\subsection{Distance sur l'espace des types}

Afin de prouver le théorème de Ryll-Nardzewski, on a besoin d'une distance sur l'espace des types 
qui raffine sa topologie et en fait un espace \emph{topométrique} \parencite{benyaacovTopometricSpacesPerturbations2008}. 
La définition de cette distance évoquera des souvenirs 
à la lectrice familière avec la distance de Wasserstein et le théorème de 
Kantorovich–Rubinstein 
(cf.\ Lecture 20 dans \cite{dudleyProbabilitiesMetricsConvergence1976}),
ou encore avec les algèbres de Banach
de fonctions lipschitziennes  (voir en particulier l'équation (2.1) de 
\cite{sherbertBanachAlgebrasLipschitz1963}).
Ce n'est pas la définition originelle de \textcite{benyaacovContinuousFirstOrder2010}, mais 
on peut montrer qu'elle lui est équivalente (cf.\ \cite[Thm. 3.3.14]{hallbackMetricModelTheory2020}). Commençons 
par isoler un sous-ensemble important de l'algèbre des prédicats définissables.

\begin{defi}
	Soit $T$ une théorie complète et $K\geq 0$.
	Un prédicat définissable $\varphi\in\mathfrak P^{\mathcal L,T}$ est dit $K$-lipschitzien
	s'il admet une interprétation $K$-lipschitzienne dans un modèle de $T$.
\end{defi}

D'après la remarque~\ref{rmk:unif continuite}, toute théorie complète \emph{sait} quels prédicats définissables
sont $K$-lipschitziens, puisque ceux sont ceux qui admettent la fonction
$K\id_{\R^+}$ comme module d'uniforme continuité. En particulier si $\varphi$ est $K$-lipschitzien, alors
$\varphi^M$ est $K$-lipschitzien dans \emph{tout} modèle $M$ de $T$.

\begin{defi}
	Soit $T$ une théorie complète, soit $n\leq \omega$. Pour deux types $p,q\in S_n(T)$, leur
	distance est donnée par 
	$$
	\dtp(p,q)=\sup \left\{ \abs{\varphi(p)-\varphi(q)}\colon 
	\varphi\in\mathfrak P^{\mathcal L,T}_n \text{ est }1\text{-lipschitzien}\right\}.
	$$
\end{defi}

Le lemme suivant va nous permettre de voir que $\dtp$ est bien une distance.

\begin{lemm}\label{lem: lipschitz est dense}
	Pour tout $n\leq \omega$, la sous-algèbre des prédicats définissables lipschitziens est dense 
	dans $\mathfrak P^{\mathcal L,T}_n$.
\end{lemm}
\begin{proof}
	Soit $\varphi(\bar x)\in \mathfrak P^{\mathcal L,T}_n$ un prédicat définissable,
	et fixons un modèle $M$ de $T$.
	Pour chaque  $K\geq 0$, considérons le prédicat définissable 
	$$\varphi_K(\bar x)=\inf_{\bar y} \left(\varphi(\bar y)+Kd(\bar x,\bar y)\right).$$
	Montrons tout d'abord que $\varphi_K$ est bien $K$-lipschitzien.
	Soit $\bar b\in M^n$, la fonction
	$\varphi^M_{K,\bar b}:\bar x\in M^n\mapsto \varphi^M(\bar b)+Kd(\bar x,\bar b)$ est $K$-lipschitienne.
	L'interprétation de $\varphi_K$ dans~$M$ étant l'infimum des fonctions $K$-lipschitziennes
	$\varphi^M_{K,\bar b}$, elle doit donc
	elle même être $K$-lipschitzienne. 
	
	Reste à montrer que $\norm{\varphi-\varphi_K}_T$ tend vers $0$ quand $K$ tend vers $+\infty$. 
	Remarquons que par définition $\varphi_K\leq \varphi$.
	Soit maintenant $\epsilon>0$.
	Comme $\varphi$ est uniformément continue on dispose de $\delta>0$ tel que 
	$d^M(\bar a,\bar b)<\delta$ implique $\abs{\varphi^M(\bar a)-\varphi^M(\bar b)}<\epsilon$. 
	Soit $K>0$ tel que $\frac{2\norm{\varphi}_T}{K}<\min(\delta,\epsilon)$.

    Prenons $\bar a\in M^n$. Pour $\bar b\in M^n$, on a deux cas à considérer.
    \begin{itemize}
    	\item Si  $d^M(\bar a,\bar b)<\delta$, alors $\varphi^M(\bar b)$ est $\epsilon$-proche de 
    	$\varphi^M(\bar a)$, et $Kd^M(\bar a,\bar b)<\epsilon$, donc $\varphi^M(\bar b)+Kd^M(\bar a,\bar b)$
    	est $2\epsilon$-proche de $\varphi^M(\bar a)$.
    	\item Sinon $d^M(\bar a,\bar b)\geq \delta$,
    	alors on a $Kd^M(\bar a,\bar b)\geq 2\norm{\varphi}_T$ et donc 
    	$\varphi^M(\bar b)+Kd^M(\bar a,\bar b)\geq
    	\varphi^M(\bar a)$. 
    \end{itemize}
	On conclut donc en prenant l'infimum sur tous les $\bar b$ dans $M^n$ que 
	$\absbig{\varphi^M(\bar a)-\varphi^M_K(\bar a)}\leq 2\epsilon$.
   
	Ceci étant valable pour tout $\bar a\in M^n$, on a montré que 
	$\normbig{\varphi^M-\varphi^M_K}_\infty\leq 2\epsilon$ dès lors que $K$ vérifie 
	$\frac{2\norm{\varphi}_T}{K}<\min(\delta,\epsilon)$, ce qui termine la preuve.
\end{proof}

\begin{prop}\label{prop: dtp}
	La fonction $\dtp$ est bien une distance sur $S_n(T)$, elle raffine la topologie logique,
	et pour tous $\bar a, \bar b\in M^n$, on a $\dtp(\tp(\bar a),\tp(\bar b))\leq d^M(\bar a,\bar b)$.
\end{prop}
\begin{proof}
	Le fait que $\dtp$ satisfasse l'inégalité triangulaire et soit symétrique est clair,
	seul la séparation pourrait poser problème. On va la montrer directement en montrant
	le deuxième point de la proposition, à savoir que $\dtp$
	raffine la topologie compacte (donc séparée) de l'espace des types.
	
	Soit donc $p$ un $n$-type, soit $U$ un voisinage de $p$.
	On dispose alors d'un prédicat définissable $\varphi$
	et d'un $\epsilon>0$ tels que $U\supseteq\{q\in S_n(T): \abs{\varphi(q)-\varphi(p)}\leq \epsilon\}$.
	
	D'après le lemme précédent, on dispose d'un prédicat définissable $\psi$ qui est $K$-lipschitzien
	pour un certain $K>0$
	et tel que $\norm{\varphi-\psi}_T<\frac\epsilon 3$. 
	Alors pour tout $q$ tel que $\dtp(p,q)<\frac{\epsilon}{3K}$,
	on aura $\abs{\frac{\psi(p)}K-\frac{\psi(q)}K}\leq \frac{\epsilon}{3K}$.
	En utilisant l'inégalité triangulaire ainsi que l'inégalité $\norm{\varphi-\psi}_T<\frac\epsilon 3$, 
	on conclut que la $\dtp$
	boule de rayon $\frac{\epsilon}{3K}$ autour de $p$ est bien contenue dans $U$, 
	ce qui termine la preuve que $\dtp$ raffine la topologie logique.
	
	Il nous reste alors à montrer que $\dtp(\tp(\bar a),\tp(\bar b))\leq d^M(\bar a,\bar b)$, mais 
	c'est une conséquence immédiate du fait que dans la définition de $\dtp$, on ne considère que 
	des prédicats $1$-lipschitziens.
\end{proof}

\subsection{Preuve du théorème de Ryll-Nardzewski}

On commence par une proposition importante qui fait le lien avec la section~\ref{sec: prp}
et généralise ce qu'on avait fait pour $\Aut([0,1],\lambda)$ (voir le théorème~\ref{thm: real defin xmu}).

\begin{prop}\label{prop:tp realise si compacité}
	Soit $T$ une théorie complète, soit $M$ un modèle de $T$ et soit $n<\omega$. 
	Si l'action diagonale de $\Aut(M)$ sur $M^n$ est 
	approximativement cocompacte, alors  tout $n$-type $p\in S_n(T)$
	est isolé.
\end{prop}
\begin{proof}
	Considérons l'application $\tp:M^n\to S_n(T)$ qui à $\bar a\in M^n$ associe $\tp(\bar a)$.
	C'est une contraction si on munit $S_n(T)$ de la distance $\dtp$,
	elle passe donc au quotient en une contraction
	$\tilde \tp:\Aut(M)\bbslash M^n\to S_n(T)$.
	Par compacité et densité de l'image de $\tp$ dans $S_n(T)$, l'application $\tilde\tp$ est surjective,
	donc $\tp$ elle-même est surjective: tous les $n$-types sont réalisés dans $M$.

	L'application $\tilde \tp:(\Aut(M)\bbslash M^n,d^M_{\Aut(M)})\to (S_n(T),\dtp)$ est  une contraction, 
	en particulier elle est continue. Donc 
	$(S_n(T),\dtp)$ est compact, et comme $\dtp$ raffine la topologie de $S_n(T)$
	on conclut par compacité que $\dtp$ induit la topologie de $S_n(T)$.
	
	Soit maintenant un $n$-type $p$, soit $D$ l'ensemble de ses réalisations dans $M^n$,
	on doit montrer que $d(\cdot, D)$ est un prédicat définissable.
	La fonction $\dtp(\cdot,p)$ est continue sur
	l'espace des types: par le théorème~\ref{thm:dualité}, c'est un prédicat définissable,
	et son ensemble de zéros dans $M$ est clairement l'ensemble $D$ des réalisations de $p$.
	
	Il suffit maintenant de montrer que $\dtp(\cdot,p)$ satisfait les hypothèses du théorème~\ref{thm:definissabilité}, c'est-à-dire que pour tout $\epsilon>0$ il existe
	$\delta>0$ tel que  pour tout $\bar b\in M^n$, si $\dtp(\tp(\bar b),p)<\delta$ 
	alors $d(\bar b,D)<\epsilon$.
	Supposons par l'absurde que ce ne soit pas le cas: on dispose alors d'une suite $(\bar b_k)_{k\in\N}$
	telle que $\dtp(\tp(\bar b_k),p)\to 0$ mais $d(\bar b_k,D)\geq \epsilon$ pour tout $k$.
	Par compacité de $\Aut(M)\bbslash M^n$, la suite 
	$([\bar b_k])_{k\in\N}$ admet une sous-suite convergente. Soit~$[\bar b]$ sa limite, alors par continuité de $\dtp$ on a $\tp(\bar b)=p$
	et d'autre part $d(\bar b,D)\geq \epsilon$, ce qui est contradictoire.
\end{proof}
	
Nous pouvons maintenant prouver le théorème de Ryll-Nardzewski sous sa forme pertinente pour
l'étude des groupes polonais Roelcke-précompacts. 

\begin{theo}\label{thm: Ryll-Nardzewski}
	Soit $G$ le groupe d'automorphismes d'une structure métrique $(M,d)$ séparable. Si l'action
	de $G$ sur $M$ est approximativement oligomorphe, alors pour tout $n\leq\omega$ l'application 
	 de $G\bbslash M^n$ dans $S_n(T)$ est 
	une bijection, et $T$ est $\aleph_0$-catégorique.
\end{theo}
\begin{proof}
	Soit $n< \omega$. D'après la proposition précédente, tout $n$-type est isolé, en particulier
	tout $n$-type est réalisé dans $M$: $G\bbslash M^n\to S_n(T)$ obtenue par passage au quotient
	est surjective. Son injectivité est une conséquence immédiate du théorème~\ref{theo:vaetvient} pour $M=N$.
	Le fait que $T$ soit $\aleph_0$-catégorique
	découle du corollaire~\ref{cor: aleph categ} puisque tout $n$-type est isolé.	
	Le cas $n=\omega$ découle du cas précédent. 
\end{proof}
\begin{rema}
	Une application directe du théorème~\ref{thm:dualité} nous donne sous les hypothèses
	du théorème ci-dessus une interprétation
	très concrète des prédicats définissables vus comme des fonctions sur $M^n$:
	ce sont exactement les fonctions continues sur $M^n$ qui sont $G$-invariantes.
\end{rema}

\begin{rema}Le théorème de Ryll-Nardzewski sous sa forme générale dit en plus que
réciproquement, si la théorie d'une structure est $\aleph_0$-catégorique
\emph{et le langage est dénombrable} alors
son groupe d'automorphismes agit de manière approximativement oligomorphe 
sur la structure. Pour le voir, il faut faire appel au théorème d'omission
des types, cf.\ la section~12 de \textcite{benyaacovModeltheorymetric2008}. Ce théorème tire
son nom de son homologue en théorie des modèles classique, dû à Ryll-Nardzewski
mais aussi indépendamment à Engeler et  Svenonius en 1959. La version métrique ci-dessus est
due à Ben Yaacov, Berenstein, Henson et Usvyatsov.
\end{rema}

\begin{coro}\label{coro:closure}
	Soit $G$ le groupe d'automorphismes d'une structure métrique séparable $(M,d)$, supposons que l'action
	de $G$ sur $M$ soit approximativement oligomorphe. Alors l'adhérence de~$G$
	dans~$M^M$ est égale à l'ensemble des plongements élémentaires de~$M$ dans~$M$, et 
	s'identifie naturellement au complété $\widehat G_\mathcal L$ de~$G$ pour sa structure uniforme
	gauche.
\end{coro}
\begin{proof}
	L'ensemble des plongements élémentaires de $M$ dans $M$ étant l'ensemble des 
	plongements qui commutent aux interprétations des prédicats définissables,
	c'est un fermé de $M^M$. Puisque tout automorphisme est un plongement 
	élémentaire, l'adhérence de $G$ est 
	incluse dans ce fermé.
	
	Ensuite, montrons que $G$ est dense. Soit $\rho: M\to M$ un plongement élémentaire. 
	Soient $a_1,\dots,a_n\in M^n$ et $\epsilon>0$, il s'agit de trouver $g\in G$
	tel que $d(ga_i,\rho(a_i))<\epsilon$ pour tout $i=1,\dots,n$. Mais comme $\rho$
	est élémentaire, $\tp(a_1,\dots,a_n)=\tp(\rho(a_1),\dots,\rho(a_n))$, et donc par injectivité
	de l'application $G\bbslash M^n\to S_n(T)$ on conclut que $(a_1,\dots,a_n)$ et $(\rho(a_1),\dots,\rho(a_n))$
	sont dans la même adhérence de $G$-orbite pour l'action diagonale de~$G$ sur~$M^n$, 
	ce qui donne immédiatement
	le résultat voulu.

	Pour montrer que l'adhérence de $G$ dans $M^M$ s'identifie avec $\widehat G_\mathcal L$,
	il reste à montrer que la structure uniforme induite sur $G$ coïncide 
	avec sa
	sa structure uniforme à  gauche puisque $M^M$ est un espace uniforme complet. 
	Or cette dernière est clairement invariante par $G$-multiplication à gauche
	et compatible avec la topologie de $G$, ce qui termine la preuve\footnote{Plus concrètemenent, on 
	peut comme dans la preuve de la proposition~\ref{prop: adhérence Aut} fixer $(a_n)_{n\in\N}$ dénombrable dense dans $M$, considérer 
	la distance $\sum_{n\in\N}\frac 1{2^n}d^M(\rho(a_n),\rho'(a_n))$
	sur l'espace fermé des applications $1$-lipschitziennes $M\to M$ 
	et constater qu'elle définit une distance compatible complète qui se restreint en une distance
	invariante à gauche sur $G$.}.
\end{proof}

Mis bout à bout avec la proposition~\ref{prop: plongements comme type} qui identifiait les plongements élémentaires à un ensemble d'éléments
de~$M^\N$ partageant le même type (et formant donc un sous-ensemble définissable de~$M^\N$ 
d'après la Proposition~\ref{prop:tp realise si compacité}), le corollaire précédent nous fournit donc 
une identification naturelle entre $\widehat G_{\mathcal L}$ et un sous-ensemble définissable de~$M^\N$.
Cette identification a été utilisée de manière cruciale pour $\Aut([0,1],\lambda)$ au début de la preuve du théorème~\ref{THM:MAINTER}
et reste fondamentale dans le cas général.
Elle est également le premier pas d'un théorème de Ben Yaacov et Kaïchouh\footnote{Plus précisément, le fait 
	que nous venons d'établir correspond
	à la proposition 12 de \textcite{benyaacovReconstructionSeparablyCategorical2016}.} qui dit que deux groupes
d'automorphismes de structures métriques $\aleph_0$-catégoriques sont topologiquement isomorphes 
si et seulement si les structures sous-jacentes sont \emph{bi-interprétables}, généralisant au cas
métrique un résultat 
d'\textcite{ahlbrandtQuasiFinitelyAxiomatizable1986}.

\printbibliography

\end{document}
